\documentclass[11pt]{amsart}
\usepackage{amscd}
\usepackage[arrow,matrix]{xy}
\usepackage{graphicx, tikz}
\usepackage{hyperref}
\usepackage{comment}
\usepackage{amsmath}
\usepackage{amsmath, latexsym, amssymb}
\numberwithin{equation}{section}
\theoremstyle{plain}

\newtheorem{lemma}{Lemma}[section]
\newtheorem{proposition}[lemma]{Proposition}
\newtheorem{theorem}[lemma]{Theorem}
\newtheorem{corollary}[lemma]{Corollary}

\theoremstyle{definition}
\newtheorem{definition}[lemma]{Definition}
\newtheorem*{definition*}{Definition}
\newtheorem{remark}[lemma]{Remark}
\newtheorem{example}[lemma]{Example}
\newtheorem{examples}[lemma]{Examples}

\usepackage{color}
\usepackage{cancel}

\definecolor{brown}{RGB}{150,100,0}

\definecolor{purple}{RGB}{150,0,100}
\definecolor{grey}{RGB}{128,128,128}
\newcommand{\grey}[1]{\textcolor{grey}{#1}}

\newcommand{\R}{{\mathbb R}}

\newcommand{\E}{{\mathbb E}}

\newcommand{\Z}{{\mathbb Z}}
\newcommand{\N}{{\mathbb N}}
\renewcommand{\P}{{\mathbb P}}
\newcommand{\Q}{{\mathbb Q}}

\newcommand{\Aa}{{\mathcal A}}
\newcommand{\Bb}{{\mathcal B}}
\newcommand{\Cc}{{\mathcal C}} 

\newcommand{\Ee}{{\mathcal E}}
\newcommand{\Ff}{{\mathcal F}}
\newcommand{\Hh}{{\mathcal H}}

\newcommand{\Ll}{{\mathcal L}} 
\newcommand{\Mm}{{\mathcal M}} 
\newcommand{\Nn}{{\mathcal N}}

\newcommand{\Pp}{{\mathcal P}}

\newcommand{\Ss}{{\mathcal S}}

\newcommand{\Vv}{{\mathcal V}}

\newcommand{\Xx}{{\mathcal X}}
\newcommand{\Yy}{{\mathcal Y}}
\newcommand{\Zz}{{\mathcal Z}}

\mathchardef\mhyp="2D

\newcommand{\om}{{\omega}}
\newcommand{\eps}{{\varepsilon}}

\newcommand{\Meas}{{\rm Meas}}

\renewcommand{\a}{{\mathfrak a}}

\def\NABLA#1{{\mathop{\nabla\kern-.5ex\lower1ex\hbox{$#1$}}}}
\def\Nabla#1{\nabla\kern-.5ex{}_#1}

\newcommand{\LRA}{\Longrightarrow}

\newcommand{\LLR}{\Longleftrightarrow}

\newcommand{\la}{\langle}
\newcommand{\ra}{\rangle}

\renewcommand{\:}{\colon}

\mathchardef\mhyp="2D

\DeclareMathOperator{\Probm}{Probm}

\DeclareMathOperator{\Id}{Id}

\begin{document}
\title[Supervised learning with probabilistic morphisms]{Supervised learning with probabilistic morphisms and kernel mean embeddings}
\author[H. V. L\^e]{H\^ong V\^an L\^e}	
\address{Institute of Mathematics of the Czech Academy of Sciences, Zitna 25, 11567 Praha 1, Czech Republic}

\email{hvle@math.cas.cz}

\thanks{Research of HVL was supported by the Institute of Mathematics,  Czech Academy of Sciences (RVO: 67985840) and GA\v CR-project 22-00091S}
	\date{\today}

\begin{abstract}    In this  paper  I  propose a   generative model of supervised learning  that unifies  two approaches  to supervised learning,  using    a concept  of a correct loss function. 
	Addressing  two 
	measurability problems,  which  have been ignored  in  statistical  learning  theory,  I propose  to use  convergence  in outer  probability  to characterize  the consistency  of a learning  algorithm. Building upon these results,  I extend a result due to Cucker-Smale, which addresses the learnability of a regression model, to the setting of a conditional probability estimation problem.  Additionally, I present a variant of Vapnik-Stefanuyk's regularization method for solving stochastic ill-posed problems, and using  it to prove the  generalizability of   overparameterized supervised learning models. 
\end{abstract}

\maketitle
\tableofcontents

\section{Introduction}\label{sec:intro}
\subsection{The concept of  a correct loss function in supervised   learning theory}
Let $\Xx$ and  $\Yy$ be   measurable spaces.
In   supervised learning,   we are given a data set of labeled items:     
$$S_n = \{(x_1, y_1), \ldots, (x_n, y_n)\}\in  ((\Xx \times \Yy)^n, \mu^n),$$  where $\mu$ is   an (unknown)  probability measure that governs the distribution  of i.i.d. labeled items $(x_n, y_n)\in \Xx \times \Yy$. The  goal  of a learner   in this scenario
is to find  the best approximation  $f_{S_n}$   of  the stochastic  relation between   the    input  $x \in \Xx$  and  its label $y \in \Yy$, formalized as  the conditional  probability  measure $[\mu_{\Yy|\Xx}]$ for a probability measure  $\mu$ on $(\Xx \times \Yy)$   with respect to the  projection  $\Pi_{\Xx}: \Xx \times \Yy \to \Xx$. 
 We refer to $[\mu_{\Yy|\Xx}]$  as the supervisor operator \cite[p. 35-36, p. 48]{Vapnik1998}.
 It is important to note that if $\Xx$ consists  of a  single  point,  this  problem is equivalent  to   estimating and approximating probability measures   on $\Yy$.

The concept of   a best  approximation  requires  a  specification of a hypothesis space $\Hh$ of possible  predictors as  well   a  the notion  of a {\it correct} loss function  that measures  the  deviation  of a  possible predictor from the  supervisor operator.  The loss function concept   in statistical analysis was  introduced by Wald in his  statistical  decision  theory \cite{Wald50}, which   can be traced back  to  Laplace's theory  of errors, further developed   by Gauss  and many  others \cite{Sheynin77}.  The importance of a correct  choice  of   a loss function has been discussed  by Chentsov. He noted that well-known examples of statistical point estimation problems, when properly formulated, admit asymptotically optimal decision rules, and the simplicity of the results  proved there is largely due to the correct choice of a loss function \cite[p. 391]{Chentsov72}.

In this  article,  we focus on characterizing the deviation of a predictor $h$ in a hypothesis class $\Hh$ from the  supervisor operator $[\mu_{\Yy|\Xx}]$,  using  the notion  of a correct  loss  function in the underlying  generative  model  of supervised learning.
Our   approach  utilizes the concept  of  a probabilistic morphism, which is  a  categorical name for    a  Markov kernel and can also be interpreted as a regular conditional probability. If   $\Yy$  is    a  separable  metrizable   topological  space  with  the Borel  $\sigma$-algebra  $ \Bb (\Yy)$,  we propose an alternative  characterization  of a regular  conditional  probability measure $\mu_{\Yy|\Xx}$    as a minimizer of   mean square error  on the space of all probabilistic morphisms  from $\Xx$ to $\Yy$,
 employing    kernel mean embeddings.
 
\subsection{Previous works}\label{subss:Pw}

1.   In his  book ``Statistical  Learning  Theory" \cite{Vapnik1998}   Vapnik discussed two approaches to the learning problem: the first approach evaluates the quality of the chosen function using a risk function, while the second approach involves estimating stochastic dependencies.  According to Vapnik, by estimating stochastic dependencies, problems such as pattern recognition and regression estimation can be solved effectively 
\cite[p. 19]{Vapnik1998}.  Vapnik connected the estimation of the supervisor operator  $[\mu_{\Yy|\Xx}]$,   which corresponds to estimating stochastic dependencies in supervised learning, to classical aspects of supervised learning theory (such as pattern recognition and regression estimation problems) through the Bayes decision rule  \cite[p. 37]{Vapnik1998}.  It is important to note that Vapnik considered conditional probability measures  $[\mu_{\Yy|\Xx}]$  but did not explicitly consider their regular versions, denoted here as    $\mu_{\Yy|\Xx}$.  However, in all cases where  $\Xx, \Yy$ are    Borel   subsets  of  $\R^n$,   a regular conditional probability measure   $\mu_{\Yy|\Xx}$  exists for any  Borel   probability measure $\mu$ on  $\Xx \times \Yy$ with respect to the projection $\Pi_\Xx: \Xx\times \Yy \to \Xx$. For further  details, refer to  Subsection \ref{subs:not}. Vapnik addressed the problem of conditional probability estimation in multi-classification supervised learning as well as conditional density estimation in supervised learning when  $\Xx \subset \R^n$  and $\Yy \subset \R$.  He proposed a class of solutions for these problems \cite[p. 36-39, 333-334, 337-338]{Vapnik1998}. See 
 also Examples \ref{ex:densitycorrect}(3 $\&$ 4)  and  Remark \ref{rem:vapnik}  for additional information.

2.  The significance of conditional probability in classical supervised learning is evident in Bayes' decision rule, which is discussed in detail in  \cite[Chapter 2]{DGL1996}. In \cite{VI2019}  Vapnik-Izmailov  considered  the machine learning  problem as  a problem of estimating the conditional  probability function rather  than problem  of of finding function  that minimizes  a given loss function. In their work, Talwai, Shameli, and Simchi-Levi \cite{TSS2022} treated the estimation of conditional probability as a supervised learning model, utilizing kernel mean embedding.    Their work builds upon the results of Park-Muandet  \cite{PM2020}, where  the authors  formalized  Gr\"unerw\"alder and co-authors' proposal to use   mean square loss for estimating  conditional mean embedding \cite{GLGB12}. For precise formulations of their results related to the present paper, refer to  Remark \ref{rem:gl} (2 $ \&$  3)   and  Remark  \ref{rem:rkhs}.

\subsection{Main contributions} 

In this article, we present several  contributions to  statistical learning  theory:

1. Introduction of the Generative Model: We introduce the concept of a generative model of supervised learning, which serves as a comprehensive framework encompassing all models in supervised learning theory, as well as models of probability measure estimation. In particular, by incorporating the concept of a correct loss function, we  unify   two approaches   to statistical learning  mentioned  above   in the  Generative  Model.

2. Characterization of Regular Conditional Probability Measures: We provide characterizations of regular conditional probability measures  $\mu_{\Yy|\Xx}$, which are essential for understanding the underlying probabilistic relations in supervised learning.  These characterizations allow us to establish a solid foundation for our subsequent analyses and investigations (Theorems \ref{thm:marginal}, \ref{thm:errorRKHS}).

3. Identification of Correct Loss Functions: By leveraging our characterizations of regular conditional probability measures, we identify and analyze various examples of correct loss functions. These include well-known loss functions such as the mean square error, the 0-1 loss function,  the log-likelihood function and the Kullback-Leibler divergence widely used in statistics and statistical learning theory. We also highlight other natural loss functions employed in mathematical statistics and statistical learning theory  (Examples \ref{ex:densitycorrect},  Examples \ref{ex:kernel}).

4. Quantifying  the Consistency  of a  Learning  Algorithm  using Inner Measure: We propose the use of inner measure, equivalently, convergence in outer probability, to quantify the consistency of a learning algorithm. We also provide a sufficient condition for the consistency of a learning algorithm. These findings shed light on the potential of a learning algorithm to perform well on unseen data (Definition \ref{def:genbi}, Theorem \ref{thm:uniform},  Remarks \ref{rem:genbi}, \ref{rem:csl2}).   In addition, relaxing the convergence  in probability in the classical  theory of statistical  learning by convergence in outer probability in our  theory  made  many  previous  results  in statistical learning  theory rigorous (Remarks \ref{rem:genbi}, \ref{rem:csl2}, \ref{rem:rkhs}, \ref{rem:vapnik}.)

5. Generalization of Cucker-Smale's Result: Building upon our characterizations of regular conditional probability measures, we present a generalization of Cucker-Smale's result (Theorem \ref{thm:msek}, Remark \ref{rem:rkhs}). 

6. Variant of Vapnik-Stefanyuk's Regularization Method: We introduce a variant of Vapnik-Stefanyuk's regularization method for solving stochastic ill-posed problems. By incorporating inner measure, we demonstrate its effectiveness in proving the learnability of conditional probability estimation problems (Theorem \ref{thm:vapnik73}, Remark \ref{rem:vapnik}, Theorem \ref{thm:main}, Corollary \ref{cor:main}. Example \ref{ex:main}). As  a result,  we provide   examples of  learnable  ``overparameterized"   discriminative models of supervised learning (Corollary \ref{cor:main}(2)).

\subsection{Organization of  this   article}\label{subs:org}
This article is organized   into several sections  as follows. 

Section \ref{sec:unified}: Bounded s-Probabilistic Morphisms.
In this section, we introduce the concept of a bounded s-probabilistic morphism, which extends the notion of a probabilistic morphism. Bounded  s-probabilistic morphisms are useful  for our investigation of variations of probabilistic morphisms. We characterize a regular conditional probability measure  $\mu_{\Yy|\Xx}$
among bounded s-probabilistic morphisms as a solution to a linear operator equation (Theorem \ref{thm:marginal}).  Additionally, this section includes various technical results that will be utilized in later parts of the paper.

 Section  \ref{sec:gen}: Generative Model of Supervised Learning.
 In Section  \ref{sec:gen}, we utilize probabilistic morphisms to introduce the concept of a generative model of supervised learning and the notion of a correct loss function. Drawing on the results from Section \ref{sec:unified}, we illustrate these concepts using examples such as the Fisher-Wald setting of density estimation and Vapnik's setting of conditional probability estimation. We propose the concept of consistency of a learning algorithm using inner measure (Definition \ref{def:genbi}), which is equivalent  to the concept  of convergence in outer  probability (Remark \ref{rem:genbi}),  and provide a sufficient condition for the consistency of a learning algorithm  (Theorem \ref{thm:uniform}). Furthermore, we discuss the relationship between these concepts and classical notions of consistency (Remarks \ref{rem:genbi}, \ref{rem:csl2}).  We also introduce the concept of a C-empirical risk minimizing (C-ERM) algorithm, which guarantees the existence of such algorithms for any statistical learning model (Definition \ref{def:aserm}, Remark \ref{rem:cerm}).
 
 Section \ref{sec:loss}: Regular Conditional Probability Measures  via Kernel Mean Embedding.
 Assuming that $\Yy$ is a separable metrizable topological space, this section builds upon the results of Section \ref{sec:unified} to characterize regular conditional probability measures  $\mu_{\Yy|\Xx}$  as the minimizers of mean square error (Theorem \ref{thm:errorRKHS}, Corollary \ref{cor:correct}). We discuss related findings by Park-Muandet \cite{PM2020}, Gr\"unerwalder et al. \cite{GLGB12}, and Talwai-Shami-Simchi-Levi \cite{TSS2022} in the context of our analysis (Remark \ref{rem:gl} (2 $\&$ 3)).

 Section  \ref{sec:genas}: Generalization of Cucker-Smale's Result.
 Drawing on the previous sections, we present a generalization of a result originally established by Cucker-Smale (Theorem \ref{thm:msek}). We also discuss related findings and implications in Remark \ref{rem:rkhs}.
 
 Section  \ref{sec:vapnik} : Variant of Stefanyuk-Vapnik's Result and Applications.
 In this section, we present a variant of Vapnik's result that incorporates inner measure (Theorem \ref{thm:vapnik73}). We explore its applications in conditional probability estimation problems, discussing the implications in
 Remark \ref{rem:vapnik73}, Theorem \ref{thm:main}, Corollary \ref{cor:main} and Example \ref{ex:main}.

Last Section \ref{sec:discuss}: Discussion of Results.
In the final section of the paper, we  summarize the main findings, highlight their significance, and offer potential future research directions.

\section{Bounded s-probabilistic morphisms}\label{sec:unified}
In this section  we introduce  the concept  of  a  {\it bounded s-probabilistic morphism}   and compare  it with  the concept of a probabilistic morphism and with  the concept  of  an s-finite kernel (Definitions \ref{def:probm}, Remark \ref{rem:s-finite}).    We  then study the properties  of   bounded s-probabilistic  morphisms (Definition  \ref{def:graph}, Lemmas \ref{lem:meast}, \ref{lem:diag}, \ref{lem:almostsurely}). Finally,   we provide a characterization  of  regular  conditional  measures using  bounded  s-probabilistic morphisms (Theorem \ref{thm:marginal}).

 \subsection{Notation, conventions, and preliminaries}\label{subs:not}
 
 \
 
 $\bullet$ Given a measurable  space $\Xx$, let     $\Sigma_\Xx$  denote the  $\sigma$-algebra of $\Xx$. We use  $\Ss(\Xx)$, $\Mm(\Xx)$  and $\Pp(\Xx)$  to represent the space  of all finite  signed  measures, the  space of all finite (nonnegative)   measures, and the space of all probability measures  on $\Xx$, respectively. The total variation norm    on $\Ss(\Yy)$ is denoted  by $\| \cdot \|_{TV}$.
 
 $\bullet$  For a given $x \in \Xx$, $\delta_x$  represents the Dirac measure
 concentrated  at $x$.
 
 $\bullet$ For   a data  set $S = (x_1, \ldots, x_n)\in \Xx^n$, 
 $\mu_S$  denotes  the empirical probability  measure  defined as 
 $$\mu_S: = \frac{1}{n} \sum_{i =1}^n \delta_{x_i}  \in \Pp (\Xx).$$

 $\bullet$  $\Pp_{emp} (\Xx)$  represents the set of  all empirical probability measures  on $\Xx$.

 $\bullet$ When $\Xx$   is a topological space, we   
  consider the   Borel  $\sigma$-algebra  $\Bb(\Xx)$ unless  stated otherwise.
 
 $\bullet$  For  a measurable   mapping $f : \Xx \to \Yy$  and $\mu \in \Ss (\Xx)$,  $f_* \mu$  denotes  the pushforwarded  measure  on $\Yy$,  defined as  $f _* \mu  (B) : = \mu  ( f^{-1} (B))$ for any $B \in \Sigma_\Yy$.

 $\bullet$  The product  $\Xx \times \Yy$ of  measurable spaces $(\Xx, \Sigma_\Xx)$ and  $(\Yy, \Sigma_\Yy)$  is  assumed to be  endowed  with the  $\sigma$-algebra  $\Sigma_\Xx \otimes \Sigma_\Yy$ unless otherwise  stated.  For a  probability measure  $\mu \in \Pp (\Xx\times \Yy)$, $\mu_\Xx: = (\Pi_\Xx)_* \mu$  represents the marginal  probability  measure of $\mu$ on $\Xx$.
 
 $\bullet$  In this paper we   consider only   measurable spaces  $\Xx, \Yy$  such  that   any $\mu \in \Pp (\Xx \times  \Yy)$ has a  regular  conditional probability measure
 $\mu_{\Yy|\Xx}$ for $\mu$  relative to $\Pi_\Xx$.  If $\Yy$ is  a Souslin measurable space  (a measurable space isomorphic to a measurable space associated with a Souslin metrizable space  \cite[Definition 16, p.46-III]{DM78},  \cite[Remark 1]{JLT2021}),   then this regular conditional probability measure    $\mu_{\Yy|\Xx}$ always exists,   \cite[Corollary 10.4.15, p.  366, vol. 2]{Bogachev2007}, \cite[Theorem 3.1 (6)]{LFR2004}.

 $\bullet$  $\mathbf{Meas} (\Xx, \Yy)$ represents the  space of all   measurable  mappings  from $\Xx$ to $\Yy$.  If $\Xx, \Yy$  are topological  spaces,  $C(\Xx, \Yy)$  denotes the space of all continuous   mappings  from $\Xx$ to $\Yy$.
 
 $\bullet$ $\Id_X$  represents the identity mapping on $X$.

$\bullet$  For $A \in \Sigma_\Xx$, $1_A$  denotes the characteristic  function of $A$.

 $\bullet$  $\Ff_s(\Xx)$ denotes the vector space  of simple (step) functions on $\Xx$  and  $\Ff_b(\Xx)$ denotes   the vector space  of  measurable bounded  functions on $\Xx$.  $\Ff_b (\Xx)$ is  a Banach  space  equipped with  the sup-norm $\| \cdot \|_\infty$.
 
 $\bullet$ We endow  $ \Ss (\Xx)$  with    the   $\sigma$-algebra $\Sigma_w$,  which is the smallest  $\sigma$-algebra such that  for any $f \in \Ff_s (\Xx)$,  the map 
 $$I_f: \Ss (\Xx) \to \R, \mu \mapsto   \int _\Xx f \, d\mu , $$   is measurable.
The restriction of $\Sigma_w$ to $\Pp (\Xx)$  and $\Mm (\Xx)$    is   also denoted by $\Sigma_w$ \cite{Lawvere1962},  \cite{JLT2021}.

$\bullet$   Let  $C_b(\Xx)$ be the space 
 of all bounded  continuous functions  on   a topological  space 	 $\Xx$. We denote by $\tau_w$ 
 the weakest topology on  $\Ss(\Xx)$  such that 
 for any  $f\in C_b (\Xx)$    the map $I_f: (\Ss(\Xx), \tau_w)\to \R$ is continuous. We also denote by
 $\tau_w$ the     restriction  of $\tau_w$  to $\Mm(\Xx)$ and $\Pp(\Xx)$.   
 If $\Xx$ is separable and metrizable then $(\Pp(\Xx), \tau_w)$ is separable and  metrizable  \cite[Theorem 3.1.4, p. 104]{Bogachev2018},  
 \cite[Theorem 6.2, p.43]{Parthasarathy1967},  and  the Borel   $\sigma$-algebra
 $\Bb (\tau_w)$ on $\Pp(\Xx)$ generated by $\tau_w$  coincides with $\Sigma_w$  \cite[Theorem 2.3]{GH89}.

 \begin{lemma}\label{lem:fsb}   For any  $h\in \Ff_b (\Xx)$  the  evaluation mapping $I_h :  (\Ss (\Xx), \Sigma_w) \to \R, \mu \mapsto  \int _\Xx h d\mu,$  is  a measurable mapping. Consequently,  $\Sigma_w$ is the smallest $\sigma$-algebra    such that  $I_h :  (\Ss (\Xx), \Sigma_w) \to \R$ is measurable  for any $h \in \Ff_b (\Xx)$.
 \end{lemma}
 \begin{proof} Let $h \in \Ff_b (\Xx)$. Then  there  exist    sequences   of   simple functions $\{ h_n ^\pm, n \in \N^+\}$   such that 
 	$$ h_n ^-  (x)\le h (x) \le   h^+ _n  (x) \text{  for all }   x \in \Xx \text{ and } \| h_n ^- - h_n ^+ \| _\infty  \le \frac{1}{n},$$
 	see, e.g., \cite[p. 66]{Chentsov72}.  It follows  that  for  any $\mu \in \Ss(\Xx)$ we have
 	\begin{equation}\label{eq:converg1}
 	\lim_{n \to \infty} \int _\Xx h^\pm_n d \mu  = \int _\Xx h d\mu.
 	\end{equation}
 	In other  words,  the  sequence of functions  $I_{ h^\pm _n}:  \Ss (\Yy)\to \R$ converges  to $I_h$ pointwise.  Hence, $I_h$  is also measurable. The last  assertion  follows since  $\Ff_s (\Xx) \subset \Ff_b (\Xx)$.
 \end{proof}

 \begin{lemma}\label{lem:scount} Assume that  $\Xx$ is  a   complete  separable metric  space. Let $\Aa: =\{A_i| i \in \N^+\}$  be the  collection of closed
 	balls   of rational  radius centered   at a countable dense   subset in $\Xx$. Then  for any $A \in \Sigma_\Xx$ and  $\mu \in \Ss(\Xx)$  we have
 	\begin{equation}\label{eq:tvar}
 	|\mu|  (A) = \sup \{|\mu |(A\cap  \cup_{i =1}^N  A_i):\, N \in  \N^+, \, A_i \in\Aa \}.
 	\end{equation}
 	
 \end{lemma}
 \begin{proof}
 	Since $\Xx$ is    a complete  metric space,  $\mu$ is   a Radon measure, see, e.g., \cite[Theorem 7.4.3, p. 85,  vol. 2]{Bogachev2007}, i.e.,  for any $\eps > 0$  there  exists     a compact set  $K_\eps \subset A$  such  that
 	\begin{equation}\label{eq:radon}
 	|\mu| (A \setminus K_\eps) \le  \eps.
 	\end{equation}
 	Let $\{ A_1, \ldots, A_n\in \Aa\}$    be a  finite  cover  of $K_\eps$.  Then it follows  from \eqref{eq:radon}
 	\begin{equation}\label{eq:gesup}
 	|\mu|  (A) \le \sup \{| \mu| (A\cap  \cup_{i =1}^N A_i):\,  N \in  \N^+, A_i \in \Aa\} + \eps. 
 	\end{equation}
 	Taking into account
 	$$ |\mu|  (A) \ge \sup \{| \mu| (A\cap  \cup_{i =1}^N A_i):\,  N \in  \N^+, A_i \in \Aa\}, $$
 	this   completes  the  proof of Lemma \ref{lem:scount}.
 \end{proof}
 
 \begin{corollary}\label{cor:scount} Let $\Ff_\Aa$ be  the countable family   consisting of  all finite  disjoint unions  of   elements in the countable  algebra $G(\Aa)$ generated by $\Aa$.  Then  we have
 	\begin{equation}\label{eq:scount}
 	|\mu| (A) = \sup \{\sum_{i =1}^N |\mu (A \cap B_i)|: \, \dot\cup_{i =1} ^N B_i \in \Ff_\Aa,\, B_i \in G(\Aa)\}.
 	\end{equation}	
 \end{corollary}
 
 \begin{proof}   Since $\dot\cup_{i =1} ^N B_i$ is a disjoint union, we have
 	\begin{equation}\label{eq:tvar1}
 	|\mu|  (A) \ge  \sup \{ \sum_{i =1}^N |\mu (A \cap B_i)|: \, \dot\cup_{i =1} ^N B_i \in \Ff_\Aa,\, B_i \in G(\Aa)\}.  
 	\end{equation}
 	Now let $\Xx = \Xx^+ \cup \Xx^-$  be the  Hahn decomposition   of $\Xx$ for  $\mu$  and $\mu  = \mu^+ - \mu^-$ be the Jordan-Hahn decomposition  of $\mu$. Then  
 	\begin{equation}\label{eq:ajh}
 	|\mu| (A) =  |\mu( A^+)| + |\mu  ( A^-)|, \text{ where } A^\pm: = \Xx^\pm \cap A.
 	\end{equation}
 	Applying Lemma \ref{lem:scount}  to  $\mu, A ^\pm$ we obtain  from \eqref{eq:ajh}
 	\begin{equation}\label{eq:tvar2}
 	|\mu|  (A) \le  \sup \{  \sum_{i =1}^N|\mu (A \cap B_i)|: \,  \dot\cup_{i =1}^N B_i \in \Ff_\Aa,\, B_i \in G(\Aa)\}.  
 	\end{equation}
 	Taking into account\eqref{eq:tvar1}, this  completes  the  proof of Corollary \ref{cor:scount}.
 \end{proof}

\begin{proposition}[Measurability of the Jordan-Hahn Decomposition and Total Variation]\label{prop:JH}
	
		Assume that $\Xx$ is a Polish space.
 The  map  ${\rm v}:(\Ss (\Xx),\Sigma_w) \to (\Mm (\Xx),\Sigma_w), \, \mu \mapsto | \mu|,$ is a measurable function. Consequently, the  Jordan-Hahn  decomposition   $\Ss(\Xx) \to \Mm(\Xx)\times \Mm(\Xx),  \mu \mapsto (\mu^+, \mu^-) ,$ is also measurable mapping. Additionally, the map  $\Ss(\Xx) \to \R, \, \mu \mapsto  \| \mu\|_{TV},$  is  measurable.
\end{proposition}
\begin{proof}
By the  definition of $\Sigma_w$, to prove  that  the map ${\rm v}: (\Ss(\Xx),\Sigma_w) \to \Mm(\Xx), \mu \mapsto |\mu|,$ is  measurable, it suffices  to show that for   any $A \in\Sigma_\Xx$  the    composition:  $I_{1_A}\circ  {\rm v}: \Ss (\Xx) \to  \R_{\ge 0}$  is  a measurable function.
Let $ \Ff_\Aa : = \{ \Ff_1, \ldots, \Ff_n, \ldots\}$. By Corollary \ref{cor:scount}, we have
	\begin{equation}\label{eq:vn}
	I_{1_A}\circ  {\rm v} (\mu) = \sup_{ i }  f^A_i (\mu),
	\end{equation} 
	where
	$$ f^A _i (\mu)  = \sum_{B_{ij} \in\Ff_i} | \mu  ( A\cap B_{ij})|: \, \dot\cup_{j =1}^n B_{ij}= \Ff_i \in \Ff_\Aa,\, B_i \in G(\Aa).$$
Since $\Ff_i$ is fixed,   taking into account  that  the function $ \R \to \R_{\ge 0}, x \mapsto |x|$, is measurable,    the      function $ f^A_i: \Ss(\Xx) \to \R, \, \mu \mapsto   f^A_i  (\mu), $ is also measurable  for  all $i$.
	Since  for any $-\infty < a   < \infty $  we have
	$$ (\sup_i f^ A_i)^{-1} \big ((-\infty, b] \big) = \bigcap_{k = 1}^\infty\bigcup_{ i =1} ^\infty (f^A_i)^{-1} (-\infty,b + \frac{1}{k}], $$
	taking into account  \eqref{eq:vn} and   the fact  that  $\Bb(\R)$ is generated by the     sets  $\{(-\infty, b],\, b \in \R\}$,  we conclude that  the function $I_{1_A}\circ {\rm v}: \Ss(\Xx) \to \R$  is  a measurable.
This   completes     the  proof of the first aseertion of  Proposition \ref{prop:JH}.
	 
	The second  and third
	assertion  of Proposition \ref{prop:JH}  follow  immediately.	
\end{proof}

\subsection{Bounded s-probabilistic morphisms and their joints}\label{subs:probm}
Following Lawvere \cite{Lawvere1962}, we consider the $\sigma$-algebra $\Sigma_w$ on $\Pp (\Yy)$ and treat
   Markov kernels  from $\Xx$ to $\Yy$ as   measurable mappings from  $\Xx$ to $\Pp(\Yy)$.

\begin{definition}\label{def:probm} cf. \cite[Definition 1]{JLT2021}, cf. \cite{Lawvere1962}. Let  $\Xx, \Yy$  be measurable space.  A {\it  bounded s-probabilistic  morphism} $T: \Xx \leadsto \Yy$  is a measurable  mapping $\overline T: \Xx \to (\Ss (\Yy), \Sigma_w)$  such that  the mapping  $\overline T: \Xx \to \Ss (\Yy)_{TV}$ is  bounded, i.e.   there exists a  constant $C > 0$ such  that $\|\overline  T(x)\|_{TV} \le C$  for all $ x\in \Xx$. In this case we shall say that $\overline T: \Xx \to \Ss(\Yy)_{TV}$ is a {\it bounded  measurable  map}. A  bounded s-probabilistic morphism $T: \Xx\leadsto \Yy$ shall  be  called  a {\it probabilistic morphism}, if $\overline T (\Xx) \subset \Pp (\Yy)$.
\end{definition}

\begin{remark}\label{rem:s-finite}  (1) The space  of all   bounded  s-probabilistic morphisms    is a normed  vector  space  with the   sup-norm
	$\| T\|_\infty : = \sup_{x\in \Xx}   \| \overline T(x)\|_{TV}$.
	
(2) A  kernel $k$  from $\Xx$ to $\Yy$ is function  $k: \Xx \times \Sigma_\Yy \to [0, \infty]$  such that   for any $x\in \Xx$  the function
$k(x, \cdot) : \Sigma_\Yy \to [0, \infty]$ is a measure and   for any $B \in \Sigma_\Yy$ the function $k(\cdot , B)  : \Xx \to [0, \infty]$ is measurable.  A kernel   $k$ is finite, if   there is a finite $r \in [0, \infty)$  such that   $k (x, \Yy)  < r $  for all $x \in \Xx$.  Clearly any   finite  kernel  $k$  generates  a     bounded  measurable  mapping  $k: \Xx \to (\Mm (\Xx), \Sigma_w), x \mapsto k (x, \cdot),$  and hence  a bounded  s-probabilistic  morphism.    Proposition  \ref{prop:JH}  implies that  if    $\Yy$  is a Polish space, then any   bounded   measurable  mapping $\overline  T: \Xx \to \Ss (\Yy)$    can be  written as  $T =  T^+  - T^-$, where
	$T^\pm : \Xx \to \Mm (\Yy)$ are  finite  kernels.
	  
(3) An s-finite   kernel is  a  map $k:  \Xx \times \Sigma_\Yy \to [0, \infty]$  such that
there  exists  a  sequence  $k_1, \ldots, k_n$ of  finite kernels  such   that 
$k= \sum_{i =1}^\infty  k_i$ \cite[Definition 2]{Staton2017}.  The space  of all  s-finite  kernels  is a cone.  According to  Staton, the  concept of an s-finite  kernel  has been  introduced by Kallenberg \cite{Kallenberg2014} and by Last-Penrose \cite{LP16}.
\end{remark}

Let $\mathbf {sbProbm} (\Xx, \Yy)$ denote the   vector  space of all  bounded s-probabilistic  morphisms   from $\Xx$  to  $\Yy$ and  by $\mathbf{Probm}(\Xx, \Yy)$ the set  of all probabilistic morphisms from  $\Xx$ to  $\Yy$.

For a bounded s-probabilistic  morphism $T: \Xx \leadsto \Yy$, we denote the associated    measurable mapping  as $\overline T:\Xx \to \Ss (\Yy)$  .  Conversely, for a  bounded measurable  mapping $f: \Xx \to \Ss (\Yy)$ we denote  the generated  bounded s-probabilistic  morphism as $\underline{f}: \Xx \leadsto  \Yy$.

\begin{examples}\label{ex:probm} (i)
Any   regular      conditional probability measure  $\mu_{\Yy|\Xx}$ for $\mu \in \Pp (\Xx \times \Yy)$ with respect  to  the projection $\Pi_\Xx: \Xx\times \Yy \to \Xx$   assigns  a  measurable  mapping $\mu_{\Yy| \Xx}: \Xx  \to \Pp (\Yy)$ by the formula  $\mu_{\Yy|\Xx}  (x) (A) = \mu _{\Yy| \Xx} (A|x)$ for  $x \in \Xx$ and $A \in \Sigma_\Yy$. We shall also call  the map $\mu_{\Yy|\Xx}:\Xx \to \Pp(\Yy)$ a regular conditonal probability measure  for $\mu$. The associated  probabilistic morphism is denoted  by $\underline{\mu_{\Yy|\Xx}}: \Xx \leadsto \Yy$.

(ii)  The    map $\delta: \Xx \to (\Pp (\Xx),\Sigma_w),  x \mapsto \delta (x) : = \delta_x$,  is measurable  \cite{Lawvere1962},  \cite[\S 1.2]{Giry1982}. If $\Xx$  is a topological space, then the map
$\delta: \Xx \to (\Pp(\Xx), \tau_w)$  is continuous, since  the composition $I_f \circ \delta: \Xx \to \R$ is continuous  for any  $f \in C_b (\Xx)$.
  We  regard   a measurable  mapping $\kappa:\Xx \to \Yy$ as a deterministic probabilistic  morphism defined  by   $\overline\kappa: = \delta \circ \kappa: \Xx \to \Pp(\Yy)$.
In particular,   the identity  mapping   $ \Id_\Xx:\Xx\to\Xx$  of a measurable   space $\Xx$   is a probabilistic morphism generated by   $ \delta:  \Xx  \to \Pp(\Xx)$, so  $\delta = \overline{\Id_\Xx}$. 
Graphically speaking,  any straight  arrow (a  measurable mapping)  $\kappa: \Xx \to \Yy$ between  measurable  spaces  can be  seen as  a curved arrow (a probabilistic  morphism).
\end{examples}

Given a  bounded s-probabilistic  morphism, following  Chentsov \cite[Lemma 5.9, p. 72]{Chentsov72},  we define a  linear map $S_* (T): S(\Xx) \to \Ss(\Yy)$ as follows

\begin{equation}\label{eq:markov1S}
S_* (T) (\mu)  (B): =\int _\Xx \overline  T (x) (B)d\mu (x).
\end{equation}
for any $\mu \in \Ss(\Xx)$ and $B \in \Sigma_\Yy$.
Following   \cite[(5.1), p. 66]{Chentsov72},   we define   a linear map  $T^*:\Ff_b (\Yy) \to \R^\Xx$  by letting
\begin{equation}\label{eq:markov*}
 T^* (f) (x):  = \int _\Yy f d\overline  T(x)  \text{ for }  x \in \Xx.
 \end{equation}

\begin{lemma}\label{lem:boundedf}   The  map  $T^*$ is a bounded  linear  map between Banach spaces $\Ff_b (\Yy)_\infty$  and   $\Ff_b(\Xx)_\infty$.
\end{lemma}

\begin{proof}   We  shall use the   Chentsov   argument  in his  proof   for the case    that  $T \in  \mathbf{Probm} (\Xx, \Yy)$  \cite[Corollary of Lemma 5.1, p. 66]{Chentsov72}.   First  we   shall show that  $T^* (\Ff_b (\Yy)) \subset \Ff_b (\Xx)$.  Note  that  $T^* (1_B) (\cdot ) = \overline T (\cdot )(B) \in \Ff_b (\Xx)$,  since $\overline T(x) (B) \le  \| T (x)\|_{TV}$.    Now  let   $h \in  \Ff_b (\Yy)$.    Then there  exist  sequence  of simple  functions   $\{ h_n ^\pm , n \in  \N^+\}$ on $\Yy$ such that
\begin{equation}\label{eq:hn}
	h_n ^- (y)\le  h (y) \le h^+_n (y)     \text{ for all } y \in \Yy \text { and } \| h_n ^- - h_n ^+\|_\infty \le \frac{1}{n}.
\end{equation}
Let $x \in \Xx$. Then   
$$ T^* h (x) =  \int _\Yy h\, d  \overline {T}(x)^+  - \int _\Yy h\,  \overline  T(x)^-.$$

For  $x \in \Xx$ and $f \in \Ff_b (\Yy)$, we  shall use  the  following shorthand notation  
$$ (T^\pm) ^* (f) (x): = \int _\Yy f \, d  \overline {T}(x)^\pm.$$
 By \eqref{eq:hn}, for any $x \in  \Xx$ and any $n\in \N^+$,   we obtain
 \begin{equation}\label{eq:est0}
  (T ^\pm)^* (h_n^-) (x) \le   (T ^\pm)^* (h) (x) \le    (T ^\pm)^* (h_n ^+)(x). 
 \end{equation}
 From \eqref{eq:hn} and \eqref{eq:est0}, it follows  that  for any $x\in \Xx$     
 \begin{equation*}
 \lim_{n \to \infty}   (T^\pm)^* (h_n^-) (x) = (T^\pm)^* h(x) = \lim_{n \to \infty} (T^\pm)^* (h_n ^+) (x).
 \end{equation*}
 Hence, for  any $x \in \Xx$ we have
 $\lim_{n \to \infty} T^*h (x) =  \lim_{n \to \infty} T^* (h^\pm_n)$.
 It  follows  that $T^* h \in \Ff_b (\Xx)$, since we have shown  that $ T^* (h^\pm_n) \in \Ff_b (\Xx)$.  
 
 Finally,   we note   that  $\| T^*  f\|_\infty  \le  \| T\| _\infty \cdot  \| f \| _\infty$. Hence,  $T^*$ is a bounded  linear  map. This    completes  the  proof of Lemma \ref{lem:boundedf}.
 
\end{proof}

For $\overline T \in \mathbf{Meas} (\Xx, (\Ss(\Yy),\Sigma_w))$, $x\in \Xx$,  and $y\in \Yy$, we shall write $d\overline T (y|x)$  instead  of $d\overline T  (x) (y)$, which  has been  denoted as $T (dy|x)$ by Chentsov \cite{Chentsov72}.
\begin{definition}\label{def:comp} Given  $T_1 \in \mathbf{sbProbm} (\Xx, \Yy)$, $T_2 \in \mathbf{sbProbm} (\Yy, \Zz)$,  the  composition   $\overline {T_2 \circ T_1}: \Xx \to \Ss(\Zz)$ is defined as follows
\begin{equation}\label{eq:comp}
\overline{T_2 \circ T_1}  (x, C) = \int_\Yy  \overline T_2  (y, C) d\overline T_1  (y|x) 
\end{equation}
for  any $x\in \Xx$,  $ C \in \Sigma_\Zz$.
\end{definition}

\begin{lemma}\label{lem:compo} (1) The composition $T_2 \circ T_1$ of  two bounded s-probabilistic morphisms  is   a bounded   s-probabilistic  morphism. 
	
(2)	 Furthermore,  the composition is associative, i.e.  $(T_3 \circ T_2) \circ  T_1 = T_3 \circ  (T_2 \circ T_1)$.
	 
	(3) For any    $T_i \in \mathbf{sbProbm} (\Xx_i,\Xx_{i+1} )$,  $i \in [1, 2]$,  we have
	 \begin{equation}\label{eq:composti}
	 \overline{T_2 \circ T_1} = S_*(T_2) \circ \overline {T_1}.
	 \end{equation}
\end{lemma}
Lemma  \ref{lem:compo}  for  probabilistic morphisms are well-known  \cite{Giry1982}, see also  \cite[Lemmas 5.4-5.6, p. 68-69]{Chentsov72}.

\begin{proof}[Proof of Lemma \ref{lem:compo}] 	
	(1)    First  we note that $\overline {T_2 \circ T_1}$  is   bounded map   from $\Xx \to \Ss (\Zz)_{TV}$,  since  $\overline  T_1  $ and $\overline T_2$ are bounded  mappings.     To show that    $T_2 \circ  T_1: \Xx \to (\Ss (\Zz),\Sigma_w) $ is  a measurable map,   it suffices  to show that  for any   $C \in \Sigma_\Zz$ the function
	$$I_{1_C}: \Xx\to \R, x \mapsto  I_{1_C}  \big ( \overline {T_2 \circ T_1} (x)\big ) $$
	is measurable.     Since   
	$$  I_{1_C}\big ( \overline {T_2 \circ T_1} (x)\big ) = \int _\Yy  \overline T_2  (y, C)\,  d\overline T_1  (y|x)  =  (T_1 )^* (\overline T_2 (\cdot | C) )(x),$$
	by Lemma \ref{lem:boundedf}, the function $ I_{1_C}  \big ( \overline {T_2 \circ T_1} \big )$ belongs to $ \Ff _b (\Xx)$.  This proves the first assertion  of  Lemma \ref{lem:compo}.

	(2) Let  $T_1 \in \mathbf {sbProbm} (\Xx, \Yy)$, $T_2 \in \mathbf{sbProbm}(\Yy, \Zz)$ and  $T_3 \in \mathbf {sbProbm} (\Zz, \Vv)$. To prove Lemma \ref{lem:compo} (2), we  have to show that   for any $ x\in \Xx$  and $D \in \Sigma_\Vv$  we have:
	\begin{equation}\label{eq:comm}
	 \int_\Yy \int _\Zz \overline {T_3} (z)(D)\, d\overline{T_2} (z|y) d\overline{T_1} (y|x)  = \int _\Zz \int _\Yy  \overline {T_3} (z)(D)\,  d\overline{T_2} (z|y)d\overline{T_1} (y|x).
	\end{equation}
	In  order  to  prove \eqref{eq:comm}, we shall  approximate the bounded measurable  function $ \overline {T_3} (\cdot )(D) $ on $\Zz$ by  step  functions $h_n \in \Ff_s (\Zz)$. This reduces  to verify \eqref{eq:comm} for the case   that  $\overline {T_3} (\cdot)(D) =  1_C$ for some  $C \in \Sigma_\Zz$. In this case    straightforward  computations yield  
$$\int_\Yy \int _\Zz 1_C \overline{T_2} (y, dz) \overline{T_1} (x, dy) = 
\int_\Yy  \overline{T_2} (y)(C)  \overline {T_1} (x, dy) = \int _\Zz \int _\Yy  1_C  \overline{T_2} (y, dz)\overline{T_1} (x, dy) .$$
This completes  the proof  of  Lemma \ref{lem:compo}(2).

	(3) To prove the  last assertion of Lemma 
\ref{lem:compo},   we notice  that  $\overline {T_2 \circ T_1} (x)(C) = S_* (T_2) \circ \overline  {T_1} (x)(C) $  for any  $x \in \Xx_1$ and $C \in \Sigma_{\Xx_3}$  by  comparing \eqref{eq:comp} with \eqref{eq:markov1S}.  This  completes  the  proof of Lemma \ref{lem:compo}.
\end{proof}

\begin{remark}\label{rem:push} (1) If $T$ is a probabilistic morphism, then the restriction $M_*(T)$  
of $S_*(T)$ to  $\Mm(\Xx)$  and the  restriction $P_*(T)$  of $S_*(T)$  to  $\Pp (\Xx)$ maps  $\Mm(\Xx)$ to  $\Mm(\Yy)$ and $\Pp (\Xx)$ to
$\Pp(\Yy)$, respectively  \cite[Lemma 5.9, p. 72]{Chentsov72}.

(2)  If  a    probabilistic morphism      is deterministic, i.e.  it  is generated by
a measurable mapping  $\overline\kappa = \delta \circ \kappa: \Xx \to \Pp (\Yy)$,  then   $S_* (\kappa): \Ss(\Xx) \to \Ss(\Yy)$ is the   push-forward     operator  $\kappa_* : \Ss (\Xx) \to \Ss(\Yy)$, i.e.
\begin{equation}\label{eq:pf}
\Ss_* (\kappa) (\mu) (B) = \mu (\kappa ^{-1}(B)) \text{ for } \mu \in \Ss(\Xx) \text{ and } B \in \Sigma_\Yy.
\end{equation}

(3)  A  bounded measurable  mapping  $\overline T: \Xx \to (\Ss (\Yy),\Sigma_w)$  generates  two   homomorphisms:
$\overline  T_*: \Ss(\Xx) \to  \Ss (\Ss (\Yy), \Sigma_w)$ and   $T_*: \Ss(\Xx) \to \Ss(\Yy)$.
\end{remark}

In view  of Remark \ref{rem:push},  we  will use   the shorthand notation $T_*$  for $S_* (T)$, $M_* (T)$ and $P_* (T)$ when referring to  a  s-probabilistic  morphism $T$, as long as there is no possibility of confusion or misunderstanding.

\begin{lemma}\label{lem:meast}  Assume  that $T \in \mathbf{Probm} (\Xx, \Yy)$.
	
(1) Then $T_*:  (\Ss  (\Xx), \Sigma_w) \to  (\Ss (\Yy), \Sigma_w)$ is a measurable mapping.  In  particular,  the  map
$T_*: (\Pp  (\Xx), \Sigma_w) \to  (\Pp (\Yy), \Sigma_w)$  is  measurable.

(2) If $\overline  T \in C (\Xx, (\Pp (\Yy), \tau_w))$ then
$T_*:  (\Ss  (\Xx), \tau_w) \to  (\Ss (\Yy), \tau_w)$ is a continuous mapping.  In  particular,  the  map
$T_*: (\Mm  (\Xx), \tau_w) \to  (\Mm (\Yy), \tau_w)$  is  continuous.
\end{lemma}
\begin{proof} (1)  To prove the first assertion  of Lemma \ref{lem:meast}(1), it suffices  to show that  for any $f \in \Ff_b(\Yy)$ the  composition
$$I_f \circ   T_*: (\Ss(\Xx),\Sigma_\om) \to \R, \mu \mapsto I_f \circ T_* (\mu), $$
is a measurable  mapping. Note  that
$$I_f \circ T_* (\mu) = \int_\Yy  f d T_*\mu  \stackrel{\eqref{eq:markov1S}}{=} \int_\Xx \int_\Yy fd\overline T(x)\, d\mu (x)\stackrel{\eqref{eq:markov*}}{=} \int_\Xx T^* (f)\, d\mu. $$
Hence 
\begin{equation}\label{eq:compost*}
I_f \circ T_*  = I_{T^*f}.
\end{equation}  Using  Lemma \ref{lem:fsb}  we  conclude  the first assertion of Lemma \ref{lem:meast}(1) from  \eqref{eq:compost*}.  The second  assertion  of  Lemma \ref{lem:meast}(1)  follows immediately.

(2) The proof  of the second assertion  is similar to the  the proof of   the  first  one, noting that  if $\overline  T: \Xx \to (\Pp(\Yy), \tau_w)$ is continuous, then  for any $f \in C_b (\Yy)$ the function
$$ \Xx \to\R, x\mapsto  T^* f (x) =\int _\Yy f d\overline  T (x) =  (f  \overline  T (x)) (\Yy), $$
is continuous. Hence the composition
$$ I_f \circ  T_*:  (\Ss (\Xx), \tau _w)\to \R, \, \mu \mapsto  I_f \circ   T_* (\mu) = I_{ T^* f} (\mu),$$
is continuous.  This completes  the  proof  of Lemma \ref{lem:meast}.
\end{proof}

\begin{remark}\label{rem:meast} The   assertion ``in particular" of  Lemma \ref{lem:meast} (1)  has been  stated  by Lawvere \cite{Lawvere1962} and proved by Giry \cite{Giry1982}.
\end{remark}

Now we  are going  to define  the {\it joint   of two bounded s-probabilistic morphisms  with the same source}.  First  we need the following.

\begin{lemma}\label{lem:diag} (1) The
multiplication mapping
$$\mathfrak m: (\Ss(\Xx), \Sigma_w) \times (\Ss (\Yy), \Sigma_w) \to  \Ss (\Xx \times \Yy, \Sigma_w), \,  (\mu, \nu) \mapsto \mu \cdot \nu$$
is  measurable.  If $\Xx$ and $\Yy$ are topological  spaces,   then  $\mathfrak m$ is   $(\tau_w, \tau_w)$-continuous.
	
(2) The  diagonal  mapping 
$${\rm diag}:  (\Ss (\Xx), \Sigma_w) \to  (\Ss  (\Xx \times \Xx), \Sigma_w),\, \mu \mapsto   \mu ^2$$
 is   measurable.  If    $\Xx$ is  a  topological space, then ${\rm diag}$  is $(\tau_w, \tau_w)$-continuous.

 (3) For  any $n \in \N^+$ the addition $\mathfrak a ^n: (\Ss(\Xx) ^n, \otimes ^n \Sigma_w)\to  (\Ss(\Xx), \Sigma_w),\\
  (\mu_1, \ldots, \mu_n)
 \mapsto \sum_{i=1}^n \mu_i,$  is  measurable. Consequently  the  $n$-Dirac map 
 $$\delta ^n: \Xx^n \to (\Ss(\Xx), \Sigma_w), \, S \mapsto \mu_S,$$
 is also measurable.  If $\Xx$ is a  topological  space  then  $\delta ^n$ is $\tau_w$-continuous.
\end{lemma}
\begin{proof}
(1)	To prove  that the map  $\mathfrak m$ is measurable, it suffices  to show that  for any $A  \in \Sigma _\Xx$  and $B \in \Sigma_\Yy$ the map $I_{1_{A\otimes B}}: (\Ss (\Xx) , \Sigma_w) \times (\Ss ( \Yy), \Sigma_w) \to \R,\, (\mu, \nu) \mapsto \mu (A)\nu (B),$
is measurable.  The  map  $I_{1_{A \otimes B}}$ is measurable,  since  it can be written as  the composition of measurable  mappings
\begin{equation}\label{eq:mult} (\Ss (\Xx), \Sigma_w) \times (\Ss (\Yy), \Sigma_w) \stackrel{(1_A, 1_B)}{\longrightarrow}\R \times \R \stackrel{\mathfrak m_\R}{\longrightarrow \R}
\end{equation}
where   $\mathfrak m_\R  (x, y)  = x\cdot  y$.
	
Similarly, we  prove   that  $\mathfrak m$ is   $(\tau_w, \tau_w)$-continuous, if $\Xx$    and $\Yy$ are topological  spaces,  since  all  the   mappings in  \eqref{eq:mult}	are  continuous.

(2)    To   prove  the second assertion, we   write
${\rm  diag}  =  \mathfrak m \circ  {\rm Diag}$,  where
$$ {\rm Diag}: (\Ss (\Xx), \Sigma_w) \to  (\Ss (\Xx), \Sigma_w)  \times  (\Ss (\Xx), \Sigma_w), \, \mu \mapsto (\mu, \mu).$$
Taking into account the measurability of the map ${\rm Diag}$ and   of the map  $\mathfrak m$, we  conclude that ${\rm  diag}$ is a measurable map. 

If  $\Xx$ and $\Yy$ are topological spaces, then  ${\rm Diag}$ is a continuous map. Taking into account   Lemma \ref{lem:diag}(1),  we prove  the  continuity  of  ${\rm diag}$. 

(3) First we  prove the  case $n =2$.  To prove the  measurability  of the map $\mathfrak a^2$ it suffices to show that for any
$f \in \Ff_s(\Xx)$ the composition $I_f \circ \a: (\Ss(\Xx)\times \Ss(\Xx),\Sigma_w \otimes\Sigma_w) \to \R$ is measurable. Using the formula
$$I_f\circ  \a (\mu, \nu) = I_{1_\Xx}(f \mu) + I_{1_\Xx}(f \nu),$$
we reduce the  proof of the measurability  of $I_f \circ \a$  to proving the   measurability of the map
$\a^2: \R \times \R \to \R, (x, y) \mapsto (x+y),$ which is well-known. 

In the same way, we prove the continuity of the map $\a^2$, if $\Xx$ is a topological space.

For $n \ge 2$ we use the formula  $\a ^n  (\mu_1, \ldots, \mu_n)= \a^ 2( \a ^{n-1}(\mu_1, \ldots, \mu_{n_1}), \mu_n)$  and taking into account  the validity of the assertion for $n =2$  we  prove the first  part  of  assertion (3).

  Similarly, the  proof  of continuity  of  $\delta^n$ if $\Xx$ is topological  space can  be  reduced  to the  case $n =2$, which is well-known.
  
 This completes  the  proof of Lemma \ref{lem:diag}.
\end{proof}
  
\begin{remark}\label{rem:m}  For the  general abstract story behind the formation of the map $\mathfrak m$, defined in Lemma \ref{lem:diag}, see Kock \cite{Kock2011},  and Fritz-Perrone-Rezagholi \cite{FPR2021}  for a similar result.
\end{remark}

From now on  we shall drop  $\Sigma_w$, as long as  there is no possibility of confusion or misunderstanding.   

\begin{definition}\label{def:graph}  (1)  Given  two  bounded s-probabilistic morphisms
	$T_i: \Xx \leadsto \Yy_i$  $ i = 1,2$,     {\it the
joint  of  $T_1$ and $T_2$ }  is  the  bounded s-probabilistic morphisms  $T_1 \cdot T_2: \Xx \leadsto \Yy_1 \times \Yy_2$   whose  generating   mapping  $\overline{T_1 \cdot T_2}: \Xx \to \Ss (\Yy_1 \times \Yy_2)$  is given by:
$$\overline{T_1 \cdot T_2} (x) = \mathfrak m (\overline   T_1(x), \overline T_2(x)): \Xx \to \Ss (\Yy_1 \times \Yy_2).$$
	
(2)  Given  a      bounded  s-probabilistic  morphism  $T: \Xx \leadsto \Yy$  we denote  the joint  of  $\Id_\Xx$ with  $T$ by
  $\Gamma  _T: \Xx \leadsto  \Xx \times  \Yy$  and call  it  the   {\it graph  of $T$}. 
  \end{definition}

It follows  from  Definition  \ref{def:graph}(2)  that   for any $T \in \mathbf{sbProbm} (\Xx, \Yy)$, $A \in \Sigma_\Xx$ and $B \in \Sigma_\Yy$, we have \begin{equation}\label{eq:jointmeasure}
(\Gamma_T)_*\mu_\Xx (A \times  B ) = \int _\Xx \overline{\Gamma_T} (x) (A \times  B)\, d\mu_\Xx (x) = \int_A  \overline{T} (x) (B)\, d\mu _\Xx(x).
\end{equation} 
Hence,  for  any $f \in \Ff_b (\Xx \times \Yy)$, we have 
\begin{equation}\label{eq:jointmeasureb}
\int _{\Xx \times \Yy} f (x, y)\, d(\Gamma_T)_*\mu_\Xx (x, y) = \int_\Xx \int _\Yy f (x, y)  d \overline T(y|x)\mu_\Xx (x).
\end{equation} 

\begin{remark}\label{rem:graph} (1) The notion of a graph of a probabilistic morphism  $f$ has been appeared  first in Jost-L\^e-Tran \cite{JLT2021}, the arXiv version,  but without a definition.  The first  definition of this  concept  has  been  given in  Fritz-Gonda-Perrone-Rischel's paper \cite{FGPR2020}, where  they call  the graph of  $f$  the {\it input-copy version} or {\it bloom} of $f$.
	
(2) If  $\kappa: \Xx \to \Yy$ is a measurable mapping, then  $\overline{\Gamma_{\kappa}} (x, y) = \delta_x  \delta_{\kappa (x)} = \delta_{ (x, \kappa (x))}$. Thus  $\Gamma _{\kappa}$ is the graph of $\kappa$, i.e. 
	$\Gamma_{\kappa} (x, y) = (x, 	\kappa (y))$.
\end{remark}

We  shall  show the $(\tau_w, \tau_w)$-continuity  of the  map $(\Gamma_T)_*$ under certain conditions.

\begin{proposition}\label{prop:tcontinuous}  Assume  that $\Xx$ is a topological space, $\Yy$   is  a compact  metrizable  topological   space,   and $\overline T: \Xx \to  (\Pp (\Yy), \tau_w)$  is a  continuous  mapping. Then $\overline T$ is a Markov kernel, i.e.  $\overline T: \Xx \to (\Pp (\Yy), \Sigma_w)$ is a measurable  mapping.  Furthermore the  map
	$$(\Gamma_T)_*: (\Ss (\Xx), \tau_w) \to (\Ss (\Xx \times \Yy), \tau_w), \mu \mapsto (\Gamma_T)_* \mu,$$
	is a continuous  map.
\end{proposition}
\begin{proof} Assume the condition  of Proposition \ref{prop:tcontinuous}.  Then $\overline T$ is a Markov kernel, since  $\Yy$  is  a separable metrizable topological space  and   $\Bb (\tau_w) = \Sigma_w$ by \cite[Theorem 2.3]{GH89}. 
	
	Now let  us  prove the  $(\tau_w, \tau_w)$-continuity  of  $(\Gamma_T)_*$. Since  $(\Gamma_T)_*$ is a linear map, it suffices to show that 
	\begin{equation*}
	\tau_w\mhyp
	\lim_{n \to \infty} (\Gamma_T)_*\mu_n = 0\in (\Ss(\Xx \times \Yy), \tau_w) \text{   if }  \tau_w\mhyp\lim_{ n\to \infty} \mu_n  = 0 \in (\Ss (\Xx), \tau_w).
	\end{equation*}
	Thus it  suffices  to show  that  if  $ \tau_w\mhyp\lim_{ n\to \infty} \mu_n  = 0$ then  for any $f \in C_b (\Xx \times \Yy)$ we have
	\begin{equation}\label{eq:lim1}
	\lim_{n \to \infty} \int_{\Xx \times \Yy}   f d  (\Gamma_T)_*\mu_n  = 0.
	\end{equation}
	\begin{lemma}\label{lem:contt1} Assume  the conditions of Proposition \ref{prop:tcontinuous}.  Then  for any $f \in   C_b (\Xx \times \Yy)$ the function  $F_{T, f}: \Xx \to \R$
		$$  F_{T, f} (x) : = \int _\Yy f (x, y)d \overline T (y|x),$$
		belongs  to $C_b (\Xx)$.
	\end{lemma}
	\begin{proof} Since  $f$ is bounded,   $F_{T, f}: \Xx \to \R$ is bounded.   
		To  show that  $F_{T, f}$ is continuous, it suffices  to show  that
		\begin{equation}\label{eq:lim2}
		\lim_{x' \to x}  F_{T, f} (x') =  F_{T, f} (x).
		\end{equation}
		We  write
		$$  F_{T, f} (x') -  F_{T, f} (x) = \int _{\Yy} \big  (f (x, \cdot ) - f(x',\ \cdot)\big)\, d T (x) + \int _\Yy f (x', \cdot )d (T(x) - T(x')).$$
		Since  $\Yy$ is compact, 
		$$\lim_{x' \to x} \|  f(x', \cdot ) - f(x, \cdot )\|_\infty = 0.$$
		It follows  that
		$$\lim_{x' \to x}\int _{\Yy} \big  (f (x, \cdot ) - f(x', \cdot) \big)\, d T (x) = 0.$$
		Since  $T: \Xx \to (\Pp (\Yy), \tau_w)$  is continuous
		$$\lim_{x' \to x}\int _\Yy f (x', \cdot )d (T(x) - T(x'))  = 0,$$
		Thus \eqref{eq:lim2} holds. This completes  the proof  of Lemma \ref{lem:contt1}.
	\end{proof}
	{\it Completion of   the  proof of Proposition \ref{prop:tcontinuous}}. Let  $f \in C_b (\Xx \times \Yy)$. Then 
	$$\int _{\Xx \times \Yy} f d(\Gamma_T)_* \mu_n = \int_\Xx F_{T, f} d\mu_n.$$
	By  Lemma \ref{lem:contt1}, $F_{T, f} \in C_b (\Xx)$.  Hence
	Equation \eqref{eq:lim1} holds, if  $\lim_{n \to \infty} \mu_n = 0 \in (\Ss (\Xx), \tau_w)$. This completes  the  proof of Proposition \ref{prop:tcontinuous}.
\end{proof}

Next  we present a concept  of  almost surely  equality  of   two
bounded s-probabilistic morphisms, which is consistent   with the  concept of
almost equality in Markov  category proposed  by Fritz \cite[Example 13.3, \S 13]{Fritz2019}.

\begin{definition}\label{def:asep} Let $\mu \in \Pp(\Xx)$.   Two bounded measurable mappings $T, T':\Xx\to  \Ss(\Yy)$  will be called  {\it equal   $\mu$-a.e.} (with the shorthand notation $ T =  T' $    $\mu$-a.e.), if  for  any $B \in \Sigma_\Yy$
	$$\mu\{x\in \Xx: T(x)(B) \not =  T'(x)(B)\} = 0.$$
\end{definition}

\begin{lemma}\label{lem:almostsurely} (1)  Assume that  $\mu_\Xx \in \Pp (\Xx)$ and  $T_1, T_2 \in \mathbf{sbProbm} (\Xx, \Yy)$. Then $(\Gamma_{T_1})_* \mu_\Xx = (\Gamma_{T_2})_* \mu_\Xx$ if and only if $\overline T_1 = \overline T_2$ $\mu_\Xx$-a.e.
	
(2) Assume that  $\mu_\Xx \in \Ss (\Xx)$ and  $T_1, T_2 \in \mathbf{Probm} (\Xx, \Yy)$. Then\\
 $(\Gamma_{T_1})_* \mu_\Xx = (\Gamma_{T_2})_* \mu_\Xx$ if and only if $\overline T_1 = \overline T_2$ $|\mu_\Xx|$-a.e.

(3) We have $\Gamma_{T_1 + T_2} = \Gamma _{T_1} + \Gamma_{T_2}$.

(4) For any  $T \in \mathbf{sbProbm}(\Xx, \Yy)$  we  have
\begin{equation}\label{eq:tgt}
T = \Pi_\Yy \circ  \Gamma_T.
\end{equation}

\end{lemma}
\begin{proof}  (1)  Assume the condition  of  Lemma \ref{lem:almostsurely}. To prove  the first assertion of  Lemma \ref{lem:almostsurely} it suffices  to show that  $(\Gamma_T)_*\mu_\Xx = 0$ if and only if $\overline T = 0$ $\mu_\Xx$-a.e. 

The  `iff" assertion follows immediately from \eqref{eq:markov1S}, so let us
prove   the  ``only if" assertion.
Since for  any $B \in\Sigma_\Yy$ the function  $\overline T(\cdot)(B): \Xx \to \R$ is measurable,  $B^+_\Xx: = \{ x\in \Xx| \overline T (x) (B) >0\}$ is  a measurable  subset  of $\Xx$.  By the   assumption, we  have
$$(\Gamma _T)_* \mu_\Xx (B^+_\Xx\times B) =\int_{B^+_\Xx} \overline T(x) (B)\, d\mu_\Xx (x) = 0.$$
Since  $\overline T(x) (B) > 0 $ for all $x \in B^+_\Xx$, it follows   that  $\mu_\Xx (B^+_\Xx) = 0$. Similarly  we prove  $\mu_\Xx (B_\Xx^-) = 0$,
where  $B^-_\Xx: = \{ x\in \Xx| \overline T (x) (B) <0\}$. This proves 
the first assertion.

(2) The second assertion of Lemma \ref{lem:almostsurely}  follows  from the first one, taking into account the Jordan-Hahn decomposition  of $\mu_\Xx$ and Remark  \ref{rem:push} (1).

(3)  The  third assertion  of Lemma \ref{lem:almostsurely}  is straightforward.

(4)  We  verify immediately from   Definition \ref{def:graph} (2)  that
\begin{equation}\label{eq:compost}
\overline T = (\Pi_\Yy)_*\circ \overline {\Gamma_T}.
\end{equation}
Combining  \eqref{eq:compost}  with \eqref{eq:composti},  we  complete  the  proof  of  the last assertion of Lemma \ref{lem:almostsurely}.
\end{proof}
 
Lemma \ref{lem:almostsurely}  motivates  the   following notation.
Given $\mu \in \Mm (\Xx)$ we  denote by $\mathbf{Meas} (\Xx,\Ss( \Yy); \mu)$   the  quotient  space
$\mathbf{Meas} (\Xx, \Ss (\Yy))$ under    $\mu$-a.e.  equality and by  $[T]_{\mu}$,  $\mathbf{Meas} (\Xx, \Pp (\Yy); \mu)$ the image  of   $T \in \mathbf{Meas} (\Xx,\Ss( \Yy))$  and of  $\mathbf{Meas}(\Xx, \Pp (\Yy))$, respectively,  in the  quotient space.

\subsection{A characterization of  regular  conditional  probability measures}

For  $p \ge 1$ and  a $\sigma$-finite measure $\mu$ on $\Xx$ we set  \cite[\S 3.3]{Le2022}, cf. \cite[(3.51), p. 144]{AJLS17}
$$\Ss^{1/p}  (\Xx, \mu): = \{\nu \in \Ss (\Xx)|\, \nu \ll \mu  \text{ and }  \frac{d\nu}{d\mu} \in L^ p (\mu) \}. $$
The  natural  identification  
$\Ss^{1/p} (\Xx, \mu) = L^ p (\mu)$ defines  a  $p$-norm on $\Ss^{1/p}  (\Xx, \mu)$ by setting
$$\| f \mu\| _p  = \| f \| _{L^p (\mu)}.$$
Then  $\Ss^{1/p} (\Xx, \mu)$ endowed  with  the $p$-norm is a Banach space, which we denote by  $\Ss^{1/p} (\Xx, \mu)_p$. For   $\mu_1 \ll \mu_2$     the linear inclusion 
\begin{equation}\label{eq:incp}
\Ss ^{1/p}  (\Xx, \mu_1)\to \Ss^{1/p}  (\Xx, \mu_2), \,   f \mu_1  \mapsto  f (\frac{d \mu_1}{d\mu_2})^{1/p}\mu_2 \in \Ss^{1/p} (\Xx, \mu_2)
\end{equation}
preserves the $p$-norm.

For   a  $\sigma$-finite measure $\mu$ on $\Xx$  we set $\Pp^{1/p} (\mu): = \Pp (\Xx) \cap \Ss^{1/p} (\Xx, \mu)$.

\begin{theorem}\label{thm:marginal}  (1) A bounded measurable mapping  $\overline T : \Xx \to \Ss( \Yy)$ is a regular conditional  probability measure for $\mu\in \Pp(\Xx \times \Yy)$   with  respect  to the    projection  $\Pi_\Xx$ if  and only if
\begin{equation}\label{eq:graph}
(\Gamma_T)_* \mu_\Xx =\mu.
\end{equation}

(2)  If  $ \overline  T,\,  \overline T':\Xx \to \Pp (\Yy)$ are  regular conditional  probability measures for $\mu\in \Pp(\Xx \times \Yy)$   with  respect  to the    projection  $\Pi_\Xx$  then  $  \overline  T =  \overline T'$ $\mu_\Xx$-a.e.

(3)  Assume that $T \in \mathbf{Probm}(\Xx, \Yy)$ and $\rho_0$ is a $\sigma$-finite measure  on $\Yy$ such that  $T(x)\ll \rho_0$ for all $x \in \Xx$. Let $\nu_0 $ be a  $\sigma$-finite measure  on $\Xx$ and  $\nu \in  \Pp^{1/p} (\Xx,\nu_0)$  for some $p \ge 1$.   Then  $ (\Gamma_T)_ *  \nu \in \Ss^{1/p} (\Xx \times \Yy, \nu_0\rho_0)$. 
\end{theorem}

\begin{proof}     (1)  Assume that $\mu \in \Pp (\Xx \times \Yy)$,   $T \in \mathbf{sProbm} (\Xx, \Yy)$ and \eqref{eq:graph} holds. 	 Then   by  \eqref{eq:compost},   $T$ must   be   a probabilistic morphism.  Now  recall  that  for   any $A \in \Sigma_\Xx$, $B \in  \Sigma_\Yy$ we have
\begin{equation*} 
	 (\Gamma_T)_*\mu_\Xx (A \times  B ) = \int _\Xx \overline{\Gamma_T} (x) (A \times  B)\, d\mu_\Xx (x) = \int_A  \overline{T} (x) (B)\, d\mu _\Xx(x),
\end{equation*} 

This  implies  that  $\overline T$ is  a   regular   conditional  measure  for $\mu$ with respect to the projection $\Pi_\Xx$. This proves the  ``if" assertion.

Conversely,  if  $\overline T: \Xx \to  \Ss (\Yy)$  is a  regular   conditional probability measure for  $\mu\in \Pp (\Xx \times \Yy)$ then $\overline T(\Xx) \subset \Pp (\Xx)$  and \eqref{eq:jointmeasure}  holds.  Taking into account \eqref{eq:markov1S}, we conclude  that  \eqref{eq:graph}  holds. This completes  the  proof of   assertion (1)
of  Theorem \ref{thm:marginal}.

\

(2)     The second assertion  follows  from the first one and Lemma \ref{lem:almostsurely}(1).
\

(3)    Let  $T\in \mathbf{Probm}(\Xx, \Yy)$ and  $T(x)\ll  \rho_0$ for all $x \in \Xx$. We write $T(x) = g(\cdot |x)\rho_0$ where $g (\cdot|x)$ is the density of $T(x)$ with respect to $\rho_0$.  Taking into account  the  disintegration  formula  for    $f \in \Ff_s (\Xx \times \Yy)$,
\begin{equation}\label{eq:graphh}
\int_{\Xx \times \Yy} f(x,y) \, d (\Gamma_{T})_* \nu (x,y) = \int_{\Xx \times \Yy}  f(x, y) g(y|x)d\rho_0(y) d \nu(x), 
\end{equation}
it  follows  that $(\Gamma_{T})_* \nu=  g\nu \rho_0$. Hence
\begin{equation}\label{eq:dominatedt}
(\Gamma_{T})_* \nu \ll \nu_0\rho_0.
\end{equation}
Since  $\nu \in \Pp^ {1/p} (\Xx, \nu)$, 
by \cite[Corollary 5.1, p. 260]{AJLS17}, which states that  if $T \in \Probm (\Xx, \Yy)$ and $\mu \in \Mm(\Xx)$ then  $T_* (S^{1/p} (\Xx, \mu)) \subset S^{1/p} (\Yy, T_* \mu)$,  we have
\begin{equation}\label{eq:corollary51}
(\Gamma_T)_* \nu \in \Ss^{1/p} (\Xx \times \Yy, (\Gamma_T)_*\nu).
\end{equation}
Taking into account \eqref{eq:graphh},  we obtain  immediately the last assertion of  Theorem \ref{thm:marginal}.
\end{proof}

\begin{remark}\label{rem:condtional} In the setting of Markov categories, Equation \eqref{eq:graph}  for   $T \in \mathbf{Probm}(\Xx, \Yy)$ is exactly the definition of conditional  \cite[\S 3]{CJ19}.
\end{remark}

\section{Generative  models  of supervised  learning}\label{sec:gen}

In this section, we introduce the concept of a generative model of supervised learning  (Definition \ref{def:superlp}), which encompasses various models used for multi-classification, regression tasks, conditional probability estimation  and probability measure estimation. We also present the concept of a correct loss function (Definition \ref{def:correctloss}), which includes many commonly used loss functions in classical statistics and statistical learning theory (Examples \ref{ex:densitycorrect}).  Additionally, we discuss important properties of inner and outer measures. Finally, we explore general statistical learning models (Definition \ref{def:unifiedl}), their learnability (Definition \ref{def:genbi}, Remark \ref{rem:genbi}), and provide a sufficient condition for the learnability of a statistical learning model (Theorem \ref{thm:uniform})  by introducing the concept of a 
  $C$-ERM algorithm (Definition \ref{def:aserm}, Remark \ref{rem:cerm}).  

 \subsection{Generative models of supervised  learning}\label{subs:gen}
\begin{definition}\label{def:superlp}  {\it A generative model of supervised learning} is  defined as a   quintuple  $(\Xx, \Yy, \Hh, R, \Pp_{\Xx \times \Yy})$,  where $\Xx$   and $\Yy$ are  measurable spaces,  $\Hh$  is a  family  of measurable  mappings $\bar h: \Xx \to \Pp (\Yy)$, $\Pp_{\Xx \times \Yy}\subset \Pp (\Xx \times \Yy)$  contains  all possible   probability measures  that govern the   distributions of  labeled  pairs $(x, y)$,  and   $R : \Hh \times \big (\Pp_{\Xx  \times \Yy}\cup \Pp_{emp}  (\Xx \times \Yy)\big) \to \R  \cup \{+\infty\}$ is  a risk/loss  function. It should  satisfy the condition  that    for  any $\mu \in \Pp_{\Xx \times \Yy}\cup \Pp_{emp} (\Xx \times \Yy)$  we have $\inf_{h \in \Hh}R(h, \mu) \not = \pm \infty$. If
$R (h, \mu) = \E_\mu  (L  (h))$ where  $L: \Xx  \times \Yy \times \Hh \to \R \{+\infty\}$ is an  instantaneous loss function  then we  can represent   the model as  $(\Xx, \Yy, \Hh, L, \Pp_{\Xx \times \Yy})$. 
{\it  A  classical generative   model  of supervised  learning}   is a   special case  of a generative model   $(\Xx,  \Yy, \Hh, L, \Pp_{\Xx\times \Yy})$, where  $\Hh$  is a  family  of dominated    regular  conditional  probability measures, i.e.,  there exists   a  $\sigma$-finite measure $\rho_0$ on $\Yy$ such that  $\bar h(x)\ll \rho_0$ for  all $ x\in \Xx$, $h \in \Hh$.
\end{definition}

\begin{remark}\label{rem:superl} 
(1)  If $\Xx  $ consists  of a single point $\{pt\}$,   then  $\Pp_{\{pt\}\times \Yy}$ can be identified  with  a    statistical model  $\Pp_\Yy \subset \Pp (\Yy)$. In this case it is natural to assume  that  the  set $\{h(\{pt\})|h \in \Hh\}$ is identified  with $\Pp_\Yy$  and our  model  $(\{pt\}, \Yy, \Hh \cong \Pp_\Yy,  R, \Pp_\Yy)$  of supervised learning   is a model  of density estimation, or more precisely,  a model of  probability  measure  estimation in unsupervised learning.

(2) If  $R$ is generated  by
an instantaneous  loss function  $L:  \Xx  \times \Yy \times \Hh \to \R \cup \{+\infty\}$, then  for any $S = (x_1, y_1, \ldots , x_n, y_n) \in (\Xx\times \Yy)^n$ we   have $R^L _{\mu_S} (h) = \frac{1}{n}\sum_{i =1} ^n L (x_i, y_i, h)$.

(3) Generative models    of supervised learning  encompass discriminative   models of supervised learning where the hypothesis  space $\Hh$ consists  of measurable  mappings. For classical multi-classification and regression tasks, as well as models for learning conditional density functions and conditional probability considered by Vapnik (see Examples \ref{ex:densitycorrect}(3 $\&$ 4)    below), generative models provide a comprehensive framework. 
\end{remark}

\begin{definition}\label{def:correctloss} A  loss  function
	$ R: \Hh \times \big (\Pp_{\Xx \times \Yy}\cup\Pp_{emp} (\Xx \times \Yy)\big) \to \R  \cup \{ + \infty\}$ will  be  called  {\it $\Pp_{\Xx\times \Yy}$-correct},
	if there exists a  set $\widetilde \Hh \subset \mathbf{Meas}  (\Xx, \Pp (\Yy))$ such  that  the following  three conditions   hold
	
	(1) $\Hh \subset \widetilde  \Hh$.
	
	(2)  For any $\mu \in \Pp_{\Xx \times \Yy}$ there exists
	$h \in \widetilde \Hh$ such  that $h$ is  a regular  conditional measure  for $\mu$  relative to the  projection $\Pi_\Xx$, i.e.. $[h]_{\mu_\Xx}= [\mu_{\Yy|\Xx}]$.
	
	(3) $R$ is  the restriction of a  loss function $\widetilde  R: \widetilde \Hh \times \big (\Pp_{\Xx \times \Yy}\cup\Pp_{emp} (\Xx \times \Yy)\big )
	\to \R  \cup \{ + \infty\}$   such that  for  any  $\mu \in \Pp_{\Xx \times \Yy}$    
	$$\arg \min _{h \in \widetilde \Hh} \widetilde R (h, \mu) = \{ h \in \widetilde \Hh|\, [h]_{\mu_\Xx}= [\mu_{\Yy|\Xx}]\}.$$

A  loss  function
$ R: \Hh \times \big (\Pp_{\Xx \times \Yy}\cup\Pp_{emp} (\Xx \times \Yy)\big ) \to \R  \cup \{ + \infty\}$ will  be  called {\it correct}, if  $R$  is   the restriction of   a  $\Pp (\Xx \times \Yy)$-correct  loss function  $\widetilde R: \Hh	\times  \Pp (\Xx \times \Yy) \to \R  \cup \{ + \infty\}$.	
\end{definition}

Given  a  $\sigma$-finite measure $\mu$ on  $\Xx$ we denote  by  $\Ll^1 (\Xx, \mu)$ the set  of  all  $\mu$-integrable functions on $\Xx$ and by $L^1 (\Xx,\mu)$  the set  of    $\stackrel{\mu}{\sim}$ equivalence  classes  in $\Ll^1 (\Xx, \mu)$ where $f\stackrel{\mu}{\sim} g$  iff  $f-g = 0$ $\mu$-a.e.

\begin{examples}\label{ex:densitycorrect}  (1) Let $\Xx = \{ pt \}$, $\Yy = \R$, and $\Hh  = \Pp (\Yy, dy)$  the set   of  all probability measures $ f dy$  where  $dy$ is the Lebesgue  measure  on $\Yy =\R$  and $f \in \Ll^1 (\R, dy)$.   Let  $\Pp_\Yy = \Pp (\Yy, dy)$.  We define
the minus log-likehood  instantaneous loss  function $L: \Hh \times \R \to \R$ by setting 
$$L (f,y)= -\log f ( y),$$    
which leads to  the  loss  function $R^L: \Hh \times \Pp(\R, dy)\to \R$,  and  hence,  for  any $\mu \in \Pp (\R, dy)$,  to   the expected  loss  function
$$R^L_{\mu}:  \Hh\to \R,  f dy \mapsto  -\int _\R \log f  d\mu.$$
Now we write $\mu= p dy$  where $p \in \Ll^1 (\R, dy)$.
By the Bretagnolle-Huber  inequality, given  in the  following  form  \cite[(1.11), p.30]{Vapnik1998}:
	$$\int_\R  | f(y) - p(y)| dy  \le 2  \sqrt{ 1 - \exp ( R^L_{\mu}(\mu) - R^L_{\mu}(fdy))}, $$
the   function  $R^L: \Hh \times \Pp(\R, dy) \to \R$ is a $\Pp(\R,dy)$-correct   loss  function.

\	
	
 (2)  Let $(\Xx, \Yy, \Hh, \Pp_{\Xx \times \Yy})$ be an arbitrary  {\it geometric  model  of  superviseed learning},  i.ee  a quadruple   underlying a   supervised  learning model $(\Xx, \Yy, \Hh, L, \Pp_{\Xx\times \Yy})$.  We  defines a loss  function: $R^{(k)}: \Hh \times \Pp_{\Xx \times \Yy}  \to \R_{\ge 0}$ as  follows
\begin{equation}\label{eq:loss1}
R^{(k)} (h, \mu): = \|(\Gamma_{ h})_*  \mu_\Xx - \mu \|^k _{TV}.
\end{equation}
By Theorem \ref{thm:marginal} (1), $R$  is a correct loss  function for any $k \in \N^+$.

Now assume  that $\mathfrak M: \Pp(\Xx \times \Yy) \to  E$ is an  embedding from   $\Pp  (\Xx \times \Yy)$ to   a metric vector  space $(E, d)$.  By  Theorem \ref{thm:marginal}(1),  the  function  
\begin{equation}\label{eq:losstwd}
R^d: \Hh \times \Pp_{\Xx \times \Yy} \to \R_{\ge 0},\,  (h, \mu)\mapsto  d\Big(\mathfrak M\big ((\Gamma_{ h})_*  \mu_\Xx\big), \mathfrak M(\mu)\Big ),
\end{equation}
 is a correct  loss  function, since $\mathfrak M$  is injective.

(3) Assume  that  $\Yy = \{ \om_1, \ldots , \om_n\}$ is a  finite  sample space consisting  of  $n$  elements,  and  $\Xx \subset \R^n$. Let $h \in \mathbf{Probm} (\Xx, \Yy)$.  In this case the  equation  $(\Gamma_h)_*\mu_\Xx = \mu$ is equivalent  to the following   equation:
\begin{equation}\label{eq:vapnik120}
\int_{-\infty} ^a \overline h (x) (\om_i) dF_{\mu_\Xx}(x) =  F_\mu (\om_i, a) \text{  for } i\in [1,n] \text{  and }   a\in \R^n
\end{equation}
where  $F_{\mu_\Xx}$ and   $F_{\mu}$ are  cummulative distribution function of $\mu_\Xx$ and  of $\mu$, respectively,  and the integral in the LHS of \eqref{eq:vapnik120} is the Lebesgue-Stieltjets integral. Note  that the LHS and RHS of  \eqref{eq:vapnik120}  is well-defined, if   $\mu$, and hence $\mu_\Xx$, are  empirical  measures. That is   Vapnik's equation  (1.20) in \cite[p. 36]{Vapnik1998} for   conditional  probability.
Assuming further that $F_\mu$ belongs to a metric  space  $(E, d)$,  we obtain  the following   correct loss  function:  
\begin{equation}\label{eq:lossd1}
	R:   \mathbf{Probm} (\Xx, \Yy)  \times \Pp(\Xx \times \Yy)  \to  \R_{\ge 0}, R(h, \mu) =  d( F_{ (\Gamma _{ h})_*\mu_\Xx}, F_\mu).
\end{equation}
 
(4) Assume  that $\Xx = \R^m$ and   $\Yy  = \R^n$. Let  $dy$ denote the Lebesgue  measure on $\R^n$  and
 $\Hh \subset  \mathbf{Meas}(\R^m, \Pp(\R^n, dy))$. For  $h \in \Hh$  we write
 	$$h (x): = \bar h (\cdot |x ) dy \in  \Pp (\R^n),$$
 where the density  function  $\bar h (\cdot |x)$ belongs  to $ \Ll^1  (\R^n, dy)$ for all $x \in \R^m$. By Tonelli's  theorem,  $\bar h \in \Ll ^1 (\R^m\times \R^n, dx dy)$,   where
 $dx$ is the  Lebesgue measure on $\R^m$s. Let $\mu \in \Pp(\R^m \times \R^n, dx dy)$.
 By \eqref{eq:corollary51}, $\mu_\Xx \in \Pp (\R^m, dx)$. 
 By Theorem \ref{thm:marginal}(5),  we have 
 \begin{equation}\label{eq:marginalr}(\Gamma_{h})_* (\mu_\Xx) \in \Ss (\R^m \times \R^n, dx dy ).
 \end{equation}

Now  we  shall   rewrite Equation \eqref{eq:graph},  using  cumulative   distribution functions as in Example \ref{ex:densitycorrect}(3), so   that the loss  function $R_\mu:\Hh \to \R$ is  also well-defined  for any $\mu \in \Pp_{emp} (\Xx \times \Yy)$.   By  Theorem \ref{thm:marginal} (1),   $\bar h$  is a    conditional density function for
 $\mu \in \ \Pp (\R^m\times \R^n)$, if  and only  if   for  any $(a, b)\in \R^m \times \R^n$ we have
 
 \begin{equation}\label{eq:graph1}
 (\Gamma_h)_* \mu_\Xx ((-\infty, a)\times (-\infty, b)) = F_\mu (a, b).
 \end{equation}
 Equivalently,
\begin{equation}\label{eq:vapnikp}
\int_{-\infty} ^b \int _{-\infty}^a \bar h d F_{\mu_\Xx} d  y =  F_\mu(a,b).
\end{equation}
This is Equation (1.21) in \cite[p. 37]{Vapnik1998}.  Assuming further that $F_\mu$ belongs to a metric  space  $(E, d)$,  we obtain  the following   $ \Pp(\R^m  \times\R^n, dxdy)$-correct loss  function, cf. \eqref{eq:lossd1},  
\begin{align} R:   \mathbf{Probm} (\Xx, \Yy)  \times \big( \Pp(\R^m  \times\R^n, dxdy)\cup \Pp_{emp} (\R^m \times \R^n)\big)  \to  \R_{\ge 0},\nonumber\\
 R(h, \mu) =  d( F_{ (\Gamma _{ h})_*\mu_\Xx}, F_\mu).\label{eq:lossd2}
 \end{align}
 
 (5) Let $d: \Pp (\Xx \times \Yy) \times \Pp (\Xx, \Yy)\to \R_{\ge 0} \cup \{ \infty\}$ be an arbitrary divergence  on $\Pp (\Xx \times \Yy)$. For instance,   $d$ is the   Kullback-Leibler divergence,    and if  $\Xx $ and $\Yy$ are  separable, metrizable  topological spaces   then   $d$    can be any   metric,  which generates
 the weak*-topology $\tau_w$  on $\Pp(\Xx \times \Yy)$. Then 
 \begin{equation}\label{eq:lossd3}
 R^d:   \mathbf{Probm} (\Xx, \Yy)  \times \Pp(\Xx \times \Yy)  \to  \R_{\ge 0}, R(h, \mu) =  d( (\Gamma _{ h})_*\mu_\Xx, \mu)
 \end{equation}
 is a   correct loss  function. This is a  general form of correct loss  functions  defined in Equations  \ref{eq:loss1}, \ref{eq:losstwd}, \ref{eq:lossd1}, \eqref{eq:lossd2}.
\end{examples}

\subsection{Inner and outer measure: preliminaries}\label{subs:ino} In this subsection we collect  necessary properties  of inner and  outer measure  which we shall need in this  article.

$\bullet$  Given a  (nonnegative) measure   $\mu$ on $\Xx$, we  denote by $\mu^*$  the outer measure defined  by $\mu$  and  by $\mu_*$  the inner  measure  defined  by $\mu$, i.e.  for any $S \subset \Xx$  we have \cite[p. 16, 56, 57, vol. 1]{Bogachev2007}:
\begin{eqnarray*}
\mu^* (S) = \inf \{ \mu (A):\,  S \subset A  \, , A \in \Sigma_\Xx\}\\
\mu_*(S) = \sup \{ \mu (A):\,   S \supset A \, ,  A \in \Sigma_\Xx\}.
\end{eqnarray*}
Then  we have \cite[p. 23, vol. 1]{Bogachev2007}:
\begin{equation}\label{eq:complement}
\mu^*(S) + \mu_* (\Xx\setminus S) = \mu (\Xx).
\end{equation}
$\bullet$ Monotonicity of inner and outer  measure \cite[p. 17, p. 70, vol. 1]{Bogachev2007}:  
\begin{equation*}
\mu_* (S_1)\le \mu_* (S_2) \text{ and } \mu^* (S_1)\le \mu^* (S_2) 
\text{ if } S_1 \subset  S_2.
\end{equation*}
$\bullet$ Countable subadditivity  of  outer measure \cite[(1.5.1), p. 17, vol. 1]{Bogachev2007}:
\begin{equation*}
\mu^* (\cup_{n=1}^\infty X_n)\le \sum_{n=1}^\infty \mu^* (S_n).
\end{equation*}

\begin{proposition}[Continuity from below  of outer measure]\label{prop:contb}\cite[Proposition 1.5.12, p. 23, vol. 1]{Bogachev2007}. Let $\mu$ be   a (nonnegative)  measure on $\Xx$. Suppose that   the set $S_n $ are such that  $S_n \subset S_{n+1}$  for all $n \in \N$. Then one has
\begin{equation}\label{eq:contb}
\mu^* \Big(\bigcup_{n =1} ^\infty S_n\Big ) = \lim_{n \to \infty} \mu^* (S_n).
\end{equation}
\end{proposition}
$\bullet$ For every  decreasing  sequence  $S_1 \supset S_2 \ldots \supset S_n$  such that $\mu_* (S_1) < \infty$ we have \cite[p. 70, vol. 1]{Bogachev2007}
\begin{equation}\label{eq:contu}
\mu_* \Big(\bigcap_{n =1} ^\infty S_n\Big) = \lim_{n \to \infty} \mu_* (S_n).
\end{equation}

\begin{proposition}\label{prop:intero}\cite[Proposition 1.5.11, p. 22]{Bogachev2007}  If $A \in \Sigma _\Xx$  then
\begin{equation}\label{eq:intero}
\mu^*(S \cap A)  + \mu  ^* (S \setminus A) = \mu ^* (S) \text{ for all } S \subset \Xx.
\end{equation}
\end{proposition}

\begin{corollary}\label{cor:interi}  For any $S \subset \Xx$ and $A \in \Sigma_\Xx$ we have
\begin{equation}\label{eq:interi}
\mu_* (S \cap A) \ge \mu_* (S) -\mu (\Xx \setminus A).
\end{equation}
\end{corollary}
\begin{proof}
Using  \eqref{eq:complement}, the validity  of  inequality  \eqref{eq:interi} is equivalent  to  validity  of  the following
\begin{eqnarray*}
\mu^* (\Xx\setminus (S\cap A)) \le  \mu^* (\Xx\setminus S) + \mu (\Xx \setminus A)\\
\LLR  \mu ^* \big ((\Xx \setminus  S) \cup  (\Xx \setminus A)\big)\le  \mu^* (\Xx\setminus S) + \mu (\Xx\setminus A),
\end{eqnarray*}
which holds   because of the subadditivity of  outer measure.
\end{proof}
\subsection{Consistency of a learning  algorithm}\label{subs:aerm}

The concept of a  generative  model of supervised  learning is a particular case  of the  concept  of a    statistical learning model  defined below.

\begin{definition}\label{def:unifiedl}  A   {\it statistical  learning  model}  consists  of  a quadruple  $(\Zz, \Hh, R, \Pp_\Zz)$  where  $\Zz$ is  a measurable space,  $\Hh$  is a decision space  containing all  possible  decisions we   have to find  based on a sequence of  observables  $(z_1, \ldots,   z_n) \in \Zz^n$, $\Pp_\Zz \subset \Pp(\Zz)$ is a     statistical model  that  contains all  possible  probability measures  $\mu$  on  $\Zz$ that govern   the distribution of an  i.i.d.   sample of data  $z_1, \ldots, z_n\in \Zz$, $R:  \Hh \times \big (\Pp_\Zz\cup\Pp_{emp}(\Zz)\big) \to  \R  \cup \{+ \infty\}$ is a loss  function  such that $\inf_{h \in \Hh} R (h, \mu) \not = \pm \infty$ for any $\mu \in \Pp_{\Xx \times \Yy} \cup \Pp_{emp} (\Xx \times \Yy)$. 
We say that  $R$ is generated    by  an    instantaneous loss function  $L:\Zz \times  \Hh\to \R$  if
	$R (h, \mu) = R^L_\mu (h) : = \E_\mu L(z, h)$.
	In this case  we shall write  $R^L$ instead  of $R$.
{\it A learning algorithm} is  a map
	$$A: \bigcup_{n =1}^\infty   \Zz^n \to  \Hh.$$
\end{definition}

Given a   statistical learning model $(\Zz, \Hh, R, \Pp_{\Zz})$   and $\mu \in \Pp_{\Zz}$, we set
\begin{equation}\label{eq:apperror}
R_{\mu, \Hh} : = \inf_{h \in \Hh}R_\mu  (h).
\end{equation}
For $h \in \Hh$  we denote its  {\it estimation  error}  as  follows:
\begin{equation}\label{eq:aprerr}
\Ee_{\Hh, R, \mu} (h): = R_\mu  (h) - R_{\mu, \Hh}.
\end{equation}

If $R = R^L$ we shall  write  $\Ee_{\Hh, L, \mu}$ instead  of 
$\Ee_{\Hh, R^L, \mu}$.

\begin{definition}\label{def:genbi}  A  statistical   learning  model
$(\Zz, \Hh, R, \Pp_\Zz)$  will be  said  to have  a {\it generalization  ability} or it  will be   called  {\it learnable}, if  there   exists  a  uniformly  consistent learning algorithm  
$$ A: \bigcup_{n=1}^\infty \Zz^n \to  \Hh, $$
i.e.  for any  $(\eps, \delta)\in (0,1)^2$  there  exists   a number $m_A  (\eps, \delta)$ such  that   for any $m \ge m_A (\eps, \delta)$ and any $\mu \in \Pp_\Zz$ we have
\begin{equation}\label{eq:pac}
(\mu^n)_*  \{ S \in \Zz^n, \Ee_{\Hh, R, \mu} (A(S)) \le \eps\} \ge 1- \delta.
\end{equation}
In this  case  $A$  will be called {\it uniformly consistent}.
\end{definition}

\begin{remark}\label{rem:genbi}  The   current  definition of     a uniform consistency, also called  a generalizability, of  a learning algorithm $A$ in literature  is   almost  identical  to our definition  but    the inner measure  $(\mu^n)_* $ is replaced  by   $\mu^n$. Equivalently, in Definition \ref{def:genbi}, we relax the  convergence in probability in the classical  requirement   by the convergence   in outer probability  for  the sequence  of functions $\Ee_{\Hh, R, \mu} \circ A$. Note that  convergence is  outer probability has been  used  in  empirical   processes  \cite[\S. 3.3]{Dudley2014}, \cite[\S 1.9]{VW1996}.   The classical requirement  poses  the following  condition   on  a learning algorithm  $A$: for  any $\mu \in \Pp_\Zz$  and any $n\in \N^+$  the    function  $\Ee_{\Hh, R, \mu} \circ A: \Zz^n \to \R$  is  $\mu ^n$-measurable. Equivalently, for any $\mu \in \Pp_\Zz$  and any $n\in \N^+$ the function $R_\mu\circ A: \Zz^n \to \R$  is  $\mu ^n$-measurable. Many  algorithms  in   statistics  and statistic learning  theory  may not be measurable, see  e.g. \cite{BP1973}.
\end{remark}

Given  a     sequence  of  data $S\in \Zz^n$ we define  the  empirical risk
$$\widehat  R_S: \Hh \to  \R, \, h \mapsto   R_{\mu_S} (h). $$
\begin{definition}\label{def:aserm} Given a   statistical model model  $(\Zz, \Hh,R, \Pp_{\Zz} )$ and a sequence   $C = (c_1\ge  \ldots\ge  c_n\ge  \ldots: \, c_i \ge 0)$,    a learning  algorithm
$$ A: \bigcup _{n \in \N} \Zz^n \to  \Hh $$
will be called  a {\it $C$-empirical  risk minimizing} algorithm, abbreviated  as  $C$-ERM algorithm,  if 	 for  any $n \in \N$  and any $S \in \Zz^n$, we have
$$\widehat R_{S} (A (S))  - \inf _{ h\in  \Hh}R_S (h) \le c_n. $$
If  $c_i = 0$ for   all $i$  we   write $C= 0$,
\end{definition}

Clearly, a $0$-ERM  algorithm is an  ERM  algorithm.

\begin{remark}\label{rem:cerm}  The concept  of a $C$-ERM algorithm is motivated  by  the fact  that, given a  sequence $C  = (c_1\ge c_2 \ge \ldots, : c_i >0)$, a  $C$-ERM algorithm    always  exists   and  an ERM-algorithm (a $0$-ERM algorithm)  may not exist. Furthermore,       with fine  tuning,   a gradient   flow       could   yield  a $C$-ERM algorithm.  Examples  of 
$C$-ERM  are  solutions of  regularized   ERM  with  parameter  $c_n$ for $S \in \Z^n$, i.e.
$$A(S)  \in \arg \min_{h \in \Hh} (\hat R_{S} (h) + c_n W(h))$$
 if    a solution of this equation  exists and  if we  know   that $W(h)\in [0,C]$ for all  $h \in \Hh$ where $C <\infty$.
\end{remark}

\begin{theorem}\label{thm:uniform} Let $(\Zz, \Hh, R, \Pp_\Zz)$ be a statistical learning model. 
 Assume  that    there  exists
	a function $m_{\Hh, R, \Pp_{\Zz}}:  (0, 1) ^2  \to \R_{+}$  such that  for any
	$(\eps, \delta) \in (0,1)^2$, any $ n \ge m_{\Hh, R, \Pp_{\Zz}} (\eps, \delta)$  and  any $\mu \in \Pp_{\Zz}$  we have
	\begin{equation}\label{eq:uniconverge}
	(\mu ^n)_* \{S\in  \Zz^n:\, \sup_{h \in \Hh} |\hat R_S  (h)  - R_\mu (h) | \le \eps  \}\ge 1- \delta.
	\end{equation}
Given any sequence $C =(c_1,\ldots  c_m\ldots |\, c_i \ge 0)$,   and  a $C$-ERM algorithm $A: \cup _{n \in \N^+} \Zz^n  \to \Hh$,   for any $m \ge m_{\Hh, R, \Pp_{\Zz}}(\eps, \delta)$  any $\mu \in \Pp_{\Zz}$  we have
	\begin{equation}\label{eq:uniformerm}
	(\mu ^m)_* \{ S \in  \Zz^m:\, \Ee_{\Hh, L, \mu} (A (S)) \le 2 \eps + c_m \} \ge  1 -\delta.
	\end{equation}
	Consequently,   $A$   is  a uniformly  consistent  algorithm, if  $\lim_{n \to \infty} c_n = 0$.
\end{theorem}
\begin{proof} 
Assume the condition \eqref{eq:uniconverge} of  Theorem \ref{thm:uniform}. Let $\mu \in \Pp_{\Zz}$ and $(\eps, \delta)\in (0,1)$.
	Then for any  $m \ge   m_{\Hh, R,  \Pp_{\Zz}} (\eps, \delta)$  we have
	\begin{equation}\label{eq:2cs1}
	(\mu ^m)_* \{ S\in \Zz^m :\, R_\mu (A(S)) \le \hat R_S  (A(S)) + \eps \} \ge  1-\delta.
	\end{equation}
	Given $\theta > 0$ let $h_\theta \in \Hh$  be   such  that
	\begin{equation}\label{eq:theta}
	R_\mu   (h_\theta) - R_{\mu,\Hh} \le \theta .
	\end{equation}
By \eqref{eq:uniconverge},  for any  $m \ge   m_{\Hh, R, \Pp_{\Zz}}(\eps, \delta)$,  we   have 
	\begin{equation}\label{eq:2cs2}
\mu^m \{ S\in \Zz^m: \, \hat R_S (h_\theta) \le R_\mu (h_\theta)  + \eps\} \ge    1-\delta .
	\end{equation}
	Since $A$  is  a $C$-ERM,    we have
	\begin{equation}\label{eq:2cs3}
	\forall  S \in \Zz^m: \, \hat R_S (A(S)) \le \hat R_S  (h_\theta) + \theta +  c_m .
	\end{equation}
	Taking  into account  \eqref{eq:2cs1}, \eqref{eq:2cs2},  \eqref{eq:2cs3}, we obtain
	\begin{eqnarray}\label{eq:2cs4}
	(\mu^m)_*\{ S \in  \Zz^m : \, R_\mu  (A(S))\le \hat R_S (A(S)) + \eps  \le \hat R_S (h_\theta) + \eps + \theta +c_m \le\nonumber\\
	 R_\mu (h_\theta) + 2\eps + \theta +c_m\} \ge \nonumber\\
\ge (\mu^m)_*  \{ S\in \Zz^m: \sup_{h\in \Hh} | \hat R_S -R_\mu (h)| \le \eps\}\nonumber \\ \ge   1-\delta.
	\end{eqnarray}
Letting $\theta $ go to zero   and taking into account \eqref{eq:contu}, we  obtain \eqref{eq:uniformerm} from \eqref{eq:2cs4}  immediately.
\end{proof}

\begin{definition}\label{def:samplecompl}  The function $m_{\Hh, R,  \Pp_\Zz}: (0,1) ^2 \to \R$  defined  by the requirement  that $m_{\Hh, R, \Pp_\Zz} (\eps, \delta)$ is the  least number   for  which  \eqref{eq:uniconverge} holds for all $\mu \in \Pp (\Zz)$ is called {\it the  sample complexity of  a statistical learning model $(\Zz, \Hh, R, \Pp_\Zz)$}.  If $R  = R^L$  we shall write $m_{\Hh, L,  \Pp_\Zz}$ instead of $m_{\Hh, R,  \Pp_\Zz}$.   If $\Pp_\Zz = \Pp (\Zz)$ then we  shall  use	the shorthand notation  $m_{\Hh,R}$  for $m_{\Hh,R, \Pp(\Zz)}$.
\end{definition}
\begin{remark}\label{rem:csl2}   Theorem \ref{thm:uniform} is a generalization  of Cucker-Smale's result   \cite[Lemma 2]{CS2001}. 
There   are many generalizations  of  \cite[Lemma 2]{CS2001} in  textbooks on machine learning, where    authors  did not consider    inner/outer measures and  take uncountable intersection  of  measurable     sets in \eqref{eq:uniconverge}. 
\end{remark}

\section{Regular conditional  measures   via  kernel mean  embedding}\label{sec:loss}
In this section   we assume  that  $K: \Yy \times \Yy  \to \R$  is  a  measurable  positive  definite symmetric (PDS)  kernel on  a measurable space $\Yy$.  For $y \in \Yy$ let  $K_y$ be the function on $\Yy$ defined by
$$K_y (y') = K(y, y') \text { for } y' \in \Yy.$$
We denote by $\Hh(K)$ the  associated  RKHS \cite{Aronszajn50}, see also \cite{BT2004}, \cite{SC2008}, i.e.,
$$ \Hh (K)  =\overline{{\rm span}}\{ K_y, y \in \Yy\},$$
where the closure  is taken with respect  to the $\Hh (K)$-norm defined  by
$$\la  K_y, K_{y'} \ra_{\Hh (K)} =  K (y, y').$$
Then for any $f \in \Hh (K)$ we have
\begin{equation}\label{eq:inner1}
f (y) = \la f, K_y \ra _{\Hh (K)}.
\end{equation}
First, we  summarize  known results  concerning  kernel  mean embeddings  $\mathfrak M_K: \Ss (\Yy) \to \Hh (K)$ and their direct  consequences,  which we shall need  in    our paper.  Using    results  concerning  probabilistic morphisms obtained  in    section  \ref{sec:unified}, we study  measurability  and continuity of kernel  mean embeddings.
Then   we give a characterization  of   regular conditional probability measure as  a minimizer of  mean square error  (Theorem  \ref{thm:errorRKHS}). We show that  the constructed mean square error is   a  correct loss function  (Corollary  \ref{cor:correct})). Finally, we     present  examples   of   (correct)  loss functions obtained  using  this  characterization, among them  there are   the 0-1  loss  function and the  mean square  error (Examples \ref{ex:kernel}).

\subsection{Kernel mean embeddings: preliminaries} \label{subs:not2}

By the Bochner theorem \cite{Bochner1933},  see also \cite[Theorem 1, p. 133]{Yosida1995},  $\int_\Yy \sqrt{K(y,y)}\, d\mu (y) < \infty$  for  $\mu \in \Pp (\Yy)$  if and only  if the   {\it kernel  mean embedding $\mathfrak M_K (\mu) $ of $\mu$}  via the  Bochner integral is well-defined  \cite{BT2004},  where
\begin{equation}\label{eq:kme1}
\mathfrak M_K (\mu)  =   \int_\Yy  K_y d\mu (y) \in \Hh(K).
\end{equation}
By \eqref{eq:inner1},  if $\mathfrak M_K (\mu) $ is well-defined, for any $f \in \Hh (K)$ we have
\begin{equation}\label{eq:inner2}
\int_\Xx f (x) d\mu (x) = \la \mathfrak M_K (\mu), f \ra_{\Hh (K)}.
\end{equation}

$\bullet$  For a  Banach space $V$, $\tau_s$  denotes the strong  topology on $V$   and   the induced  topology on its  subsets,  and  $\tau_W$  denotes the weak topology on $V$  and   the induced  topology on its  subsets.

$\bullet$ For  a locally compact  Hausdorff topological    space $\Yy$  we denote  by $C_0 (\Yy)$  the set of all continuous functions   which vanish  at infinity, i.e., for  any $\eps > 0$ the set $\{ x\in \Xx:  f(x)\le  \eps\}$ is compact.

\begin{proposition}\label{prop:srithm32}\cite[Theorem 3.2]{Sriperumbudur16}  Assume that  $\Yy$ is a Polish space that is locally compact Hausdorff.
Assume that there  exists   a continuous   bounded  kernel  $K: \Yy \times \Yy \to \R$  such that the following  conditions  hold.\\
(1)	 $K_y \in C_0 (\Yy)$ for all $ y\in \Yy$.\\
(2)  For  any $\mu \in \Ss (\Yy)$ we have
\begin{equation}\label{eq:inj}
\int_\Yy\int_\Yy K (y, y') d\mu (y)d\mu (y') > 0  \text{ for all } \mu \in \Ss (\Yy)\setminus  \{0\}.
\end{equation}
In other  words,  $\mathfrak M_K: \Ss (\Yy) \to \Hh (K)$ is injective.\\
(3) $K$   satisfies  the following   property  (P)
$$\forall y \in \Yy, \,  \forall  \eps  > 0, \, \exists  \text{ open } U_{y, \eps}\subset \Yy  \text{ such that } \| K_{y} -K_{y'}\|_{\Hh (K)} < \eps,   \,\forall  y' \in U_{y, \eps}.   $$
Then the  induced    topology  $\mathfrak M_K ^* (\tau_s)$  is the  weak*  topology $\tau_w$ on $\Pp (\Yy)$.  In particular, the kernel  mean  embedding $\mathfrak M_ K: (\Pp (\Yy), \tau_w) \to (\Hh (\Yy),\tau_s), \, \mu \mapsto \mathfrak M_K (\mu),$ is  continuous.
\end{proposition}

\begin{proposition}\label{prop:lmst2015}\cite[Theorem 1]{LMST2015}  Let $K: \Yy \times \Yy \to \R$  be a measurable  kernel such  that  $\mathfrak M_K: \Ss(\Yy) \to \Hh (K)$ is well-defined.  Assume that  $\| f\| _\infty \le 1$  for all $f \in \Hh(K)$ with $\| f\|_{\Hh (K)} \le 1$. Then   for  any $\eps \in (0,1)$ we have
\begin{align}
\mu\Big\{ S_n \in \Yy^n:\, \| \mathfrak M_K (\mu_{S_n})-\mathfrak M_K(\mu)\| _{\Hh (K) }\le  2 \sqrt{ \frac{ \int_\Yy K (y, y) d\mu (y) }{n}}\nonumber\\ + \sqrt{ \frac {2 \log \frac{ 1}{\eps}}{n}}\Big\} \ge  1 -\eps.\label{eq:lmst2015}
\end{align}
\end{proposition}
		
$\bullet$  Examples  of    PDS  kernels, which  satisfy  the conditions  of  
 Propositions \ref{prop:srithm32} and \ref{prop:lmst2015}, are Gaussian kernels $K: \R^n  \times \R^n \to \R,  (x, y)  \mapsto  \exp   (- \sigma \| x-y\|_2 ^2)$,  where $\sigma > 0$,  Laplacian kernels
 $ (x, y)\mapsto  - \sigma \| x- y\| _1 $,  where $\sigma  >0$, and  Mat\'ern kernel
 $$K (x, y) = \frac{c^{2r-d}}{\Gamma (r-d/2) 2^{r-1-d/2}}\Big (\frac {\| x-y\|_2}{c}\Big ) ^{r-d/2}B_{-d/2-r} (c \|x-y\|_2), $$
where  $ r > d/2, c> 0$,  $ B_{a}$  is the third modified  Bessel  function of the  third kind of order $a$,  and $\Gamma$ is  the Gamma   function, see, e.g., \cite[\S 2.1]{MFSS2017} for a more complete  list of  examples.

$\bullet$   Assume that $\Xx\subset \R^n$   with the induced  metric  structure  is a separable  metric  space.
Since  $i: \Xx \to \R^n$ is a continuous  mapping, by Lemma \ref{lem:meast}, the  push-forward map  $i_*: \Ss (\Xx) \to \Ss (\R^n)$ is   $(\tau_w, \tau_w)$-continuous  and hence  measurable  mapping. Clearly   $i_*$ is a linear  injective  embedding.

\begin{lemma}\label{lem:ssep}  Assume that  $\Yy$ is a Polish  subspace of $\R^n$  and  $K: \R^n \times \R^n \to \R$  is a continuous  PDS kernel  that satisfies  the condition  of Proposition \ref{prop:srithm32}. Let $\tilde K$ be the restriction of $K$ to $\Yy \times \Yy$.   Then

(1) $\tilde K$ is  a   PDS kernel  that also   satisfies  the condition  of Proposition \ref{prop:srithm32}.

(2) Denote by $i: \Yy \to \R^n$ the canonical embedding. For  any $\mu, \nu \in \Ss(\Yy)$ we have
\begin{equation}\label{eq:tildek}
\la \mathfrak M_{\tilde K} (\mu), \mathfrak M_{\tilde K} (\nu) \ra _{\Hh (\tilde K)}  = \la \mathfrak M_K (i_*\mu), \mathfrak M_K (i_*\nu) \ra_{\Hh (K)}.  
\end{equation}
\end{lemma}
\begin{proof} (1)  The first  assertion  of Lemma \ref{lem:ssep}  can be verified  easily.
	
	 (2) Let $\mu, \nu \in \Ss (\Yy)$.   We compute
$$\la \mathfrak  M_{\tilde K} (\mu), \mathfrak M_{\tilde K} (\nu)\ra _{\Hh (\tilde K)} = \int_\Yy \int_\Yy \tilde K (y, y') d\mu (y) d\nu (y')$$
$$ = \int _{\R^n} \int_{\R^n} K (y, y')d  i_*\mu (y) di_*\nu (y') = \la \mathfrak  M_K ( i_* (\mu)),  \mathfrak M_K  (i_* (\nu))\ra_{\Hh (K)}.$$
This  completes  the  proof of Lemma \ref{lem:ssep}.
\end{proof}	

If $K$ is a measurable   kernel  on $\R^n$   such that the   kernel mean embedding  $\mathfrak M_K: \Ss(\R^n) \to \Hh (K)$ is well-defined,  for  any  Polish  subspace $\Yy $ in $\R^n$  we shall denote  the  pullback inner product   on $\Ss  (\Yy) $ by $\la  \cdot, \cdot \ra_{\tilde K}$.

\subsection{Measurability and continuity  of kernel mean embeddings}

Let $\Hh (K)$ be the      RKHS   associated  to a  measurable  SPD kernel $K: \Yy \times \Yy \to \R$.
Denote by     $\Bb a (\tau_W)$ the smallest   $\sigma$-algebra  on $\Hh(K)$ such that  any  continuous  linear function on $\Hh(K)$ is $\Bb a (\tau_W)$-measurable, and by $\Bb (\Hh(K))$ the Borel $\sigma$-algebra of  $\Hh(K)$.  If $\Hh$ is   separable,  $\Bb a (\tau_W) = \Bb (\Hh)$  \cite[Theorem 88, p. 194]{BT2004}.   It  is known that   if $\Yy$ is a  separable  topological  space
and $K$ is  a continuous  PDS  kernel  on   $\Yy$  then  $\Hh (K)$ is separable  \cite[Theorems 15, 17, p. 33-34]{BT2004}  and \cite[Lemma 4.33, p. 130]{SC2008}.

\begin{lemma}\label{lem:bounded}  Assume  that  $K: \Yy \times \Yy \to \R$ is a  bounded  measurable  PDS  kernel.  Write $\Hh: = \Hh (K)$.
	
(1) Then    the  kernel mean embedding  $\mathfrak  M_K: (\Ss(\Yy), \Sigma_w) \to (\Hh, \Bb a(\tau _W))$ is measurable.	If $\Hh$ is   separable,  then  $\mathfrak  M_K: (\Ss(\Yy), \Sigma_w) \to (\Hh, \Bb(\Hh))$ is measurable.

(2) Assume further  that $\Yy$ is  a topological  space  and   $K$ is  separately continuous, i.e. $K_y \in C(\Yy, \R)$ for all $y \in \Yy$. Then   $\mathfrak  M_K: (\Ss(\Yy),\tau_w) \to  (\Hh, \tau_W)$  is   continuous,   where $\tau_W$ is the weak  topology  on $\Hh$.

(3)  Assume further that  $\Yy$ is  a topological  space,   $K$ is  continuous. Then the 
norm $\| \cdot \|_{\tilde K}: (\Ss (\Yy), \tau_w) \to \R, \mu \mapsto \|\mu\|_{\tilde K}$ is continuous.
\end{lemma}

\begin{proof}  (1)   To prove   that  $\mathfrak  M_K: (\Ss(\Yy), \Sigma_w) \to (\Hh, \Bb a(\tau _W))$ is measurable, taking into account  Lemma \ref{lem:fsb},	it suffices  to  show that for any   $ f \in \Hh(K)$  the  composition    $ \la f , \mathfrak M_K\ra: (\Ss (\Yy), \Sigma_w) \to \R,  \mu \mapsto \la \mathfrak M_K (\mu), f \ra_{\Hh (K)},$ is  a  bounded measurable map.  We   have
$$  \la f,  \mathfrak M_K (\mu)\ra_{\Hh (K)}  = \int _\Yy \la  f, K_y \ra_{\Hh (K)} d\mu (y) = \int_\Yy  f(y) d\mu (y).$$       
We  always   can assume  that   there   exists a sequence   of  functions $K_{y_1}, \ldots,  K_{y_n}$ converging  point-wise to
$f$ by \cite[Corollary 1, p. 10]{BT2004}.  Since $K$ is bounded,  $f$ is bounded.  Since $K_{y_n}$ is measurable,  $f$ is measurable. Since   $K$ is bounded,  $f$ is bounded.  
Hence   
$$\la  f, \mathfrak M_K \ra (\mu) = I_f (\mu)$$
 where $f \in \Ff_b (\Yy)$.  This completes the first assertion of  Lemma \ref{lem:bounded} (1). 

The second assertion  of Lemma \ref{lem:bounded} (1) follows immediately by \cite[Theorem 88, p. 194]{BT2004}.
	
(2)  Assume  further that  $\Yy$ is a  topological space and $K$ is   separately continuous.   Then  by \cite[Lemma 4.28, p. 128]{SC2008},   any  element  in $\Hh =\Hh(K)$  is bounded and continuous.  Hence   for  any  $f \in \Hh$  we have
$$\la f, \mathfrak M_K\ra: (\Ss(\Yy), \tau_w) \to \R, \, \mu \mapsto  \int _\Yy \la  f, K_y \ra d\mu (y) = \int_\Yy  f(y) d\mu (y)$$
i.e. $\la f, \mathfrak M_K\ra (\mu) = I_f (\mu)$.  Taking into account   the boundedness  and continuity of  $f$, this completes  the second assertion   of   Lemma \ref{lem:bounded}.	

(3)   Note  that the  map $: (\Ss(\Yy),\tau_w) \to \R, \mu \mapsto \| \mu\|_{\tilde K},$ is the composition  of   the map   ${\rm diag}: (\Ss(\Yy), \tau_w) \to (\Ss (\Yy \times \Yy), \tau_w), \mu \mapsto \mu^2, $  and the  evaluation map $I_{K}: (\Ss (\Yy \times \Yy), \tau_w) \to \R, \nu \mapsto  \int _{\Yy \times \Yy} K d\nu$.  By Lemma \ref{lem:diag} (2), the map  ${\rm diag}$ is continuous. Since $K$ is bounded and continuous,  the map $I_K$ is  continuous. Hence  $\| \cdot \|_{\tilde K}$ is a continuous map. This completes the proof of Lemma \ref{lem:bounded}.
\end{proof}

\subsection{Instantaneous  correct loss  functions  via kernel  mean embeddings}

\begin{lemma}\label{lem:kme2}   Assume that  $\Yy$ is a   separable metrizable topological  space. Let $K:\Yy\times  \Yy \to \R$  be a bounded SPD  continuous kernel  and $h \in \mathbf{Meas} (\Xx, (\Ss(\Yy), \Sigma_w))$.
Write $\Hh = \Hh (K)$.

(1) Then the map
	$\mathfrak M_K$ is well-defined  on $ \Ss (\Yy)$  and   the  function  
	\begin{equation}
	\label{eq:lkh1}
L_h^K: (\Xx \times \Yy) \to \R_{\ge 0} \:	(x, y) \mapsto  \|\mathfrak M_K (  h(x) ) -  K_y\|_\Hh ^2, 
	\end{equation} 
	is  measurable. 
	
(2)  If $h \in  \mathbf{Meas}(\Xx, (\Pp(\Yy), \Sigma_w))$ then  the   function  $L_h^K$ is bounded.
\end{lemma}
\begin{proof}[Proof of Lemma  \ref{lem:kme2}]  (1)  The first assertion of Lemma  \ref{lem:kme2} (1) is a consequence of  the Bochner  theorem.
	
To prove the second assertion of Lemma  \ref{lem:kme2} (1),   we write
	\begin{equation}\label{eq:kme2}  \|\mathfrak M_K (  h(x) ) -  K_y\|_\Hh ^2   =  \|\mathfrak  M_K  (h (x))\| _\Hh ^2   + \| K_y\|^2_\Hh  - 2 \la  \mathfrak M_K (h (x)), K_y\ra_\Hh. 
	\end{equation}
	
We shall show  that  the  first summand  in the RHS of \eqref{eq:kme2},  $\rho_1:\Xx \times \Yy \to \R, \, (x, y)\mapsto  \| \mathfrak   M_K (h (x))\|_\Hh ^2,$ is a measurable bounded
function.  We write  $\rho_1$ as   the composition of the following  maps
$$ (\Xx \times \Yy)\stackrel{\Pi_\Xx}{\to} \Xx \stackrel{h}{\to}(\Ss(\Yy), \Sigma_w)  \stackrel{{\rm diag}}{\to}  (\Ss (\Yy \times \Yy),\Sigma_w) \stackrel{ \widehat M_K}{\to} \R,$$
where   
$$\widehat  M_K : (\Ss (\Yy \times \Yy), \Sigma_w )   \to  \R,   \nu \mapsto  \int_\Yy\int _{\Yy}  K(y, y')\,d\nu (y, y'), $$
and ${\rm  diag}$  is defined   in Lemma \ref{lem:diag} (2).
Since  $K : \Yy \times \Yy  \to \R$ is  a  bounded and continuous  function, and  $\Yy\times \Yy$ is  separable  and metrizable, the function  $ \widehat M_K$   continuous  in  the  weak-topology $\tau_w$ on $\Ss(\Yy \times \Yy)$,   and hence  measurable with respect to  the $\sigma$-algebra $\Sigma_w$ on $\Ss (\Yy \times \Yy)$. 
Since  $\Pi_\Xx$  and  $h$ are measurable maps,   taking into account   Lemma \ref{lem:diag}, 
we conclude   that $\rho_1$ is a measurable  function.  

Next, we    observe that the   second summand  in the RHS of \eqref{eq:kme2}, 
$$\rho_2:\Xx \times \Yy, (x, y) \mapsto  \| K_y\|^2_\Hh =  K(y, y),$$
 is a measurable and bounded function, since the function $K: \Yy \times  \Yy \to \R$ is measurable and bounded,  and    the mappings $\Pi_\Xx$ and   $\Delta_\Yy: \Yy \to \Yy\times \Yy, y \mapsto  (y, y),$  are  measurable.
 
Let $p: \Hh  \times   \Hh \to \R, (h, h')\mapsto \la  h, h'\ra_\Hh$ be the pairing  map.
Now we shall  prove that the function 
$$\rho_3:  \Xx \times \Yy\stackrel{(h, \Id_\Yy)} {\to} \Ss (\Yy) \times \Yy \stackrel{(\mathfrak M_K,   \hat K) }{\to }   (\Hh \times \Hh, \Bb(\Hh)\otimes \Bb (\Hh)) \stackrel{p}{\to }  \R, $$
is measurable.

  By  Lemma \ref{lem:bounded},  the  map $\mathfrak  M_K : (\Ss (\Yy), \Sigma_w) \to 
(\Hh, \Bb a (\tau_W))$ is measurable.  Since $\Yy$   is a separable  topological space  and $K$ is a continuous  PDS kernel,  $\Hh$ is   separable,  and hence $\Bb a (\tau_W) = \Bb (\Hh)$  by \cite[Theorem 88, p. 194]{BT2004}. Hence  $\mathfrak  M_K : (\Ss (\Yy), \Sigma_w) \to 
(\Hh, \Bb(\Hh))$ is measurable.  Since   
  $\hat K: \Yy \to (\Hh,\Bb (\Hh)), \, y \mapsto K_y,$ is  the composition of the measurable mappings $\delta: \Yy \to \Ss(\Yy)$, see Lemma \ref{lem:diag}(3), and $ \mathfrak M_K$,  the map $\hat K$ is measurable.

We note the   pairing  map $p: (\Hh \times \Hh, \Bb(\Hh)\otimes \Bb (\Hh)) \to \R$  is  measurable, since  $p: (\Hh \times \Hh, \tau_s \otimes \tau_s)\to \R$ is continuous.

Taking into account \eqref{eq:kme2},   we complete  the proof of Lemma \ref{lem:kme2} (1).

(2)  Let $C = \sup_{y\in \Yy} K(y, y)$.
If   $h \in \mathbf{Meas} (\Xx, (\Pp (\Yy), \Sigma_w))$ then  we        have
$L^K_h (x, y) \le   4 C$, This completes  the  proof  of  Lemma \ref{lem:kme2}.
\end{proof}

\begin{theorem}\label{thm:errorRKHS}  Assume  that  $\Yy$ is a  separable   metrizable topological   space,  $K: \Yy \times \Yy \to \R$ is a  bounded  continuous  PDS  kernel.   
	
(1) Then	the  kernel mean  embedding $\mathfrak M_K$ is well-defined  on  $\Ss (\Yy)$. 
	
(2)  For any $h \in \mathbf{Meas} (\Xx, \Ss (\Yy))$      the  function
	\begin{equation*} 
	L_h^K: (\Xx \times \Yy) \to \R_{\ge 0},\, 	(x, y) \mapsto  \| h(x) -  \delta_y\|_{\tilde K}^2, 
	\end{equation*} 
	is  measurable.

(3)  We  define an  {\it instantaneous   quadratic loss function} as follows
	$$L^K:  \mathbf{Meas}(\Xx, \Pp(\Yy)) \times \Xx \times \Yy \to \R_{\ge 0},  (h, x, y) \mapsto   L ^K_h (x, y).$$
	For any $\mu \in \Pp(\Xx \times \Yy)$ 
	the  expected  loss function  $R^{L^K}_\mu: \mathbf{Meas}(\Xx, \Pp(\Yy)) \to \R$,
	$$R ^{L^K} _\mu(h)  = \int _{ \Xx \times \Yy}   L _h^K (x, y) d\mu (x, y)$$
	takes a finite   value   at any  $h \in \mathbf{Meas}(\Xx , \Pp (\Yy))$.
	
(4) A regular conditional probability measure $\mu_{\Yy|\Xx}$  for $\mu\in \Pp (\Xx \times 
\Yy)$    is  a minimizer of   $R^{L^K}_\mu$.     
	
(5)  If the    kernel embedding $\mathfrak M_K : \Ss(\Yy)  \to \Hh (K)$  is injective,  then  any  minimizer  of   $R^L_\mu:\mathbf{Meas} (\Xx, \Pp(\Yy))\to \R$, where $\mu \in \Pp (\Xx\times \Yy)$,  is a  regular  conditional  measure $\mu_{\Yy|\Xx}$.
\end{theorem}

\begin{proof}[Proof of  Theorem \ref{thm:errorRKHS}]  (1) The first assertion  of Theorem \ref{thm:errorRKHS}  follows  from Lemma \ref{lem:kme2}.

(2) The second  assertion is  a reformulation of   the second assertion of  Lemma \ref{lem:kme2}(1).

(3)  The third  assertion of Theorem \ref{thm:errorRKHS}  follows from  Lemma \ref{lem:kme2}(2).

(4) We write
\begin{align}
  L_h^K (x, y)  =& \| h (x) - \mu_{ \Yy|\Xx} (\cdot |x)) +  \mu_{ \Yy|\Xx} (\cdot |x) - \delta_y\|_{\tilde K} ^2 \nonumber\\ 
 =& \| h (x) - \mu_{\Yy|\Xx} (\cdot |x)\| ^2_{\tilde K} + \| \mu_{ \Yy|\Xx}(\cdot |x) -  \delta_y\| ^2_{\tilde K}\nonumber\\
+ &  2 \la  (h (x) - \mu_{\Yy|\Xx} (\cdot |x),  \mu_{ \Yy|\Xx)}(\cdot |x) -  \delta_y \ra _{\tilde K}. \label{eq:error1}
 \end{align} 
Next, using the following disintegration  formula  for
 $f \in L^1 (\Xx \times \Yy, \mu)$,  where $\Xx, \Yy$ are measurable   spaces  and  $\mu\in \Pp(\Xx\times\Yy)$ admits  a regular  conditional probability $\mu_{\Yy|\Xx}$,
\begin{equation}\label{eq:disreg}
\int_{\Xx \times \Yy} f(x, y) d\mu (x, y)  = \int_\Xx\int_\Yy  f(x, y)\,  d\mu_{\Yy|\Xx}  (y|x)\, d\mu_\Xx,
\end{equation} 
we obtain 
\begin{eqnarray}
\int_{\Xx \times \Yy} \la  h (x) - \mu _{\Yy|\Xx} (\cdot |x)),  \mu_{ \Yy|\Xx)}(\cdot |x) -  \delta_y \ra _{\tilde K}  \, d\mu (x,y)\nonumber\\
 = \int_{\Xx} \Big \la h (x) - \mu_{\Yy|\Xx}  (\cdot|x),\int _\Yy  (\mu_{ \Yy|\Xx}(\cdot |x) -  \delta_y)\, d \mu_{\Yy|\Xx} (y|x)\Big\ra_{\tilde K} d(\pi_\Xx)_* \mu(x)  = 0,\nonumber
\end{eqnarray}
since  by \eqref{eq:kme1}   we have
\begin{equation}\label{eq:unbiased1}
\int _\Yy (\mu_{\Yy|\Xx}(\cdot|x) - \delta_y)\, d\mu_{\Yy|\Xx} (y|x) = 0.
\end{equation}
Hence, by \eqref{eq:error1}, we obtain
\begin{equation} 
R^{L^K}_\mu (h)   = \int_{\Xx \times \Yy} \| h (x) - \mu _{\Yy|\Xx} (\cdot |x))\| ^2_{\tilde K} + \| \mu_{ \Yy|\Xx}(\cdot |x) -  \delta_y\| ^2_{\tilde K}\, d\mu (x, y). \label{eq:van2}
\end{equation}
   
 Theorem  \ref{thm:errorRKHS} (4)  follows immediately from \eqref{eq:van2}.  
 
(5) The last assertion  of Theorem  \ref{thm:errorRKHS} follows immediately from  \eqref{eq:van2}. 
\end{proof}

The  proof  of  Theorem \ref{thm:errorRKHS} (5)  yields immediately the following.

\begin{corollary}\label{cor:correct}  Assume  $\Xx$ is a measurable  space and  $K: \Yy \times \Yy \to \R$ is a  continuous bounded characteristic  PDS kernel on a  separable metrizable topological   space $\Yy$.  Let $\Hh \subset \Meas(\Xx, \Pp (\Yy))$  and set 
$\Pp_{\Xx \times \Yy} : =  \cup_{h \in \Hh}  h (\Xx)$. Assume  that  the  kernel mean embedding
$\mathfrak M_K : \Pp_{\Xx \times \Yy} \to  \Hh (K)$
is injective. 
Then 
the  instantaneous loss  function  
\begin{equation}\label{eq:LK}
L^K: \Xx \times \Yy \times \mathbf{Meas} (\Xx, \Pp(\Yy)): (x, y, h)\mapsto  L^K_{h}(x, y),
\end{equation}
 is a  $\Pp_{\Xx \times \Yy}$-correct loss function.	  
\end{corollary}

\begin{examples}\label{ex:kernel}
	
(1) Let  $\Yy = \R^n$  and $K: \Yy \times \Yy \to \R$ be  defined  by $K(y, y') = \la  y, y'\ra$. The   restriction of $L^K$  to  $\Xx \times \Yy \times \mathbf{Meas} (\Xx, \Yy)$ is the  quadratic  loss  function  $L_2(x, y, f) =  | f(x) - y|^2$,
and  the minimizer of $R^{L_2}_\mu$ in $\mathbf{Meas} (\Xx, \Yy)$ is  the regression function $f (x)  = E_\mu (y |\P_Xx =x)$.

(2)  Let  $\Yy = \{ 0,  1\}$.   Let  us embed  $\Yy$ in $\R$   by  setting  $ \psi( i)   = i \in \R$ \ for  $i  = 0,1$.  
Then we  set  for $i, j \in \{ 0, 1\}$
$$K(i,j ) = \la \psi(i), \psi(j)\ra.$$
Then    $L^K: \Xx \times \{ 0, 1\} \times \mathbf{Probm} (\Xx, \Yy) \to  \R$  is
 \begin{equation}\label{eq:lp}
L_h  ( x, y)= | \psi(\overline h(x)) - \psi(y)|^2 
\end{equation}
and its   restriction  to  $ \Xx \times \{ 0, 1\} \times \mathbf{Meas} (\Xx, \Yy)$ is the 0-1 loss function.

(3)  Let  us generalize      Example \ref{ex:kernel} (2) and assume that $\Yy= \{  y_1, \cdots, y_n\}$.    Let $K: \Yy \times \Yy \to \R,\,  K(y, y') = {1/\sqrt 2}\delta_y (y')$.  Then  $\Hh (\Yy) = \R^n =\Ss(\Yy)$  and  $K_y = \delta _y$  for $y \in \Yy$. Then  for any $h \in \mathbf{Meas}(\Xx, \Yy)$, we have
 \begin{equation}\label{eq:ln}
 L_h  ( x, y)=  \frac{1}{2} \| \delta_{h(x)} - \delta_y\|^2.
 \end{equation}

(4)  Let   $\Yy$ be  a Polish  subset  in   $\R^n$ an $K: \R^n \times \R^n \to \R$ be a  continuous  bounded  PDS kernel  that  satisfies   the conditions of Propositions \ref{prop:srithm32} and  \ref{prop:lmst2015}. Then the restriction of $K$ to $\Yy$ also satisfies     all the conditions  in Theorem \ref{thm:errorRKHS}.
\end{examples}

\begin{remark}\label{rem:gl} 
(1)  The   instantaneous loss  function in \eqref{eq:lp}  was proposed  to the author by Frederic  Protin in 2021.

(2)   Park-Muandet   considered    the kernel  embedding   $\mathfrak M_K  (\mu_{\Yy|\Xx})$ of a regular   conditional probability  measure  $\mu_{\Yy|\Xx}$  in their  paper \cite[Definition 3.1]{PM2020}, assuming that $\Hh(K)$ is separable.  They proved  that  $\mathfrak M_K (\mu_{\Yy|\Xx}): \Xx \to \Hh (K)$  is measurable  with respect to the Borel  $\sigma$-algebra  $\Bb (\Hh)$  \cite[Theorem 4.1]{PM2020},  moreover it minimizes the   loss  function $R^{L^K}_\mu$    defined  on $L^2 (\Xx,\Hh (K), \mu_\Xx)$ \cite[Theorem 4.2]{PM2020}.	 Their  results  generalize the  results  due  to   Gr\"unerwalder  et al.  in \cite{GLGB12}, where    the  authors  considered   the problem of  estimating  the    conditional expectation $\E_\mu (h (Y)| \Pi_\Xx = x)$,  where  $h$ belongs to a    $\Hh (K)$- valued  RKHS  $\Hh_\Gamma$, using   the mean square  error      $R^L_\mu$  we considered  in  Lemma  \ref{lem:kme2}, under     many strong assumptions. They also noted  that  such a  $\Hh_\Gamma$ belongs to  the space of continuous  functions  from  $\Xx$ to $\Hh (K)$. 

(3) In \cite{TSS2022}  Talwai-Shami-Simchi-Levi   considered  the problem  of  estimating   conditional   distribution  $[\mu _{\Yy|\Xx}]$ by   representing them as  a  operator  $C_{\Yy|\Xx}: \Hh_\Xx  \to \Hh _\Yy$, where  $\Hh_\Xx$ and $\Hh_\Yy$ are  RKHSs     associated   with  PSD kernels  on $\Xx$ and $\Yy$ respectively.  

\end{remark}

\section{A generalization  of Cucker-Smale's result}\label{sec:genas}

In this section, we  keep the notation used in the previous sections.

$\bullet$  For a  Hilbert  space $\Hh$    and a topological  space  $\Xx$,
we  denote by   $C_b(\Xx, \Hh)$ the   space  of  all  continuous bounded   mappings from
$\Xx$ to $\Hh$  endowed  with the  sup-norm
$$ \|  f\|_{\infty}= \sup _{x \in \Xx} \| f (x)\|_\Hh.$$

$\bullet$ For any   precompact  metric space $\Hh$,  and $s>0$, we denote  the $s$-covering  number of $\Hh$  by  $\Nn (\Hh,s)$, i.e. 
$$\Nn (\Hh,s): = \min \{ l \in \N|\, \exists l \text{  balls  centered in $\Hh$ of radius  $s$  covering $\Hh$} \} < \infty.$$

In Subsection \ref{subs:state}, we state  our theorem (Theorem \ref{thm:msek}) and  discuss  the relation with previous  results. In   Subsection \ref{subs:msek},  we   give a proof of Theorem \ref{thm:msek}.
\subsection{Statement of the result}\label{subs:state}

\begin{theorem}\label{thm:msek}  Let $\Xx$ be a topological space and  $\Yy$   a Polish  subset in  $\R^m$.  Let $K$ be the restriction to $\Yy$ of a continuous  bounded  SPD kernel  on $\R^m$, which   satisfies  the conditions  of Proposition \ref{prop:srithm32}.   
Assume that $\Hh\subset C(\Xx, (\Pp(\Yy), \tau_w))$. 

(1) Then    $\Hh \subset C_b (\Xx, \Ss(\Yy)_{\tilde K} )$.

(2) Let  $A^{K}: \cup_{n =1} ^\infty (\Xx \times \Yy)^n \to \Hh, \, S_n\mapsto A^{K}_{S_n}, $ be  a C-ERM algorithm  for  the supervised  learning model  $(\Xx, \Yy, \Hh, L^K, \Pp (\Xx\times \Yy))$, where $L^K$ is defined in \eqref{eq:LK}, i.e.
$$L^K (x, y, h)   = \| h (x) - \delta_y\|_{\tilde K}.$$
  If  $\Hh  $ is a    pre-compact subset  in  $ C_b (\Xx,\Ss(\Yy)_{\tilde K} )_\infty$ then   for  any $\eps > 0, \delta > 0$ there  exists $m (\eps, \delta)$ such that for  any $m \ge m (\eps, \delta)$  and any $\mu \in \Pp (\Xx \times \Yy)$ we have
  	\begin{equation}\label{eq:covering}
  (\mu^m)_*\{ S_m \in (\Xx \times \Yy)^m |\, \Ee_{\Hh, L^K, \mu} (A^{K}_{S_m}) < 2\eps + c_m\} \ge  1- 2\Nn  ( \Hh, \frac{ \eps}{8C_K} )  2 \exp ( - \frac{m\eps^2}{4 C_K^2} )
  \end{equation}
where $$ C_K : = \sup_{y \in \Yy} \sqrt {|K(y, y)|} <\infty.$$
Hence  a C-ERM  algorithm     for    the supervised  learning model  $(\Xx, \Yy, \Hh, L^K, \Pp (\Xx\times \Yy))$ is a uniformly consistent   learning  algorithm, if  $\lim_{n \to \infty } c_n = 0$.

In particular,  if
	$$A_{S_n}^{K} \in  \arg \min_{h \in \Hh} (R^{L^K}_{\mu_{S_n}}(h) + c_n \| h\|_\infty ), $$
	then  $A^{K}$ is a C-ERM  algorithm   and
 	\begin{equation}\label{eq:coveringcn}
(\mu^m)_*\{ S_m \in (\Xx \times \Yy)^m |\, \Ee_{\Hh, L^K, \mu} (A^{K}_{S_m}) < 4\eps +c_mC_K\} \ge  1- 2\Nn  ( \Hh, \frac{ \eps}{8C_K} )  2 \exp ( - \frac{m\eps^2}{4 C_K^2} ).
\end{equation}	
\end{theorem}

\begin{remark}\label{rem:rkhs} 	 (1) Theorem \ref{thm:msek}(2) is a  generalization  of     Cucker-Smale's result \cite[Theorem C]{CS2001}, where  Cucker-Smale  also  assumed  implicitly that    their     ERM  algorithm  must  satisfy  certain  measurability, see  Remark \ref{rem:csl2}.  Cucker-Smale also assumed the boundedness  of
	predictors $h$ in   hypothesis classes  $\Hh \subset \Meas (\Xx, \R^n)$.  In our setting of conditional  probability  estimation,  the boundedness  of  Markov kernels  $h \in  C (\Xx, \Pp (\Yy)_{K_2})$  is  automatically  satisfied  (Theorem \ref{thm:msek} (1)). 
	
(2) Gr\"unerwalder  et  al. \cite{GLGB12}, based  on  results due to Capponento and de Vito \cite{CD07},    proved   the  learnability  of   their   statistical learning  model for  conditional mean  embeddings    under a strong assumptions on  statistical  model $\Pp_{\Xx \times \Yy}$,  on the finiteness    dimension of   a      RKHS  $\Hh (K)$  associated  to    a  PSD kernel  $K$ on $\Yy$    and moreover the  hypothesis  space $\Hh_\Gamma$, which is a  vector valued RKHS,  must contain   a minimizer $\mu_{\Yy|\Xx}$ of   the mean square error.
	Park-Muandet     proved   the  universal  consistency  of a regularized  ERM algorithm  of  their   statistical learning model  $(\Xx, \Yy, \Hh_\Gamma, L^K, \Pp (\Xx\times \Yy))$, where  a hypothesis  space $\Hh_\Gamma$  is a  vector-valued  RKHS  associated  with a $C_0$-universal kernel $l_{\Xx\Yy}: \Xx \times \Xx \to \Hh (K_\Yy)$,  $\Hh (K_\Yy)$ is a separable  RKHS  associated with a  measurable bounded  SPD kernel $K_\Yy$ on $\Yy$ and  $\Xx$ also admits   a  bounded  measurable  kernel $K _\Xx$  such that $\Hh (K_\Xx) $ is a separable Hilbert space \cite[Theorem 4.4]{PM2020}.   They provided a   convergence  rate  of   their  learning algorithm  under the condition that $\Hh_\Gamma$  contains a minimizer $\mu_{\Yy|\Xx}$  of   the mean square error \cite[Theorem 4.5]{PM2020}.  In \cite{TSS2022} Talwai-Shameli and Simchi-Levi improved  the result by Gr\"unerwalder  et  al. \cite{GLGB12}  by dropping many  technical assumptions  in \cite{GLGB12}  and applying  the theory of interpolation spaces for RKHS \cite{FS20}. Note that      a  minimizer  of the loss function  in  $\Hh_\Gamma$  in   the papers considered  above may not  correspond to  a Markov  kernel,  i.e., to  a measurable map   $\overline T:\Xx \to \Pp (\Yy)$.
\end{remark}

\subsection{Proof of Theorem \ref{thm:msek}}\label{subs:msek}

(1)  Assertion  (1) of Theorem \ref{thm:msek} stating that $\Hh \subset C_b (\Xx, \Ss (\Yy)_{\tilde K})$  follows from   Proposition \ref{prop:srithm32}  and Lemma \ref{lem:ssep}.

\

(2) To prove  assertion (2) of Theorem \ref{thm:msek}, we   shall   apply   Theorem \ref{thm:uniform}, namely  we shall   give  an upper bound  for  the sample complexity  of the   supervised  learning model  $(\Xx, \Yy, \Hh, L^K, \Pp (\Xx \times \Yy))$  in  Proposition \ref{prop:BCS} below,  using  the same  strategy
in the  proof of Cucker-Smale's theorem \cite[Theorem  C]{CS2001},  though  Cucker-Smale   considered  $\Yy = \R^n$, $ n  < \infty$  and  assumed  that $\Hh \subset C(\Xx, \R^n)$ is compact \cite[Remark 14]{CS2001}, they also did not use  inner/outer measure.  

First we  shall   prove  Lemmas \ref{lem:ACS}, \ref{lem:CSP3}, \ref{lem:csl1},  then we  shall    prove  Proposition \ref{prop:BCS} and    complete  the  proof  of Theorem \ref{thm:msek} (2).

\begin{lemma}\label{lem:ACS} Assume the  condition of Theorem \ref{thm:msek}. Then for any $h \in \Hh$  the following claims hold.

(1)  The   function $L^K_h: (\Xx \times \Yy)\to  \R, (x, y)\to L^K (x, y, h)$ is 
measurable.

(2) For any $\mu \in \Pp (\Xx \times \Yy)$  we have	\begin{equation}\label{eq:CSa}
	\mu ^m \{ S \in  (\Xx \times \Yy)^m:\,  |R^{L_K}_\mu  (h) - \widehat {R^{L_K}_S} (h)| \le \eps \} \ge 1 - 2\exp (- \frac{m \eps^2}{4C_K^2})\}
	\end{equation}
\end{lemma}

\begin{proof}   (1) The first assertion  of Lemma \ref{lem:ACS}   follows  from  Lemma \ref{lem:kme2}.
	
(2) We note that 
	\begin{equation}\label{eq:boundedkme}
	\| \mu \| _{\tilde  K} \le  C_K  \text{ for  all } \mu \in \Pp(\Yy). 
	\end{equation}
	
Using the first assertion  of Lemma \ref{lem:ACS}  and \eqref{eq:boundedkme}, we derive  \eqref{eq:CSa}  from  the Hoeffding  inequality, stating that  for any
	 measurable mapping $\xi: \Zz  \to [a,b]  \subset \R$,   $\mu \in \Pp (\Zz)$, we have \cite[Theorem 2.8, p. 34]{BLM13}
	\begin{equation}\label{eq:CSh}
\mu^m \Big\{ (z_1, \ldots, z_m)\in \Zz^m \big | \frac{1}{m}	\sum_{i =1}^m \xi (z_i)   - \E_\mu (\xi)\big | \ge \eps \Big\} \le  2 \exp ( - \frac{2m\eps^2}{  (b-a)^2} ),
\end{equation}
noting  that  \eqref{eq:boundedkme}  implies
\begin{equation}\label{eq:bound1}
 0 \le L^K_h (x, y)   = \| h(x) - \delta_y \|^2 _{\tilde K}  \le  4 C_K^2,
\end{equation}
 and   then  plugging  $\xi: = L^K_h : \Xx \times \Yy  \to [0, 4C_K^2]$ into  \eqref{eq:CSh}.
\end{proof}

\begin{lemma}\label{lem:CSP3}  For  any $f , g \in \Hh$, $\mu \in \Pp (\Xx \times \Yy)$ and  $S_n \in (\Xx \times \Yy)^n$  we  have
	\begin{equation}\label{eq:CSP3}
|(R^{L^K}_\mu (f)- R^{L^K}_{S_n} (f)) - (R^{L^K}_\mu (g) -R^{L^K}_{S_n} (g))| \le 8 C_K  \|  f -  g \| _\infty.
	\end{equation}
\end{lemma}
\begin{proof} Let $f, g \in C(\Xx, \Pp(\Yy))$.  Using  \eqref{eq:boundedkme}, we obtain
$$\vert \| f (x)\|^2  _{\tilde K} - \| g(x)\|^2_{\tilde K}\, \vert = | \la f(x)- g(x)| f(x)\ra_{\tilde K} + \la g(x) - f(x)|- g (x)\ra_{\tilde K}|$$
$$ \le |\| f (x)  - g(x)\|_{\tilde K}|\cdot  ( \| f(x)\|_{\tilde K} + \| g (x) \|_{\tilde K}) \le  2 C_K  \|f (x) - g(x)\|_{\tilde  K}. $$
We also have
$$ | 2\la  f(x) - g(x)|\delta_y\ra_{\tilde K}|\le 2 C_K \| f (x)- g(x)\|_{\tilde K}.$$
Hence,
\begin{align}
| R^{L^K}_\mu (f) - R^{L^K}_\mu (g)| & = \big | \int_{\Xx \times \Yy} \| f(x)\|_{\tilde K} ^2   - \|g(x)\|_{\tilde K} ^2  - 2 \la f(x)-g(x)| \delta_y \ra_{\tilde K} \, d\mu(x, y) \big |\nonumber\\
& \le   4 C_K \| f -  g \| _\infty  .\label{eq:CSP3a}
\end{align}
In particular, for any $S_n  = (x_1, y_1, \ldots, x_n, y_n) \in (\Xx \times \Yy)^n$ we have
\begin{align}
\Big | R^{L^K}_{S_n} (f) -  R^{L^K}_{S_n} (g)\Big | = \frac{1}{n}  \sum_{i =1}^n \Big (\| f(x_i)\|_{\tilde K} ^2   - \|g(x_i)\|_{\tilde K} ^2\nonumber\\
- 2 \la f(x_i)-g(x_i)| \delta_{y_i} \ra_{\tilde K}\Big ) \le   4 C_K \|f - g \| _\infty .\label{eq:CSP3b}
\end{align}
Clearly  \eqref{eq:CSP3}  follows from \eqref{eq:CSP3a} and \eqref{eq:CSP3b}.
\end{proof}

\begin{lemma}\label{lem:csl1} cf. \cite[Lemma 1]{CS2001}
	Let $\Hh = D_1 \cup   \ldots \cup D_l$, where  $D_i$ are  open balls   centered  in $h_i \in \Hh$, $i \in [1, l]$,    and $\eps  >0$.   Then
\begin{align}\label{eq:lem1CS}
(\mu ^m)^* \{S \in (\Xx \times \Yy)^m:  \sup _{h \in \Hh} |R^{L^K}_\mu (h) -  \widehat  R^{L^K}_S (h)| > \eps \}\nonumber \\
\le \sum_{j =1}^l (\mu^m)^* \{S \in (\Xx \times \Yy)^m:\sup_{ h \in  D_j} |R^{L^K}_\mu (h) -  \widehat  R^{L^K}_S (h) | > \eps  \}
\end{align}
\end{lemma}
\begin{proof} Lemma \ref{lem:csl1} follows  from the equivalence
$$\sup _{ h \in \Hh} |R^{L^K}_\mu (h) -  \widehat  R^{L^K}_S (h) |\ge \eps \}  \LLR  \exists  j \in [1,l] :  \sup_{h \in D_j}  |R^{L^K}_\mu (h) -  \widehat  R^{L^K}_S (h) |\ge \eps,$$
 taking into account the countable  subadditivity of outer measure.
\end{proof}
\begin{proposition}\label{prop:BCS} Assume the  condition of Theorem \ref{thm:msek}. Then for any $\eps > 0$, $\mu  \in \Pp (\Xx \times \Yy)$    and $h \in \Hh$ we have
	\begin{eqnarray}\label{eq:CSb}
	(\mu ^m)_* \{ S \in  (\Xx \times \Yy)^m:\, \sup_{h \in \Hh} |R^{L_K}_\mu  (h) - \widehat {R^{L_K}_S} (h)| \le 2\eps \} \ge \nonumber\\
	1 -\Nn (\Hh, \frac{\eps}{8C_K^2}) 2\exp (- \frac{m \eps^2}{4C_K^2})\}.
	\end{eqnarray}
\end{proposition}
\begin{proof}
Let  $l = \Nn (\Hh,\frac{\eps}{8C_K} )$ and consider  $h_1, \ldots, h_l$  such that  the collection of balls $D_j$ centered  at $h_j$  with radius $\frac{\eps}{8C_K}$  covers  $\Hh$.  By Lemma \ref{lem:CSP3},  for any $S \in (\Xx \times \Yy)^m$  and $f \in D_j$  we have
$$| ( R^{L^K}_\mu (f) -  \hat  R^{L^K}_S (f))-  ( R^{L^K}_\mu (f_j) -  \hat  R^{L^K}_S (f_j)) | \le  8 C_K\frac{\eps} {8 C_K} =  \eps. $$
It follows that
\begin{equation*}
\sup _{ f \in D_j}  |  R^{L^K}_\mu (f) -  \widehat  R^{L^K}_S (f)|  \ge 2\eps  \LRA  |  R^{L^K}_\mu (f_j) -  \hat  R^{L^K}_S (f_j)| \ge \eps.
\end{equation*}
Hence,   taking into account Lemma \ref{lem:kme2}, for any $j \in [1, l]$   we obtain:
\begin{align}\label{eq:probsup}
(\mu^m)^*\{ S\in (\Xx \times \Yy)^m:\, \sup_{ f \in D_j} |R^{L^K}_\mu (f) -  \widehat  R^{L^K}_S (f)|  \ge 2\eps  \} \le\nonumber\\
 \mu^m \{ S\in (\Xx \times \Yy)^m:| R^{L^K}_\mu (f_j) -  \widehat  R^{L^K}_S (f_j)| \ge \eps  \}\nonumber \\
\text{ (by Lemma \ref{lem:ACS}) }\le     2 \exp  ( -\frac{m \eps ^2}{4 CK^2} ).
\end{align}
Taking into account Lemma \ref{lem:csl1}  and \eqref{eq:complement}, 
we obtain from  \eqref{eq:probsup}  Proposition \ref{prop:BCS}. 
\end{proof}

\begin{proof}[Proof of  Theorem \ref{thm:msek}(2)]  Using Theorem \ref{thm:uniform},   we obtain the second  assertion  of  Theorem \ref{thm:msek}(2)  from Proposition \ref{prop:BCS}. 
	
The last assertion of  Theorem \ref{thm:msek}(2) follows immediately.
\end{proof}

\section{A variant of Vapnik-Stefanyuk's theorem and  applications}\label{sec:vapnik}

\subsection{A variant  of Vapnik-Stefanyuk's regularization method}\label{subs:vapnikstoc}
In  \cite[Chapter 7]{Vapnik1998}  Vapnik  explained  methods  due  to Stefanyuk-Vapnik \cite{VS1978} to   solve  the following  operator  equation
\begin{equation}\label{eq:vapnik71}
Af = F
\end{equation}
defined  by a  continuous  operator $A$ which  maps  in a one-to one manner  the elements  $f$  of a metric  space $(E_1, \rho_{E_1})$ into  the  elements  of a metric
space $(E_2, \rho_{E_2})$  assuming that 
a solution  $f \in  E_1$  of \eqref{eq:vapnik71}  exists  and is unique. 

We consider  the situation  when    $A$ belongs  to a   space  $\Aa$ and  instead  of  Equation \eqref{eq:vapnik71}  we are given a   sequence $\{ F_{S_l} \in E_2, \, l \in \N^+\}$,  a sequence  $\{ A_{S_l} \in \Aa, l \in \N^+ \}$, where $S_l$ belongs  to a  probability space $(\Xx_l,\mu_l)$ and  $A_{S_l}, F_{S_l}$   are   defined  by a family of   maps $\Xx_l \to E_2,\,  S_l \mapsto \Ff_{S_l},$  and  $\Xx_l \to \Aa,\, S_l \mapsto A_{S_l}$.

Let $W:  E_1 \to  \R_{\ge 0}$  be a lower semi-continuous     function   that satisfies  the   following  property (W).

(W)  The sets  $\Mm_c =  W^{-1} ([0,c])$ for  $c \ge  0$ are all  compact.
 
 Given  $A_{S_l}, F_{S_l}$, and $\gamma_l >0$, let us      define  a  regularized risk function $R^*_{\gamma_l} (\cdot, F_{S_l}, A_{S_l}): E_1 \to \R$  by
\begin{equation}\label{eq:regl}
R^*_{\gamma_l} (\hat f, F_{S_l}, A_{S_l}) = \rho_{E_2}^2  (A_{S_l}\hat f, F_{S_l}) + \gamma_l W(\hat f).
\end{equation}
 We shall say   that  $f_{S_l}\in E_1$  is  an {\it $\eps_l$-minimizer}    of  $R^*_{\gamma_l}$ if
 \begin{equation}\label{eq:regel}
 R^*_{\gamma_l} (f_{S_l}, F_{S_l}, A_{S_l}) \le  R_{\gamma_l} (\hat f, F_{S_l}, A_{S_l}) + \eps_l \text{ for   all } \hat f \in  E_1.
 \end{equation}
We  shall also use the shorthand notation $A_l$ for  $A_{S_l}$, $F_l$ for  $F_{S_l}$, $f_l$ for $f_{S_l}$,
$\rho_2$ for $\rho_{E_2}$, $\rho_1$ for  $\rho_{E_1}$.
 For any $\eps_l >0$,  an $\eps_l$-minimizer  of  $R^*_{\gamma_l}$  exists.
We will measure  the closedness  of operator  $A$ and operator $A_l$ by the distance
\begin{equation}\label{eq:v712}
\|A_l - A\| = \sup_{\hat f \in E_1}\frac{\|A_l \hat f -A\hat f\|_{E_2}}{W^{1/2}(\hat f)}.
\end{equation}
\begin{theorem}\label{thm:vapnik73} cf. \cite[Theorem 7.3, p. 299]{Vapnik1998} Let $f_{S_l}$ be  a $\gamma_l^2$-minimizer of $R^*_{\gamma_l}$ in \eqref{eq:regl}  and  $f$  the  solution of \eqref{eq:vapnik71}.    For any  $\eps >0$ and any constant $C_1, C_2  >0$ there
exists a value $\gamma_0 >0$ such that    for any $\gamma_l \le \gamma_0$
	\begin{eqnarray}\label{eq:vapnik713}
(\mu_l)^* \{S_l\in \Xx_l:\, \rho_1 (f_{S_l}, f) >\eps \}\le (\mu_l)^* \{S_l \in \Xx_l:\, \rho_2  (F_{S_l}, F) > C_1 \sqrt{\gamma_l}\}\nonumber \\ + (\mu_l)^* \{ S_l \in \Xx_l:\, \|A_{S_l} - A\| > C_2\sqrt{\gamma_l} \}
	\end{eqnarray}
	holds true.
\end{theorem}
\begin{remark}\label{rem:vapnik73} Note  that  Theorem \ref{thm:vapnik73}  is a slight  generalization of  \cite[Theorem 7.3, p. 299]{Vapnik1998}, where   Vapnik considered  the case   that  $f_l$  is a minimizer  of $R^*_{\gamma_l}$.  Our  proof  of  Theorem \ref{thm:vapnik73} follows   the arguments in  Vapnik's  proof  of \cite[Theorem 7.3]{Vapnik1998}, carefully  estimating all ``$\eps, \delta$"  for  outer  measure, instead  of   measure  as in  Vapnik's  proof, which   requires  also   measurablity  of  sets  involved.
\end{remark}

\begin{proof}[Proof of Theorem \ref{thm:vapnik73}]  Since   $f_{l}$ is a $\gamma_l^2$-minimizer of   $R^*_{\gamma_l} (\cdot, F_l, A_l)$,  by \eqref{eq:regel}
we have 
\begin{align}\label{eq:v723}
\gamma_l W(f_{l}) &\le R^*_{\gamma_l}(f_l, F_l, A_l)\le R^*_{\gamma_l}(f, F_l, A_l) + \gamma_l^2\\
& = \rho^2_2 (A_lf, F_l) + \gamma_l W(f) + \gamma_l^2 
\end{align}
where $f$ is the desired   solution of \eqref{eq:vapnik71}. From \eqref{eq:v723}  we find
\begin{equation}\label{eq:vapnik723a}
W(f_l) \le W(f)  + \frac{ \rho_2 ^2  (A_l f, F_l)}{\gamma_l} + \gamma_l.
\end{equation}
Since  according  to  the triangle  inequality, we have
\begin{align}
\rho_2 (A_l f, F_l)  & \le \rho_2  (A_lf, F) + \rho_2  (F, F_l)\nonumber\\
& \le  \| A_l -A\|  W ^{1/2} (f)  + \rho_2  (F, F_l), \label{eq:vapnik724}
\end{align}
we obtain 
\begin{equation}\label{eq:vapnik725}
W(f_l)\le W(f) +\frac{1}{\gamma_l} 
\Big( \| A_l - A\| W^{1/2}(f) + \rho_2  (F, F_l)\Big)^2 + \gamma_l.
\end{equation}
Since $f_l$ is a $\gamma_l^2$-minimizer of  $R^*_{\gamma_l} (\cdot,  F_l, A_l)$, we have
\begin{equation}\label{eq:vapnik725a}
\rho^2_2 (A_lf_l, F_l) \le R^*_{\gamma_l} (f_l, F_l, A_l) + \gamma_l^2.
\end{equation}
From \eqref{eq:vapnik725a}, \eqref{eq:v723}, and \eqref{eq:vapnik724}, we obtain
\begin{equation}\label{eq:vapnik726}
\rho^2_2 (A_lf_l,  F_l) \le \gamma_l W(f)  + \Big (  \| A_l -A\| W^{1/2} (f) + \rho_2 (F, F_l) \Big ) + \gamma_l^2.
\end{equation}
From this, using \eqref{eq:v712} and \eqref{eq:vapnik724}, we derive
\begin{align}
\rho_2 (Af_l, F)&\le \rho_2 (Af_l, A_l f_l) + \rho_2 (A_l f_l, F_l) + \rho_2 (F_l, F)\nonumber\\
&\le  W^{1/2} (f_l) \| A_l-A\| \nonumber\\
 & +\Big (\gamma_l W(f) + \big ( \|A_l - A\| W^{1/2}(f) + \rho_2 (F, F_l)  \big )^2 + \gamma_l^2 \Big)^{1/2}\nonumber\\
 & + \rho_2 (F_l, F) = \sqrt{\gamma_l} \Big ( \frac{\rho_2 (F_l, F)}{\sqrt{\gamma_l}} + W^{1/2} (f) \frac{ \| A_l - A\|} {\sqrt{\gamma_l}}\Big ) \nonumber\\
 & + \sqrt{\gamma_l} \Big ( W(f) + \big ( \frac{\rho_2 (F_l, F)}{\sqrt{\gamma_l}}  + W^{1/2} (f) \frac{\| A_l -A\|}{\sqrt{\gamma_l}}\big)^2 +  \gamma_l ^{3/2}\Big ) ^{1/2}\label{eq:v727}
\end{align}
Given $C_1, C_2 >0$   we set
\begin{equation}\label{eq:as727a}\Cc_l:= \{ S_l \in \Xx_l :   \frac{ \rho_2 (F, F_l)}{\sqrt {\gamma _l}} \le C_1 \text{ and } \frac{\| A_l -A\|}{\sqrt{\gamma_l}} \le C_2\}.
\end{equation}

Now  assume that
\begin{equation}\label{eq:as727b}
 S_l \in \Cc_l.
 \end{equation}
From  \eqref{eq:vapnik725}  we have
\begin{equation}\label{eq:v728}
 W(f_l) \le  W(f) +  (C_1 + C_2 W^{1/2} (f)) ^2  + \gamma_l = d + \gamma_l < \infty
\end{equation}
where 
\begin{equation}\label{eq:Wd}
d =  W(f) +  (C_1 + C_2 W^{1/2} (f)) ^2  >0.
\end{equation} 
From   \eqref{eq:v727}, using \eqref{eq:as727a} and \eqref{eq:v728},  we obtain
\begin{align}\label{eq:v729}
\rho_2 (Af_l, F)  &\le \sqrt{\gamma_l} (C_1  +  W ^{1/2}(f) C_2)\nonumber\\
& + \sqrt{\gamma_l} (d + \gamma_l ^{3/2}) ^{1/2} \nonumber\\
  & \le \sqrt{\gamma_l} ( \sqrt{d} + \sqrt{d + \gamma_l^{3/2}} )\nonumber\\
  	&  \le  2\sqrt{\gamma_l}  \sqrt{d + \gamma_l ^{3/2}}.
\end{align}
Now  we pose  the following condition on $\gamma_l$:
\begin{equation}\label{eq:gammal1}
\gamma_l \le \min\{ d, d ^{2/3} \}.
\end{equation}
Taking into account \eqref{eq:Wd} and the properties  of  the functional $W:D \to \R_{\ge 0}$, inequality  \eqref{eq:v728} implies that,  under the condition \eqref{eq:gammal1}  on $\gamma_l$, both $f$ and  $f_l$  belong to  the compactum $W^{-1} ([0,2d])$. 

\begin{lemma}\label{lem:vapnikl1}\cite[Lemma in p. 53]{Vapnik1998}. Let $E_1, E_2$ be metric spaces.
If $A:E_1 \to E_2$ is a  one-to one  continuous operator  defined   on  a compact set $M \subset E_1$,  then  the inverse  operator  $A^{-1}$ is continuous  on the set  $N =  A(M)$. 
\end{lemma}
By Lemma \ref{lem:vapnikl1},  for  any $\eps  > 0$ there exists   $ \delta(\eps)  >0$  such that
\begin{equation}\label{eq:v730}
\rho_2 (Af_l, Af) \le \delta(\eps)   \LRA  \rho_1 (f, f_l) \le \eps.
\end{equation}
Now we set  
\begin{equation}\label{eq:Lgamma0}
\gamma_0 = \min\{d, d ^{2/3}, \big (\frac{ \delta (\eps)^2}{8d})\}   > 0.
\end{equation}
Equations \eqref{eq:gammal1},  \eqref{eq:Lgamma0}, and \eqref{eq:v729}  imply  that   for   $\gamma_l \le \gamma_0$   we have
 \begin{equation}\label{eq:epsdelta}
 \rho_1 (f, f_l) \le \eps.
 \end{equation} 
 Hence, the condition \eqref{eq:Lgamma0}  for $\gamma_0$ implies that, for all $\gamma_l \le  \gamma_0$, by \eqref{eq:as727a}  and \eqref{eq:as727b} and taking into account  the subadditivity of outer measure, we have
$$(\mu_l)^*\big\{S_l: \rho_1 (f_l, f)> \eps\big\} \le (\mu_l)^* \big\{ S_l: 
\frac{\rho_2 (F_l, F)}{\sqrt{\gamma_l}} > C_1\big\} + 
(\mu_l)^*\big \{ S_l:\frac{\|A_l -A\|}{\sqrt {\gamma_l}} > C_2\big\}.$$
Note that $\gamma_0 = \gamma (C_1, C_2, W(f),  A^{-1}, \eps)$, where $C_1, C_2, \eps$ are arbitrary  fixed constant.  This completes the proof of Theorem \ref{thm:vapnik73}.
\end{proof}

\begin{remark}\label{rem:vapnik}
In \cite[\S 7.10]{Vapnik1998} Vapnik  applied   Stefanyuk-Vapnik's theorem \cite[Theorem 7.3, p. 299]{Vapnik1998}  for    estimation of  smooth conditional  probability  densities  on a line. 
 He also  presented  a method  of  estimation  of a  conditional  probability    $\mu_{\Yy|\Xx}$ where  $\Yy$ is finite  and  $\Xx = \R$ \cite[\S 7.12]{Vapnik1998}. In these cases  Vapnik assumed  that the hypothesis  spaces are finite  dimensional  and   explicit  minimum  of  the  expected  risk function  can be found.  In \cite{VI2019}    Vapnik and Izmailov    consider  estimation  of a  conditional  probability    $\mu_{\Yy|\Xx}$ where  $\Yy$ is finite  and   $\Xx \subset \R^n$, using    cumulative functions.  We shall show  an  application  of Theorem \ref{thm:vapnik73} below  for proving  the consistency  of a model  for  estimating   conditional  probability where $\Xx, \Yy$ are Polish subsets  in $\R^m$.  
\end{remark}

For metric  spaces  $F_1, F_2$,   we denote  by  $C_{Lip} (F_1, F_2)$  the  space of all  Lipschitz continuous  mappings  from  $F_1$ to $F_2$.
By Lemma \ref{lem:diag}(1),  $\Gamma_h: \Xx \to \Mm (\Xx \times \Yy)_{\tilde  K_1}$ is continuous, since  $h \in C (\Xx, \Mm (\Yy)_{\tilde K_2})$.  Since $\Gamma_h$ is continuous,  the function $x\mapsto \| \Gamma_h  (x)\| _{\tilde K_1}$  is continuous, and it is bounded, since $\Xx$ is compact.
Now we   define  the  distance  $d_M$ on $C (\Xx, \Mm (\Yy)_{\tilde K_2})$ by
\begin{equation}\label{eq:M}
d_M (f, f') = \sup _{x\in \Xx}(\| (f -f') (x)\|_{\tilde   K_2} +\| \Gamma_f  (x) - \Gamma_{f'}(x)\|_{\tilde K_1}).
\end{equation}

In other  words,  the metric $d_M$ is induced  by the norm $\|\cdot \|_M$ defined on the space $C (\Xx,\Ss(\Yy)_{\tilde  K_2})$ as follows:
$$\| f\|_M = \sup _{x \in  \Xx}  (\| f (x)\|_{\tilde K_2} + \|  \Gamma_f (x)\|_{\tilde K_1}).$$
The space  $C (\Xx, \Mm (\Yy)_{\tilde K_2})$   endowed  with the metric $d_M$  shall be denoted by  $C(\Xx, \Mm (\Yy)_{\tilde K_2})_M$.

  \begin{theorem}\label{thm:main}  Let  $\Xx $ be a compact subset in $\R^n \times \{0\}\subset \R^{n +m}$   and $\Yy$ a     compact subset  in $\{0\}\times \R^m$.   Let $K_2: \Yy \times \Yy \to \R $ be  the restriction  of a continuous  bounded SPD kernel $K$  on $\R^{m +n} \times \R^{m+n}$, which  satisfies    the conditions of Propositions \ref{prop:srithm32}  and \ref{prop:lmst2015}.
  	Denote by $\Pp_{Lip}(\Xx, \Pp(\Yy)_{\tilde K_2}, vol_\Xx)$  the set  of all   probability measures $\mu \in \Pp (\Xx \times \Yy)$  such that:\\
  	(i) ${\rm  sppt}\,\mu_\Xx  = \Xx$,  where $\mu_\Xx = (\Pi_\Xx)_*\mu$; \\
  	(ii)  there  exists a   regular   conditional   measure $\mu_{\Yy|\Xx} \in C_{Lip} (\Xx, \Pp (\Yy)_{\tilde K_2})$ for $\mu$  with respect to  the projection $\Pi_\Xx: \Xx \times \Yy \to \Xx$. 
  	
  	Let  $K_1: (\Xx \times \Yy)\times (\Xx \times \Yy)  \to \R$  be  the restriction  of the kernel $K$ to  $(\Xx \times \Yy)$. We define   a  loss function
  	\begin{equation}\label{eq:lostvk}
  	R^{K_1}:  C_{Lip}(\Xx, \Pp(\Yy)_{\tilde K_2})\times \Pp_{Lip} (\Xx, \Pp(\Yy)_{\tilde K_2}, vol_\Xx)\to \R_{\ge 0},  (h, \mu) \mapsto \| (\Gamma_h)_*\mu_\Xx -\mu\|_{\tilde K_1}.
  	\end{equation}
Then   for any  $\mu \in \Pp_{Lip}  (\Xx, \Pp(\Yy)_{\tilde K_2}, vol_\Xx)$, there  exists a  consistent    \grey{$C$-ERM} algorithm  
  	$A$  for  the supervised  learning model $(\Xx, \Yy, C_{Lip} (\Xx, \Pp (\Yy)_{K_2}),R^{K_1},\\  \Pp_{Lip} (\Xx, \Pp(\Yy)_{\tilde K_2}, vol_\Xx))$. Moreover,  for  any $\eps, \delta > 0$  there exists  $N(\eps, \delta)$ such that  for any $n \ge N(\eps, \delta)$ we have
 \begin{equation}\label{eq:van1}
 (\mu ^n)^*\{ S_n \in  (\Xx\times \Yy)^n :  d_M (A(S_n),\mu_{\Yy|\Xx}) > \eps \} \le \delta,
 \end{equation} 
 where $\mu_{\Yy|\Xx}\in C_{Lip}(\Xx, \Pp(\Yy)_{\tilde K_2}$ is  the unique  regular conditional  probability measure  for $\mu$ with respect to the  projection $\Pp_\Xx: \Xx \times \Yy \to \Xx$. 
  \end{theorem}
  
  \begin{remark}\label{rem:van1}  It follows from  Inequality \eqref{eq:estb} below that there exists
  	a positive constant $a$ depending only on $K$  such that   $ a \cdot d_M (h, \mu_{\Yy|\Xx}))\ge   R ^{K_1} _\mu (h)$  fro any $\mu \in \Pp_{Lip} (\Xx, \Pp(y)_{\tilde K_2}, vol_\Xx)$ and $h \in C_{Lip} (\Xx, \Pp (\Yy))$. Hence  we obtain  from \eqref{eq:van1}  the following  estimation
  	\begin{equation}\label{eq:van2}
  		(\mu ^n)^*\{ S_n \in  (\Xx\times \Yy)^n :  R^{K_1}_{\mu}  (A(S_n))\ge a \cdot\eps \} \le \delta.
  	\end{equation} 
  \end{remark}
  
A   proof of    Theorem \ref{thm:main} shall  be given   in the next subsection, in particular the uniqueness of $\mu_{\Yy|\Xx}$ in \eqref{eq:van1} shall be proved in Lemma \ref{lem:uniq}.  We  shall   provide   the  uniform consistency of a  learning algorithm $A$  for a  supervised learning model  $(\Xx, \Yy, C_{Lip} (\Xx, \Pp (\Yy)_{K_2}),R^{K_1},  \Pp_{\Xx \times \Yy})$, where  $\Pp_{\Xx \times \Yy} \subset   \Pp_{Lip} (\Xx, \Pp(\Yy)_{\tilde K_2}, vol_\Xx))$, in Corollary \ref{cor:main}   at the  end of the next subsection.

\subsection{Proof  of  Theorem \ref{thm:main}}

Assume the  condition  of  Theorem \ref{thm:main}.
We shall  apply Theorem \ref{thm:vapnik73} to  prove  Theorem \ref{thm:main}. First we  shall prove  Lemmas \ref{lem:continuity}, \ref{lem:uniq}  and Propositions  \ref{prop:empcl}, \ref{prop:W}  to ensure  that  the  conditions  of Theorem \ref{thm:vapnik73}     are satisfied. We also prove Lemma \ref{lem:continuous3}, which shall guide us  to define the   lower semi-continuous function $W: C_{Lip}(\Xx, \Pp (\Yy)_{\tilde  K_2}) \to \R$  in  Proposition \ref{prop:W}, which is the most delicate technical part  of the  proof. Lemma \ref{lem:continuous3} is also important in the proof of Corollary \ref{cor:main}.

For  $h \in C (\Xx, \Mm (\Yy)_{\tilde K_2})$, we   use the notation  $\Gamma_h$  instead of 
$$ \overline {\Gamma_{\underline h}}:  \Xx \to \Mm (\Xx \times \Yy),  x \mapsto \overline{\Gamma_{\underline h}}  (x) = \delta_x \cdot  h(x) .$$

\begin{lemma}\label{lem:continuity}  For any  $\mu \in \Ss(\Xx)$,  the map  
	$$\widehat \mu:  C (\Xx, \Mm (\Yy)_{\tilde K_2})_{M}  \to \Ss (\Xx \times \Yy)_{\tilde K_1} , \,  h \mapsto  (\Gamma_{\underline h} )_*\mu $$
	is continuous. If  ${\rm  sppt}\,\mu  = \Xx$,   then  $\mu$ maps  $C_b (\Xx, \Mm (\Yy)_{\tilde K_2})$  1-1 onto    its image.
\end{lemma}
\begin{proof} (1) Let  $h, h' \in  C(\Xx, \Mm (\Yy)_{\tilde K_2})_{M} $   be bounded  by a number $\alpha$, i.e., 
	$$ \| h \|_M \le \alpha \, \&  \| h'\|_M  \le \alpha $$
	for  some positive number $\alpha < \infty$.
	Let $\mu \in \Ss(\Xx)$. By \eqref{eq:tildek} we have 
\begin{align}
| \| (\Gamma_{\underline {h}})_* \mu - (\Gamma_{\underline {h'}})_* \mu\| _{\tilde K_1}^2| \le | \la (\Gamma_{\underline h} )_* \mu - (\Gamma_{\underline{h'}})_*\mu |  (\Gamma_{\underline h} )_*\mu\ra_{\tilde K_1} | \nonumber \\
+ |\la (\Gamma_{\underline {h'} })_*\mu| (\Gamma_{\underline h})_*\mu -(\Gamma_{\underline{h'}})_* \mu \ra_{\tilde K_1}|\nonumber \\
=| \int _{\Xx \times \Xx}    \la\Gamma_{h}(x) -\Gamma_{h'} (x), \Gamma_{h}(x')\ra_{\tilde K_1} d\mu (x)d\mu (x') |\nonumber\\
+ | \int _{\Xx \times \Xx}  \la  \Gamma_{h'}(x), \Gamma_{h'} (x')-\Gamma_h (x')\ra_{\tilde K_1} d\mu (x)d\mu (x')|\nonumber\\
  \le 2 \alpha\cdot   d_M  (h, h')\cdot  	\mu (\Xx)^2.	 \label{eq:estb}
 \end{align}
Inequality \eqref{eq:estb} implies  that    $\widehat \mu$ is a  continuous  map.  This completes  the  proof of the first assertion of Lemma \ref{lem:continuity}.
		
(2)	The second  assertion  of Lemma \ref{lem:continuity} follows  from Lemma \ref{lem:almostsurely}(1). This completes  the  proof  of  Lemma \ref{lem:continuity}.	
\end{proof}

\begin{lemma}\label{lem:uniq}  Assume that  $\mu \in \Pp_{Lip} (\Xx, \Pp (\Yy)_{\tilde K_2}, vol_\Xx)$. There  exists a unique    regular conditional probability measure  $\mu_{\Yy|\Xx} \in 
	C_{Lip} (\Xx, \Pp (\Yy)_{\tilde K_2})$ for  $\mu$ with respect  to the projection $\Pi_\Xx: \Xx\times \Yy\to \Xx$. 
\end{lemma}
\begin{proof} The existence  of  a regular conditional probability measure  $\mu_{\Yy|\Xx} \in 
	C(\Xx, \Pp (\Yy)_{\tilde K_2})$ for  $\mu$  follows  from the condition (ii) of Theorem \ref{thm:main}. The uniqueness  follows  from  Theorem \ref{thm:marginal}(2), taking into account the condition (i)
	of  Theorem \ref{thm:main}.
\end{proof}

The following Lemma   is a variant  of Proposition \ref{prop:tcontinuous}.

\begin{lemma}\label{lem:continuous3}  Assume the condition of Theorem \ref{thm:main}.  Let  $\overline  T\in C_{Lip}(\Xx, \Pp (\Yy)_{\tilde K_2})$. Let $K_3$  denote the restriction of $K$ to $\Xx$. Then  the map
$$(\Gamma_T)_*: \Pp (\Xx)_{\tilde K_3} \to \Ss (\Xx\times \Yy)_{\tilde K_1}$$
is continuous.
\end{lemma}
\begin{proof} By   Proposition  \ref{prop:tcontinuous},  the map
$$(\Gamma_T)_*: (\Pp (\Xx), \tau_w) \to  (\Pp (\Xx \times \Yy), \tau_w)$$ is continuous.  Since  the  weak* topology $\tau_w$ on $\Pp (\Xx)$ and on  $\Pp (\Xx \times \Yy)$ are   generated by the metric $\tilde K_3$ and  $\tilde K_1$, respectively,  it follows  that
$$(\Gamma_T)_*: \Pp (\Xx)_{\tilde K_3} \to  (\Pp (\Xx \times \Yy)_{\tilde K_2}$$
is continuous. Taking into account   the continuity of  the inclusion $ (\Pp (\Xx \times \Yy)_{\tilde K_2} \to \Ss (\Xx\times \Yy)_{\tilde K_1}$,  we conclude  Lemma \ref{lem:continuous3}.
\end{proof}

\begin{proposition}\label{prop:empcl}  Let $\mu \in \Pp_{emp} (\Xx \times \Yy)$.
	Then there exists     a map $f \in C_{Lip}(\Xx, \Pp(\Yy)_{\tilde K_2})$  such that
	$$(\Gamma_{\underline  f})_*\mu_\Xx  = \mu.$$
\end{proposition}
\begin{proof} Let $\mu: = \sum_{i=1}^N \sum_{j = 1}^M a_{ij} \delta_{x_i}\delta_{y_j} \in  \Pp_{emp} (\Xx \times \Yy)$ where $a_{ij} \in \Q^+$,   $ x_i \in \Xx$  and   $y_j \in \Yy$.  Then
	\begin{equation}\label{eq:proj}
	\mu_\Xx = \sum_{i =1}^N \sum_{j =1}^M a_{ij} \delta_{x_i}.
	\end{equation}
	\begin{lemma}\label{lem:regem}   If  $ f\in  C_{Lip}(\Xx, \Ss(\Yy)_{\tilde K_2})$ satisfies
		\begin{equation}\label{eq:regem}
		f (x_i) = \frac{ \sum_{j=1} ^M a_{ij}  \delta_{y_j}} { \sum _{j =1}^M  a_{ij}} \text { for  any }   i \in [1, N],
		\end{equation}
		then 
		$$(\Gamma_{\underline  f})_*\mu_\Xx  = \mu.$$
	\end{lemma}
	\begin{proof}  Lemma  \ref{lem:regem}  follows  directly  from the  equation
		\eqref{eq:markov1S}.
	\end{proof}
	{\it Continuation of the  proof  of  Proposition \ref{prop:empcl}.}  For $i \in [1, N]$ we set
$$Y_i = \frac{ \sum_{j=1}^M a_{ij} \delta_{y_j}}{\sum_{j=1}^M a_{ij}}.$$
Let  $V= {\rm span} (Y_1, \ldots, Y_N)$ be the  linear span  of $Y_1, \ldots, Y_N$.    First,  we  shall find a  Lipschitz continuous  map $f$  from  $\R^n$  to  $V$  such  that  $f $ satisfies  the    Equation \eqref{eq:regem}, equivalently:
$$ f(x_i)  = Y_i.$$
Then  the  restriction of  $f$ to $\Xx$ is  the required  Lipschitz  continuous map  in Proposition \ref{prop:empcl}.

Let  $\R$ be a  straight line in   $\R^n$ such that  the projection $\Pi_1: \R^n \to \R$    maps  $\{ x_1 \ldots, x_N\in \Xx\}$    1-1  onto its image in  $\R$. Since  the  restriction $\Pi_1|_{ 
\Xx}: \Xx \to \R$ of $\Pi_1$ to $\Xx$ is a  Lipschitz continuous  map, to construct  the  required  Lipschitz continuous  map $f:\R^n \to V$, it  suffices   assume  that $n =1$.   Now  we shall  construct  a polynomial mapping  $f$ from $\R$  to  $V$,    which satisfies  the  interpolation  equation \eqref{eq:regem}.

Given   a tuple of $N$  points $( Y_1, \ldots, Y_N)$ in the vector space  $V$ of dimension $d \le  N-1$  and   a   tuple  of  $N$ points  $ x_0 = 0 < x_1 < \ldots   < x_{N-1}\in \R$,  we  shall construct a  polynomial 
mapping $f: \R \to V$   such   that  
\begin{equation}\label{eq:inter}
f (x_i) =  Y_i.
\end{equation}
The coordinates of $Y_i$ are denoted by $ y _i ^j$.
Our map  $f$  consists of   polynomial
functions $f^1, \ldots,  f^d: \R \to \R$   such that
\begin{equation}\label{eq:interpolation}
f^j (x_i) =  y_i  ^j  \text{ for } j \in [1, d].
\end{equation}
We set
\begin{equation*}
f^i (x) := \lambda ^0 _i  + \lambda ^1 _i x + \lambda ^2 _i x (x  - x_1) + \cdots + \lambda ^n _i x ( x-x_1) \cdots (x-x_n),
\label{4.1.1}
\end{equation*}
where   the coefficients $\lambda^k _i$ are defined inductively: $\lambda^0 _i = y ^0 _i$, $\lambda ^1 _i = y ^1 _i - \lambda ^0 _i$, etc...
Clearly  the  constructed  polynomial  mapping $ f = (f^1, \ldots, f^d)$ 
satisfies  \eqref{eq:interpolation} and hence \eqref{eq:inter}.
This completes the  proof of  Proposition \ref{prop:empcl}.
\end{proof}

For any subset  $S$ of the  metric  space $C(\Xx, \Mm (\Yy)_{\tilde K_2})_M$ we denote by $S_M$  the  metric  space $S$ endowed  with  the induced metric  $d_M$.

\begin{proof}[Proof of Theorem \ref{thm:main}]  Let  $E_1 = C_{Lip} (\Xx, \Pp (\Yy)_{\tilde K_2})_{M}$,  $\Aa = \Pp (\Xx)$, $E_2  = \Ss (\Xx\times \Yy)_{\tilde K_1}$  and $\mu  \in \Pp (\Xx, \Pp(\Yy_{\tilde K_2}), vol_\Xx) \subset E_2$.   By Lemma \ref{lem:continuity}, for any  $\mu \in \Aa$, the  operator
	$\hat \mu: C_{Lip} (\Xx, \Pp (\Yy)_{\tilde K_2})_{M} \to E_2$  is a continuous  operator. We shall  apply Theorem  \ref{thm:vapnik73} to   prove the existence  of approximate  solutions   $f_{S_l} \in E_1$ for  solving the following equation of  a regular  conditional probability measure  $f$  for $\mu$ relative to the projection  $\Pi_\Xx: \Xx \times \Yy \to \Xx$: 
	\begin{equation}\label{eq:le1}  
	(\Gamma_{\underline f})_* ((\Pi_\Xx)_* (\mu)) = \mu.
	\end{equation}
By  Theorem \ref{thm:errorRKHS}, $f$ is the  minimizer   of the  loss function $R^{K_1}$. 

	Let  us  define  a function $W:  C_{Lip} (\Xx, \Pp (\Yy)_{\tilde K_2})_{M}  \to \R_{\ge 0}$ as  follows
	\begin{equation}\label{eq:w}
	W(f): = (\| f\|_{M}  +  L(f) + \| \Gamma_{\underline f}\|_{(\tilde K_3, \tilde K_1)} ) ^2
	\end{equation}

where   $L(f)$ is the  Lipschitz  constant  of $f$ and
\begin{equation}\label{eq:normvan}
\|\Gamma _{\underline f}\|_{\tilde K_3, \tilde K_1}: = \sup_{A, B \in \Pp (\Xx)}\frac{\|(\Gamma_{\underline f})_* (A-B)\|_{\tilde K_1}}{\|A-B\|_{\tilde K_3}}.
\end{equation}
By  Lemma \ref{lem:continuous3}, $\|\Gamma _{\underline f}\|_{\tilde K_3, \tilde K_1} < \infty$  for $f \in  C_{Lip} (\Xx, \Pp (\Yy)_{\tilde K_2})_{M}$.

\begin{proposition}\label{prop:W}   (1) $W: C_{Lip}(\Xx, \Pp (\Yy)_{\tilde K_2})_{M} \to \R_{\ge 0}$ is   a lower semi-continuous function. 
	
(2) Furthermore, for any $c \ge 0$, the  set $W^{-1} [0, c]$   is a compact  set in   $C_{Lip}(\Xx, \Pp (\Yy)_{\tilde K_2})_M$.
\end{proposition}

\begin{proof}
(1) To prove the  first assertion of   Proposition \ref{prop:W} we need  the following.
	
	\begin{lemma}\label{lem:lipschitzlconti}  The function 
		$$ C_{Lip}(\Xx, \Pp (\Yy)_{\tilde K_2})_M \to \R_{\ge 0}, f \mapsto L(f)$$
		is  a lower  semi-continuous function.
	\end{lemma}
	\begin{proof}[Proof of Lemma \ref{lem:lipschitzlconti}] We need  to prove  that  for any $f_0 \in C_{Lip}(\Xx, \Pp(\Yy)_{\tilde K_2})_M$  we have
		\begin{equation}\label{eq:semilow}
\lim _{f\stackrel{d_M}{\to} f_0}\inf  L(f) \ge  L(f_0)
		\end{equation}
where $f\stackrel{d_M}{\to} f_0$  means  that  the convergence is   in the topology generated  by $d_M$.     Since
$$ \lim _{f\stackrel{d_M}{\to} f_0}\inf  L(f) \ge \lim _{f\stackrel {\| \cdot \|_\infty}{\to}f_0}\inf  L(f),$$
the  inequality \eqref{eq:semilow}  follows  from  the lowe semi-continuity  of the Lifschitz  constant with respect  to the sup-norm, which is well-known. This  completes  the proof  of Lemma \ref{lem:lipschitzlconti}.
	\end{proof}

Now we shall prove  that $W$ is a lower semi-continuous function.   
Taking into account Lemma \ref{lem:lipschitzlconti}, it suffices  to show that  the function $C_{Lip} (\Xx, \Pp (\Yy)_{\tilde K_2})_{M}\to \R_{>0},  f \mapsto \| \Gamma _{\underline f}\|_{\tilde K_3, \tilde K_2},$  is lower semi-continuous.  For $f  \in C_L (\Xx, \Pp (\Yy)_{\tilde K_2})_{M}$,  let $A_\eps,  B_\eps \in \Pp (\Xx)$ be chosen such that
 \begin{equation}\label{eq:eps1}
 \| \Gamma _{\underline f}\|_{\tilde K_3, \tilde K_1} \le \frac{\|(\Gamma_{\underline f})_*( A_\eps - B_\eps)\|_{\tilde K_1}}{\| A_\eps - B_\eps\|_{\tilde K_3}} + \eps.
 \end{equation}
Write $C_\eps = A_\eps - B_\eps\in \Ss(\Xx)$. Let  $f' \in C_L (\Xx, \Pp (\Yy)_{\tilde K_2})_{M}$ such that
\begin{equation}\label{eq:b2}
\|f\|_\infty \le   \alpha   \text{ and } \| f'\| _\infty < \alpha 
\end{equation}
for some positive  $\alpha < \infty$.
\begin{lemma}\label{lem:est3} Under  the   assumption \eqref{eq:b2}  we  have
	\begin{equation}\label{eq:est3}
\| (\Gamma_{\underline f})_* C_\eps \| _{\tilde K_1}\le 2 C_K\alpha  \text{ and } \| (\Gamma_{\underline {f'}})_* C_\eps \| _{\tilde K_1}\le 2 C_K\alpha. 
\end{equation} 
\end{lemma}
\begin{proof} Lemma  \ref{lem:est3}  follows immediately from \eqref{eq:estb} taking into account  that $\|C_\eps\|_{TV} \le 2$.
\end{proof}
{\it Completion of the  proof of Proposition \ref{prop:W}}(1). Let
$$C_K  = \max  _{z \in \Xx \times \Yy} K (z, z).$$
 Then  we have (cf. the proof of Lemma \ref{lem:continuity}  and the  proof of Lemma \ref{lem:CSP3})
 \begin{align}
|\|( \Gamma _{\underline {f'}})_* (C_\eps)\|_{\tilde K_1}^2 - \|(\Gamma_{\underline f})_*(C_\eps) \|_{\tilde K_1}^2 |\le |\la  (\Gamma_{\underline {f'}})_* (C_\eps) - (\Gamma_{\underline {f}})_* (C_\eps)| (\Gamma_{\underline {f'}})_* (C_\eps)\ra_{\tilde  K_1} |\nonumber \\
 + |\la (\Gamma_{\underline {f}})_* (C_\eps ) -(\Gamma_{\underline {f'}})_* (C_\eps )| (\Gamma_{\underline {f}})_* (C_\eps )\ra_{\tilde K_1}|\nonumber \\
 \stackrel{\eqref{eq:est3}}{\le } 4 \alpha\cdot  C_k \|(\Gamma_{\underline{f-f'}})_* (C_\eps) \|_{\tilde K_1}\stackrel{\eqref{eq:estb}}{\le} 4\alpha \cdot C_K  \| f -f'\|_\infty  \cdot 2\label{eq:eps2},
 \end{align}
 since $\|C_\eps\|_{TV} \le 2$.

 Since 
 \begin{equation}\label{eq:lower}
 \|\Gamma_{\underline {f'}}\|_{\tilde K_3, \tilde K_1} \ge  \Big (\frac{ \|(\Gamma_{\underline{f'}})_*  (C_\eps)\|_{\tilde K_1} }{\|C_\eps \|_{\tilde K_3} } \Big), 
 \end{equation}
 taking into account \eqref{eq:eps1}  and \eqref{eq:lower}, we conclude from  \eqref{eq:eps2}  that
 \begin{equation}\label{eq:lowersemi}
 \lim_{f'\to f} \|\Gamma_{\underline {f'}}\|_{\tilde K_3, \tilde K_1}\ge  \| \Gamma _f \|_{\tilde K_3, K_1}.
 \end{equation} 
 
 Equation \ref{eq:lowersemi}  implies  that the function  $f \mapsto \| \Gamma _{\underline f}\|_{\tilde K_3, \tilde K_2}$ is lower semi-continuous, and hence $W$ is a   lower semi-continuous function.
 This completes  the   proof of the first assertion of Proposition \ref{prop:W}.

(2)	Since $\Yy$  is a compact  subset of $\R^m$, the set  $\Pp(\Yy)$ is compact in the  weak*-topology $\tau_w$.  Since $\Xx$  is compact, by   Arzel\`a-Ascoli theorem \cite[Theorem 18, Chapter 7, p. 234]{Kelley75}, for  any $ c\ge  0$  the    set  
$W^{-1}[0, c]$  is  compact  in $E_1$. This completes  the  proof of Proposition \ref{prop:W}.
\end{proof}
 By Lemmas \ref{lem:continuity}, \ref{lem:uniq} and Propositions \ref{prop:empcl}, \ref{prop:W},  all  the  requirements  of
	Theorem \ref{thm:vapnik73}  are satisfied.

Let  $\mu_{\Xx, l}  = (\pi_\Xx)_*\mu_{S_l}\in \Pp (\Xx)$ where  $\mu_{S_l} \in \Pp _{emp} (\Xx \times\Yy)$.  	For  $\mu \in \Pp (\Xx \times \Yy)$,   we have 
	\begin{align}
	\| \mu_{\Xx, l} - \mu_\Xx\|\stackrel{\eqref{eq:v712}}{ = }\sup _{f \in E_1}\frac{\| (\Gamma_{\underline f})_* (\mu_{\Xx,l} - \mu)\|_{\tilde K_1}}{W^{1/2}(f)}\nonumber\\
	\le \sup _{f \in E_1}\frac{\| (\Gamma_{\underline f})_* (\mu_{\Xx, l} - \mu_\Xx)\|_{\tilde K_1}}{\|\Gamma_f\|_{\tilde K_3, \tilde K_2}}\stackrel{\eqref{eq:normvan}}{\le} \|\mu_{\Xx, l} -\mu_\Xx\|_{\tilde K_1}\label{eq:vnorm1}.
	\end{align}
 Applying Theorem \ref{thm:vapnik73} to our   case with  
 $(\Xx_l, \mu_{l}) = ((\Xx \times \Yy)^l, \mu^l)$,  and letting $f_{S_l}$ to be  a $\gamma_l^2$-minimizer  of the    regularized  risk  function
 $$R^{K_1}_{\gamma_l}  (f, \mu_{S_l}) = R^{K_1}  (f, \mu_{S_l}) + \gamma_l W (f),$$
 taking into account \eqref{eq:vnorm1}, we obtain  the following estimation
\begin{align}
(\mu^l)^* \{ S_l : d_{M}(f_{S_l}, \mu_{\Yy|\Xx}) > \eps
\}\le (\mu^l)^* \{ S_l : \frac{\|\mu_l -\mu\|_{\tilde K_1}}{\sqrt{\gamma_l}} \ge C_1\}\nonumber\\
 +(\mu^l)^*\{ S_l : 
 \frac{\|\mu_{\Xx, l} -\mu_\Xx\|_{\tilde K_3}}{\sqrt{\gamma_l}} \ge C_2\}.\label{eq:vnorm}
\end{align}
Taking  into   account  Proposition \ref{prop:lmst2015} and Lemma \ref{lem:continuity},  we conclude  Equation \eqref{eq:van1} in Theorem \ref{thm:main}   from \eqref{eq:vnorm}. The consistency of the  C-ERM  algorithm  follows    from Equations \eqref{eq:van1}  and  \eqref{eq:estb}.
\end{proof}

\begin{corollary}\label{cor:main}  Assume that  $\Pp_{\Xx \times \Yy} \subset\Pp_{Lip} (\Xx, \Pp(\Yy)_{\tilde K_2}, vol_\Xx)$ satisfies   the following  condition (L).
	
(L)  $\Pp_{\Xx \times \Yy}$ is compact in the weak*-topology and the     function  $(\Pp_{\Xx \times \Yy})_{\tilde K_1} \to \R, \mu \mapsto    L (\mu_{\Yy|\Xx})$, where  $\mu_{\Yy|\Xx} \in C_{Lip} (\Xx, \Pp (\Yy)_{\tilde K_2})$,    takes a value  on a  finite interval $[a, b] \subset \R$.

(1) Then  the algorithm $A$  defined  in  the proof  of Theorem \ref{thm:main}  is uniformly consistent  for the  supervised  learning  model $(\Xx, \Yy, C_{Lip} (\Xx, \Pp (\Yy)_{\tilde K_2}), \Pp_{\Xx \times \Yy})$.

(2)  Assume further that the  PDS kernel $K: \R^{n+m} \times \R^{n+m} \to \R$ in Theorem \ref{thm:main}  is continuously ($C^1$-) differentiable and  $\Hh$ is a    subspace   of  the space $C_{Lip}(\Xx, \Yy)$. Then  the   supervised  learning  model $(\Xx, \Yy, \Hh, R^{K_1}, \Pp_{\Xx \times \Yy})$  has a generalization ability. 
\end{corollary}
\begin{proof}
	
(1)  To prove  the first assertion of  Corollary \ref{cor:main}, taking into account  the condition  (L) and Equation \ref{eq:estb}, it suffices to     show that  the  coefficient $\gamma_0  = \gamma (C_1, C_2, W(\mu_{\Yy|\Xx}), A^{-1}(\mu), \eps)$    defined  in  \eqref{eq:gammal1} in  the  proof  of Theorem \ref{thm:vapnik73}, adapted  to the  proof of Theorem \ref{thm:main}, can be   chosen  independently of $\mu\in \Pp_{\Xx\times \Yy}$, where
$A^{-1}(\mu) = \mu_{\Yy|\Xx}$.  By the  condition of Corollary \ref{cor:main},
$W(f)$,  and hence  the  constant $d$   in  Equation \eqref{eq:Wd},  and therefore the compactum $W^{-1}[0,2d]$ can be chosen  independently from $\mu$. 
Let $\Aa^\circ : = \{(\Pi_\Xx)_*\mu: \mu \in \Pp _{\Xx \times \Yy}\}$.  Since $\Pp_{\Xx\times \Yy}$ is compact in the weak*-topology,   $\Aa^\circ$ is compact in the  weak*-topology.  Now   we define  the  operator
$$U:  \Pp(\Xx)_{\tilde K_3}\times W^{-1}[0, 2d] \to \Pp_{Lip}  (\Xx, \Pp(\Yy)_{\tilde K_2}, vol_\Xx),\, (\mu,  h)\mapsto  \Gamma_{\underline h} (\mu).$$
We shall show  that $U$  is a continuous  map.  Let   $\mu, \mu ' \in \Aa^\circ$ and $h , h' \in  W^{-1}[0, 2d]$. Then  we have
\begin{align}
\| (\Gamma_{\underline h' })_*\mu' - \Gamma _{\underline h})_*\mu\|_{\tilde K_1} & \le \|  (\Gamma_{\underline h' })_*\mu' -  (\Gamma_{\underline h })_*\mu'\|_{\tilde K_1}\nonumber\\
 & + \| (\Gamma_{\underline h })_*\mu' - \Gamma_{\underline h}\mu\|_{\tilde K_1}\nonumber\\
  &  \stackrel{\eqref{eq:estb}}{\le }  2 C_K \cdot d_M (h, h') +  (\Gamma_{\underline h })_* (\mu' -\mu).
\label{eq:jconti}
\end{align}
By  Lemma \ref{lem:continuous3}, $(\Gamma_{\underline h })_* (\mu' -\mu)$ converges to zero when $\mu'$ converges  to $\mu$.  Taking  into account \eqref{eq:jconti}, we conclude  that $U$ is   a continuous map. 
By  Theorem \ref{thm:marginal} and taking into account the  condition (i) of Theorem \ref{thm:main},  $U$ is a 1-1 map.  By Lemma \ref{lem:vapnikl1},  
  the  map 
$$U^{-1}: (U (\Aa^\circ \times W^{-1}[0, 2d]))\to \Aa^\circ \times W^{-1}[0,d], \mu \mapsto (\mu_\Xx, \mu_{\Yy|\Xx}),$$  is  continuous. Since $U (\Aa^\circ \times W^{-1}[0, 2d])$ is  compact, $U^{-1}$  is uniformly continuous.  It follows  that  for any $\eps >0$ there  exists  $\delta (\eps) > 0$ such  that 
\begin{equation}\label{eq:uniform}
\|\mu-\mu'\|_{\tilde K_1}\le \delta (\eps) \LRA d_M (\mu_{\Yy|\Xx}, \mu'_{\Yy|\Xx}) \le \eps.
\end{equation}
Since   $\Pp_{\Xx \times \Yy} \subset  U (\Aa^\circ \times W^{-1}[0, 2d])$, \eqref{eq:uniform} implies  that $\delta (\eps)$ in Equation \ref{eq:v730} and  hence $\gamma_0 = \gamma (C_1, C_2, W(\mu_{\Yy|\Xx}), A^{-1}\mu, \eps)$ can be  chosen  independently  from $\mu$. This completes the  proof of   Corollary \ref{cor:main}(1).

(2) Assume  the condition  of  Corollary \ref{cor:main}(2). Since   the PDS   kernel $K$ is  $C^1$-differentiable, the  map
$  \delta: \Yy \to \Pp (\Yy)_{\tilde K_2},  y\mapsto   \delta_y,$  is a Lipschitz map. Hence  for any   $h \in \Hh$ the composition $\delta\circ h$ belongs to  $C_{Lip}(\Xx, \Pp(\Yy)_{\tilde K_2} )$.    
 To prove   Corollary \ref{cor:main}(2),   it suffices  to find  an algorithm
$\bar A: \cup_{n} (\Xx \times \Yy)^n \to \Hh$  with the  following property.
For any $\eps, \delta >0$  there exists  $N (\eps, \delta)$ such that  for any
$n \ge N(\eps, \delta)$ and for any $\mu \in \Pp_{\Xx \times \Yy}$ we have
\begin{equation}\label{eq:uniform2}
(\mu^n)^* \{S_n \in  (\Xx\times \Yy)^n :  R ^{K_1}_\mu (\bar A  (S_n))- \inf_{h \in \Hh} R^{K_1}_\mu  (h)  \ge \eps\} \le \delta.
\end{equation}

Let $A: \cup_n (\Xx \times \Yy)^n \to  C_{Lip}(\Xx, \Pp (\Yy)_{\tilde K_2})$ be   the  uniformly consistent  learning  algorithm    in the proof of Theorem \ref{thm:main}. By \eqref{eq:van2},   
for
any $\eps, \delta >0$  there exists  $N_A (\eps, \delta)$ such that  for any
$n \ge N_A(\eps, \delta)$ and for any $\mu \in \Pp_{\Xx \times \Yy}$ we have
\begin{equation}\label{eq:uniform3}
(\mu^n)^* \{S_n \in  (\Xx\times \Yy)^n : \| (\Gamma_{A  (S_n)})_*(\mu_{S_n})_\Xx -\mu\|_{\tilde K_1}  \ge \frac{\eps}{3}\} \le \delta.
\end{equation}
 Let $\bar A (S_n) \in  \Hh$ be  any element  such  that
\begin{align}\label{eq:a*}
   \|(\Gamma_{\bar A (S_n)})_*  (\mu_{S_n})_\Xx  - (\Gamma_{A (S_n)})_*(\mu_{S_n})_\Xx\|_{\tilde K_1}  \nonumber\\
    \le  \inf_{h \in \Hh} \| (\Gamma_h)_* (\mu_{S_n})_\Xx  - (\Gamma_{A (S_n)})_*(\mu_{S_n})_\Xx \|_{\tilde K_1} +\alpha_n
\end{align}
where $\alpha_n$ is a  sufficiently small number  to be  defined later.
Then we have
\begin{eqnarray}
 R ^{K_1}_\mu (\bar A  (S_n))- \inf_{h \in \Hh} R^{K_1}_\mu  (h) \nonumber\\
 \le \|(\Gamma_{\bar A (S_n)})_*  (\mu_{S_n})_\Xx  - (\Gamma_{A (S_n)})_*(\mu_{S_n})_\Xx\|_{\tilde K_1}  + \| (\Gamma_{A  (S_n)})_*(\mu_{S_n})_\Xx -\mu\|_{\tilde K_1}\nonumber \\
 -  \inf_{h \in \Hh} R^{K_1}_\mu  (h)\nonumber\\
 \le   \|(\Gamma_{\bar A (S_n)})_*  (\mu_{S_n})_\Xx  - (\Gamma_{A (S_n)})_*(\mu_{S_n})_\Xx\|_{\tilde K_1}  + \| (\Gamma_{A  (S_n)})_*(\mu_{S_n})_\Xx -\mu\|_{\tilde K_1}\nonumber \\
 -   \inf_{h \in \Hh} \| (\Gamma_h)_* (\mu_{S_n})_\Xx  - (\Gamma_{A (S_n)})_*(\mu_{S_n})_\Xx \|_{\tilde K_1} +  \| (\Gamma_{A  (S_n)})_*(\mu_{S_n})_\Xx -\mu\|_{\tilde K_1} \nonumber\\
 \stackrel{\eqref{eq:a*}}{\le} 2   \| (\Gamma_{A  (S_n)})_*\mu_\Xx -\mu\|_{\tilde K_1} +  \alpha_n. \label{eq:estinf}
\end{eqnarray}
Taking into  account \eqref{eq:uniform3}, we obtain  from  Equation  \ref{eq:estinf}  the following
\begin{equation}\label{eq:uniform4}
(\mu^n)^* \{S_n \in  (\Xx\times \Yy)^n : \|(\Gamma_{\bar A (S_n)})_*  (\mu_{S_n})_\Xx  - \mu\|_{\tilde K_1} \ge \frac{2\eps}{3} + \alpha_n\} \le \delta.
\end{equation}
By letting $\alpha _n \le \frac{\eps}{3}$,   we  obtain  \eqref{eq:uniform2} and, hence,  the  last assertion  of  Corollary \ref{cor:main} immediately from  \eqref{eq:uniform4}.
\end{proof}

\begin{example}\label{ex:main}  Let $\Xx = [0, 1]^n \subset \R^n \times \{0\} \subset \R^{n +m}$ and $\Yy = [0, 1]^m \subset \{ 0\}\times \R^m \subset \R^{n +m}$. Denote by $dx$ the  restriction of the  Lebesgue  measure on $\R^n$  to $\Xx$ and  by $dy$  the restriction of the Lebesgue measure  on $\R^m$  to $\Yy$. Let   $\Pp_{\Xx \times \Yy}$  consist  of  probability measures   $\mu_f: =f dxdy$, where  $f\in C^1 (\Xx\times \Yy)$ and moreover there exist  $c_1, c_0 > 0$ such that  $L (f) \le c_1$  and $ c_ 1\ge f (x, y)  \ge c_0>0$  and  for all $x, y \in \Xx \times \Yy$.  By the Arzel\`a-Ascoli  theorem,  $\Pp_{\Xx \times \Yy}$ is compact in the weak*-topology. We  shall show  that if the PDS kernel $K: \R^{n +m} \times \R^{n+m}\to \R$ in Theorem  \ref{thm:main} is continuously ($C^1$-)  differentiable, then  $\Pp_{\Xx\times \Yy}$ satisfies  the condition of Corollary \ref{cor:main}.  First  we note  that
$$\mu_\Xx=  q\, dx  \text { where }  q(x) = \int _\Xx  f (x, y) dy. $$
Since  $f(x, y) \ge  c_0$,  we have $q (x) \ge c_ 0$ for all $x\in \Xx$. Therefore   ${\rm  sppt}\mu_\Xx =  \Xx$.  Hence  the condition (i) in Theorem \ref{thm:main} on  $\mu_f$ is satisfied.
  
Next we note  that  $q  \in C^1 (\Xx)$, since $f \in C^1 (\Xx\times \Yy)$ and $\Xx \times \Yy$  is compact \cite[Theorem 16.11, p. 213]{Jost05}.  We observe that
+$$(\mu_f)_{\Yy|\Xx} (x) =  g (\cdot |x) \, dy, \text{ where  }  g (y|x) = \frac{ f (x, y)}{ q(x)}.$$
Note that $\mu_f: \Xx \to \Pp (\Yy), x \mapsto (\mu_f)_{\Yy|\Xx} (x), $ is a continuous  map with respect to the strong   topology  on $\Pp(\Yy)$, and hence $\mu_f$ is a measurable  map.  Let  $K_2: \Yy \times \Yy \to \R$ and  $K_1: (\Xx \times \Yy)\times (\Xx \times \Yy) \to \R$  be    continuous  bounded SPD kernels defined as in  Theorem \ref{thm:main}.   We shall  show  that both the maps 
$$\mathfrak M_{K_2}\circ \mu_f: \Xx \to \Hh (K_2), x \mapsto   \mathfrak M_{K_2}\circ \mu_f (x), $$
$$\mathfrak  M_{K_1} \circ \Gamma_{\mu_f}: \Xx \to  \Hh (K), x \mapsto  \mathfrak M_{K_1}  \circ \Gamma_{\mu_f} (x), $$
are $C^1$-differentiable. 

Let us first  prove that  $\mathfrak M_{K_2}\circ \mu_f$ is $C^1$-differentiable. 
Let $x_t$ be a smooth  curve $(-\eps, \eps) \to \Xx, t \mapsto x_t$.
Then
\begin{align}\label{eq:bochner}
\mathfrak M_{K_2}  \big (\mu_f (x_t)\big)    = \int _{ \Yy} (K_2) _{y}\, d\mu_f  (y|x_t) = \int _\Yy (K_2)_y  g(y | x_t)\, dy
\end{align}
where  the  RHS  of \eqref{eq:bochner} is understood  as the Bochner integral.
Since $K_2:\Yy \times \Yy \to \R $ is bounded and continuous,  and  $g(\cdot |\cdot)$ is  bounded  and $C^1$-differentiable  in $x$,  we can perform differentiation under  integral. Thus,   for  $v = \frac{dx_t}{dt}_{| t = 0} \in  T_{x_0} \Xx$  we have
$$ d_{x_0} (\mathfrak M_{K_2}  \circ \mu_f) v  = \int _\Yy (K_2)_y  \frac{d g(y | x_t)}{dt}_{| t =0}\, dy.$$
Hence, $\mathfrak M_{K_2}\circ \mu_f$ is $C^1$-differentiable.  Similarly, we prove   that $\mathfrak  M_{K_1} \circ \Gamma_{\mu_f}$  is  $C^1$-differentiable, using  also the  differentiability of  $K_1$  and of $g (\cdot |\cdot )$ in variable  $x$. Hence, these  maps are  Lipschitz continuous    with Lipschitz  constants  belonging to  an  bounded  interval  $[a', b']\subset \R$.  Consequently, $\mu_f$  satisfies  the  condition of Corollary \ref{cor:main}. 

 Note  that the Gaussian kernels  on $\R^n$ are $C^1$-differentiable but the  Laplacian kernels  are not  $C^1$-differentiable.
\end{example}

\section{Discussion of  results}\label{sec:discuss}

(1)  In this  paper   we have  demonstrated  the usefulness  of the new concept  of  a generative model of supervised  learning, which naturally incorporates    the concept of a correct loss function.  The generative model provides a unified framework for density estimation problems, supervised learning problems, and many problems in Bayesian statistics.

(2)  We have developed the theory of probabilistic morphisms and applied the obtained results to characterize regular conditional probability measures, which are key tools for defining a correct loss function in a generative model of supervised learning. Probabilistic morphisms also offer tools and insights for proving the generalizability of generative models in supervised learning.

(3) The usefulness of the inner measure,  equivalently, the  convergence in outer probability,  for defining and proving the generalization ability of statistical learning models has been demonstrated. 

 (4)  We   present  examples of uniformly  consistent  algorithms for overparameterized   discriminative models of supervised  learning (Corollary \ref{cor:main} (2)). 
 

(5)   To improve    Theorem \ref{thm:main} and Corollary \ref{cor:main},   we can relax  the condition (i)  in Theorem  \ref{thm:main}.  One possible approach could be a combination of the methods developed in this paper and techniques of vector-valued reproducing kernel Hilbert spaces (RKHSs), as developed in    \cite{CDT06}, \cite{CDTU10}.    Corollary  \ref{cor:main} and Examples \ref{ex:main}   suggest  that  the generalizability of  a   generative model   of supervised learning  $(\Xx, \Yy, \Hh, L, \Pp_{\Xx \times \Yy})$ depends  on   the geometry of   $\Hh$ and    $\Pp_{\Xx\times \Yy} $ as well as on the  choice of   positive definite symmetric  (PDS) kernels.  

(6) Another possible research direction could involve combining the methods developed in this paper with those presented in \cite{LV2009}.

\section*{Acknowledgement}
The author  wishes  to thank   Frederic  Protin for  proposing   her   the   loss function in \eqref{eq:lp}, which led  her  to  use  kernel  mean  embedding for  a characterization of  regular  conditional  probability measures   (Theorem \ref{thm:errorRKHS}). She    would like to thank Tobias Fritz  for  helpful comments    on     results  in   Section \ref{sec:unified}. 
A part  of this paper  has been   prepared  while the author  was Visiting   Professor   of the Kyoto  University  from July  till October  2022. She  is grateful to Kaoru Ono   and  the Research Institute of Mathematical Sciences  for their hospitality  and  excellent working  conditions.  The author   also thanks  Wilderich  Tuschmann  and Xia Kelin for inviting   her to the   Online Conference ``Applied Geometry for Data Sciences"   in
Chongquing, China, July  2022, where   she  had  an opportunity to  report   some results   of this   paper. Finally she  expresses  her indebtedness  to  Alexei Tuzhilin  for   stimulating discussions    on measure  theory  for a long time and his suggestion to use  outer measure   about   five years ago.

\end{document}